\documentclass{article}[12pt]

\usepackage{amsmath}
\usepackage{amssymb}
\usepackage{amsfonts}
\usepackage{amsthm}
\usepackage[toc,page]{appendix}
\usepackage{setspace}
\usepackage{bm}
\usepackage{bbm}
\usepackage{marvosym}
\usepackage{mathrsfs}  
\usepackage[english]{babel}
\usepackage{graphicx}
\usepackage{color}
\usepackage{fancyhdr} 
\usepackage[left=1in,right=1in,top=1in,bottom=1in]{geometry}
\usepackage{amsmath,amscd}
\usepackage[all]{xy}
\usepackage{tikz}
\usepackage{tikz-cd}
\usetikzlibrary{shapes}
\usetikzlibrary{arrows}
\usetikzlibrary{matrix}
\usetikzlibrary{calc}
\usetikzlibrary{positioning}
\usetikzlibrary{decorations.pathreplacing}
\usetikzlibrary{decorations.markings}
\usetikzlibrary{intersections}
\usepackage{mathtools}
\usepackage[hypcap=true]{caption}
\usepackage{hyperref}
\usepackage[numbers]{natbib}
\DeclareFontFamily{U}{mathx}{\hyphenchar\font45}
\DeclareFontShape{U}{mathx}{m}{n}{
      <5> <6> <7> <8> <9> <10>
      <10.95> <12> <14.4> <17.28> <20.74> <24.88>
      mathx10
      }{}
\DeclareSymbolFont{mathx}{U}{mathx}{m}{n}
\DeclareFontSubstitution{U}{mathx}{m}{n}
\DeclareMathAccent{\widecheck}{0}{mathx}{"71}
\usepackage{theoremref}
\usepackage{tabularx}
\usepackage{pdflscape}
\usepackage{float}
\usepackage{stmaryrd}
\usepackage[shortcuts]{extdash} 

\newcommand{\tc}[2]{{\textcolor{#1}{#2}}}
\newcommand{\lrp}[1]{\left(#1\right)}
\newcommand{\lrb}[1]{\left[#1\right]}
\newcommand{\lrbb}[1]{\llbracket{ #1 }\rrbracket }
\newcommand{\lrm}[1]{\left|#1\right|}
\newcommand{\lrc}[1]{\left\{#1\right\}}
\newcommand{\lra}[1]{\left\langle{#1}\right\rangle}

\newcommand{\eq}[2]{\begin{equation}\label{#2} \begin{split} #1  \end{split} \end{equation}}
\newcommand{\eqn}[1]{\begin{equation*} \begin{split} #1 \end{split} \end{equation*}}

\newcommand{\pd}[2]{\frac{\partial #1}{\partial #2}}
\newcommand{\pds}[2]{\frac{\partial{\phantom{#2}}}{\partial{#2}}\left( #1 \right)}

\newcommand{\Q}{\mathbb{Q} }
\newcommand{\R}{\mathbb{R} }
\newcommand{\C}{\mathbb{C} }
\newcommand{\A}{\mathbb{A} }
\newcommand{\PP}{\mathbb{P} }

\newcommand{\TT}{\mathbb{T} }

\newcommand{\Z}{\mathbb{Z} }

\newcommand{\gv}{\mathbf{g} }
\newcommand{\cv}{\mathbf{c} }
\newcommand{\vb}[1]{\mathbf{#1}}
\newcommand{\cA}{\mathcal{A} }
\newcommand{\cAp}{\mathcal{A}_{\mathrm{prin}} }

\newcommand{\cAps}[1]{\mathcal{A}_{\mathrm{prin},{#1}} }
\newcommand{\prin}{{\mathrm{prin}} }
\newcommand{\cX}{\mathcal{X} }
\newcommand{\ssO}{\mathcal{O} }

\newcommand{\uf}{\mathrm{uf} }

\newcommand{\tp}{\raisebox{-1.5ex}{\scalebox{2}{$\ulcorner$}} }
\newcommand{\bp}{\scalebox{2}{$\lrcorner$} }
\newcommand{\Xsp}{\widehat{\cX} }
\newcommand{\Xfam}{\mathscr{X} }
\newcommand{\Xfsp}{\widehat{\Xfam} }

\newcommand{\Xt}{\widetilde{X} }
\newcommand{\At}{\widetilde{A} }
\newcommand{\cham}{\mathcal{G}}

\newcommand{\sK}{\mathcal{K} }
\newcommand{\negphantom}[1]{\ifmmode\settowidth{\dimen0}{$#1$}\else\settowidth{\dimen0}{#1}\fi\hspace*{-\dimen0}}

\newcommand{\fg}{\mathfrak{g} }
\newcommand{\fp}{\mathfrak{p} }

\newcommand{\Nuf}{N_{\text{uf}}}

\newcommand{\Iuf}{I_{\text{uf}}}

\newcommand{\wall}{\mathfrak{d}}

\DeclareMathOperator{\Sing}{Sing}

\newcommand{\scat}{\mathfrak{D}}
\newcommand{\seed}{\textbf{s}}

\newcommand{\tf}{\vartheta}

\newcommand{\ubu}{{{\Iuf}\times{\Iuf}}}
\newcommand{\ubI}{{{\Iuf}\times{I}}}

\newcommand{\Bt}{\widetilde{B}}
\newcommand{\epshat}{\widehat{\epsilon}}
\newcommand{\qbz}{q_{\mathrm{BZ}}}
\newcommand{\qfg}{q_{\mathrm{FG}}}
\newcommand{\cc}{\Delta^+}
\newcommand{\gp}{\mathrm{gp}}

\DeclareMathOperator{\Sp}{span}

\DeclareMathOperator{\Hom}{Hom}

\DeclareMathOperator{\sgn}{sgn}

\DeclareMathOperator{\Spec}{Spec}

\DeclareMathOperator{\lcm}{lcm}

\theoremstyle{plain}
\newtheorem{theorem}{Theorem}[section]
\newtheorem{prop}[theorem]{Proposition}
\newtheorem{lemma}[theorem]{Lemma}

\newtheorem{cor}[theorem]{Corollary}

\theoremstyle{definition}
\newtheorem{definition}[theorem]{Definition}
\newtheorem{example}[theorem]{Example}

\theoremstyle{remark}
\newtheorem{remark}[theorem]{Remark}

\newenvironment{problem}[1]{ \flushleft \textcolor{blue}{\normalsize {#1}}}

\newenvironment{subprob}[1]{ \flushleft \textcolor{blue}{\normalsize {#1}}}

\newcommand{\etalchar}[1]{$^{#1}$}

\mathtoolsset{centercolon}

\begin{document}

\renewcommand{\labelenumi}{(\arabic{enumi})}

\newcommand*{\equal}{=}
\newcommand*{\cequal}{{:}=}

\title{Quantization of deformed cluster Poisson varieties}
\author{Man-Wai Mandy Cheung, Juan Bosco Fr\'ias-Medina, and Timothy Magee}
\date{}

\maketitle

\begin{abstract}
    Fock and Goncharov described a quantization of cluster $\cX$-varieties (also known as {\it{cluster Poisson varieties}}) in \cite{FG_cluster_ensembles}.
    Meanwhile, families of deformations of cluster $\cX$-varieties were introduced in \cite{BFMNC}.
    In this paper we show that the two constructions are compatible-- we extend the Fock-Goncharov quantization of $\cX$-varieties to the families of \cite{BFMNC}.
    As a corollary, we obtain that these families and each of their fibers have Poisson structures.
    We relate this construction to the Berenstein-Zelevinsky quantization of $\cA$-varieties (\cite{BZquantum}).
    Finally, inspired by the counter-example to quantum positivity of the quantum greedy basis in \cite{Lee9712}, we compute a counter-example to quantum positivity of the quantum theta basis.
\end{abstract}



\section{Introduction}

Cluster varieties come in two flavors, $\cA$ and $\cX$.
Both types are a union of tori glued via a special class of birational maps, known as {\it{mutation maps}}.
The input of this construction is a collection of combinatorial data $\Gamma$ and the equivalence class $[\seed]$ of a seed $\seed$, which we review in Section~\ref{normalcluster}.
The output is a pair of schemes $\lrp{\cA_{\Gamma,[\seed]},\cX_{\Gamma,[\seed]}}$, together with maps $p:\cA_{\Gamma,[\seed]} \to \cX_{\Gamma,[\seed]}$.
Up to a certain rescaling of character lattices, the tori in the atlas of $\cA_{\Gamma,[\seed]}$ and the tori in the atlas of $\cX_{\Gamma,[\seed]}$ are dual to one another.
Replacing $\Gamma$ and $\seed$ with {\it{Langlands dual data}} $\Gamma^\vee$ and {\it{Langlands dual seed}} $\seed^\vee$ (reviewed in Section~\ref{normalcluster}) gives a new pair $\lrp{\cA_{\Gamma^\vee,[\seed^\vee]},\cX_{\Gamma^\vee,[\seed^\vee]}}$ where this duality is precise-- the tori in the atlas of $\cA_{\Gamma,[\seed]}$ are dual to the tori in the atlas of $\cX_{\Gamma^\vee,[\seed^\vee]}$, and the tori in the atlas of $\cX_{\Gamma,[\seed]}$ are dual to the tori in the atlas of $\cA_{\Gamma^\vee,[\seed^\vee]}$.

While closely related, $\cA$- and $\cX$-varieties differ in some important ways. $\cA$-varieties exhibit the {\it{Laurent phenomenon}} (\cite{FZ_clustersI}) and as a result come with a natural collection of linearly independent global regular functions ({\it{theta functions}}).
A consequence of this fact is that $\cA$-varieties naturally (partially) compactify with projective geometry constructions.
These constructions are reminiscent of the construction of (relative) projective toric varieties via convex polyhedra. (See \cite{CMNcpt}.)
On the other hand, $\cX$-varieties are generally non-separated and so cannot possibly be compactified in this way.
Instead, they come with an atlas naturally described in terms of a fan, and partially compactify in much the same way that a fan defines a toric variety partially compactifying a torus.
In fact, a single object ({\it{a scattering diagram}}) encodes both the theta functions on $\cA_{\Gamma,[\seed]}$ and this partial compactification (known as the {\it{special completion}}) $\Xsp_{\Gamma^\vee,[\seed^\vee]}$ of $\cX_{\Gamma^\vee,[\seed^\vee]}$.
The special completion was initially explored by Fock and Goncharov in \cite{FG_X}.

In both cases, the partial compactifications can be put into families by introducing coefficients to the mutations.
For $\cA$, these coefficients already appeared in the first work on cluster algebras \cite{FZ_clustersI}, and the resulting toric degenerations were explored in \cite{GHKK}.
For $\cX$, the coefficients and corresponding toric degenerations were given in \cite{BFMNC}.

Next, Fock and Goncharov describe in \cite{FG_cluster_ensembles} a quantization $\cX_q$ of the $\cX$-variety. 
In fact, they show that $\cX$-varieties have a Poisson structure-- one of their fundamental features-- by taking a semi-classical limit of $\cX_q$. 
Meanwhile, Berenstein and Zelevinsky address the quantization of $\cA$ varieties in \cite{BZquantum}.
The Berenstein-Zelevinsky $\cA$-quantization admits coefficients, and thus admits families of quantum $\cA$-varieties. 

In this paper we show that the quantization of cluster Poisson varieties due to Fock and Goncharov (\cite{FG_cluster_ensembles}) extends to the family of cluster Poisson varieties with coefficients of \cite{BFMNC}. 
To do so, we introduce coefficients to the quantum mutation formulas given by Fock and Goncharov in \cite[Section~3.3.1]{FG_cluster_ensembles} and obtain
\begin{prop}[\thref{prop:qdmutation}]
The quantum $\cX$-mutation with coefficients 
\[\mu_{k,\vb{t};\cham}^q: \sK\lrp{\A^{\lrm{I}}_{M;\cham',q}\lrp{R}} \to \sK\lrp{\A^{\lrm{I}}_{M;\cham,q}\lrp{R}} \]
is given in cluster coordinates by
\begin{equation*}
    \mu_{k,\vb{t};\cham}^q \lrp{\Xt_{i;\cham'}} = \begin{cases} \Xt_{i;\cham}^{-1} &\text{ if } i=k,\\
& \\
\Xt_{i;\cham} \lrp{\displaystyle\prod_{\ell=1}^{|\epsilon_{ik}|} \lrp{\vb{t}^{[\sgn{\lrp{\epsilon_{ik}}}\vb{c}_{k;\cham}]_+}+\vb{t}^{[-\sgn{\lrp{\epsilon_{ik}}}\vb{c}_{k;\cham}]_+} q_k^{2\ell-1} \Xt_{k;\cham}^{-\sgn{\lrp{\epsilon_{ik}}}}}}^{-\sgn{\lrp{\epsilon_{ik}}}}  &\text{ if } i\not =k \end{cases}.
\end{equation*}
Setting the coefficients $t_i$ to $1$, we recover the usual Fock-Goncharov quantum $\cX$-mutation.
Meanwhile, taking the quantum parameter $q$ to $1$, we recover the classical $\cX$-mutation formula with coefficients of \cite{BFMNC}.
\end{prop}

We further obtain a Poisson structure on the family $\Xfsp$ of \cite{BFMNC} and its fibers through a semi-classical limit: 
\begin{prop}[\thref{PFam}, \thref{PFib}] 
$\Xfsp$ is a Poisson scheme over $\C\lrb{t_i: i\in I}$, and its fibers over closed points are Poisson schemes over $\C$.
\end{prop}

Finally, we show that this quantization of $\cX$-varieties with coefficients is compatible with the Berenstein-Zelevinsky quantization of $\cA$-varieties with coefficients.

\begin{prop}[\thref{prop:quantum-pstar}]
With the identification $\qfg= \qbz^{-\frac{1}{2}\lcm\lrp{d_i: i \in \Iuf}} $, 
a $p^*$-map on the level of character lattices of cluster tori induces a $*$\=/algebra homomorphism from the quantum $\cX$-torus algebra with coefficients to the quantum $\cA$-torus algebra with coefficients which commutes with mutation. 
That is, it induces a map from the Berenstein-Zelevinsky quantum $\cA$-variety with coefficients to the Fock-Goncharov $\cX$-variety with coefficients. 
\end{prop}

To complete the story, we highlight a key difference between the classical and quantum settings.
Specifically, inspired by an example of \cite{Lee9712} for the quantum greedy basis, we illustrate a wall in a rank two quantum scattering diagram whose scattering function is non-positive.
This leads to the failure of ``quantum positivity'' for the quantum theta basis, which is expected from the corresponding result for the quantum greedy basis in \cite{Lee9712}.

\subsection*{Structure of the paper}
Section~\ref{sec:background} is an extended background section, split into several parts.
First, we recall basic notions for (classical) cluster varieties and cluster varieties with principal coefficients in Section~\ref{normalcluster}. 
This is followed by a review of the formulation of scattering diagrams in Section~\ref{sec:scat}. 
We then move on to the quantum setting for both $\cX$ and $\cA$ cluster algebras in Section~\ref{sec:quantum}. 
We conclude the background section with a sketch of the construction of quantum scattering diagrams in Section~\ref{sec:q-scat}, where we also provide an example of a quantum scattering diagram which has wall function with a non-positive coefficient-- a key difference from the classical case. The detailed computation used in this example is provided in Appendix~\ref{app:wallx}.

In Section~\ref{sec:Xfamq}, we construct the quantization of the family $\Xfsp$.
For the convenience of the reader, we review some properties of the quantum affine algebra in Section~\ref{sec:qAff}.
The quantum mutation of the $\Xfsp$ family is introduced in Section~\ref{sec:qmut}. 

Section~\ref{sec:q-pstar} is devoted to relating quantum $\cX$-mutation with coefficients (Section~\ref{sec:Xfamq}) with Berenstein-Zelevinsky's quantum $\cA$-mutation and Fomin-Zelevinsky's $\cA$-mutation with coefficients.
We focus on the case of principal coefficients. 

Using the quantum family $\Xfsp_q$ and taking the classical limit $q\to 1$, we prove in Section~\ref{sec:PoisFromQ} that the family $\Xfsp$ has a Poisson structure, as do its fibers. 
Finally, for the reader who wishes to see a more elementary proof of this result, in Appendix~\ref{app:poisson} we provide a direct proof of the Poisson structure of the $\Xfsp$ family without passing through quantization.


\section{Background}\label{sec:background}
\subsection{Cluster varieties}\label{normalcluster}
Before going into the quantum setting, we recall the definition of cluster algebras and cluster varieties following \cite{GHK_birational} and \cite{GHKK}. 
We start with the \emph{fixed data} $\Gamma$:

\begin{itemize}
\setlength\itemsep{0em}
    \item a finite set $I$ of \emph{directions} with a subset of \emph{unfrozen directions} $\Iuf$; 
    \item a lattice $N$ of rank $ |I|$;
    \item a saturated sublattice $ \Nuf \subseteq N$ of rank $|I_{\text{uf}}|$;
    \item a skew-symmetric bilinear form $\lbrace \cdot , \cdot \rbrace: N\times N \to \Q$;
    \item a sublattice $N^{\circ}\subseteq N$ of finite index satisfying
    $\lbrace N_{\text{uf}}, N^{\circ}\rbrace \subset \Z  $ and $  \lbrace N,N_{\text{uf}}\cap N^{\circ}\rbrace \subset \Z$;
    \item a tuple of positive integers $(d_i:i\in I)$ with greatest common divisor 1;
    \item $M=\Hom(N,\Z)$ and $M^{\circ}=\Hom(N^{\circ},\Z)$.
\end{itemize}

Then a \emph{seed} is a tuple $\seed =(e_{i;\vb{s}}\in N :i\in I)$ such that $\{e_{i;\vb{s}}:i\in I\}$ is a basis of $N$, 
$\{ e_{i;\vb{s}} : i\in \Iuf\} $ is a basis of $\Nuf$, 
and $\{d_i e_{i;\vb{s}} : i\in I\}$ is a basis $N^{\circ}$.
Define the matrix $\hat{\epsilon}_{ij} = \{ e_{i;\vb{s}}, e_{j;\vb{s}} \}$.
The exchange matrix is defined as 
\eqn{\epsilon = (\epsilon_{ij}), \quad \text{where} \ \epsilon_{ij} := \{ e_{i;\vb{s}}, e_{j;\vb{s}} \} d_j.} 
We denote the dual basis to $\vb{s}$ by $\lrc{e_{i;\vb{s}}^*: i \in I}$, and we set $f_{i;\vb{s}} = d_i^{-1}e_{i;\vb{s}}^*$.
Observe that $\lrc{f_{i;\vb{s}}: i \in I}$ is a basis for $M^\circ$.
For $ r \in \R$, define $[r]_+ =\max (0,r)$. Given seed data \textbf{s} and $k \in \Iuf$, the mutation $\vb{s}'$ of $\vb{s}$ is a new basis
\begin{equation} \label{eq:xmutatelattice}
        e_{i;\vb{s}'} :=
\begin{cases}
    e_{i;\vb{s}} + [\epsilon_{ik}]_+ e_{k;\vb{s}} & \text{ for } i \neq k, \\
    -e_{k;\vb{s}} & \text{ for } i = k.
\end{cases}
\end{equation}
Note that $\{e_{i;\vb{s}'} : i \in \Iuf \}$ is still a basis for $\Nuf$ and $\lrc{d_i e_{i;\vb{s}'}:i\in I}$ is still a basis for $N^{\circ}$. 
The new basis $\lrc{ f_{i;\vb{s}'} :i \in I} $ of $ M^{\circ}$ is obtained as the dual of $\lrc{d_i e_{i;\vb{s}'}:i\in I}$.
We can now associate two tori
\begin{equation} \label{def:axtori}
    \cX_{\textbf{s}} = T_M = \Spec \C[N] \text{\qquad and \qquad} \cA_{\textbf{s}} = T_{N^{\circ}} = \Spec \C [M^{\circ}].
\end{equation}

Write $\lrc{X_i:i \in I}$ for the coordinates on $\cX_{\textbf{s}}$ corresponding to the basis vectors $\lrc{e_{i;\vb{s}}:i \in I}$, i.e. $X_i = X^{e_{i;\vb{s}}}$,
and similarly write $\lrc{A_i :i \in I}$ for the coordinates on $\cA_{\vb{s}}$ corresponding to the basis vectors $\lrc{f_{i;\vb{s}}: i\in I}$, i.e. $A_i = A^{f_{i;\vb{s}}}$. 
The coordinates $X_i, A_i$ are called {\it{cluster variables}}. 
These coordinates give identifications 
\[
\cX_{\vb{s}}  \rightarrow \lrp{\C^*}^{\lrm{I}} , \qquad \cA_{\vb{s}} \rightarrow \lrp{\C^*}^{\lrm{I}}. 
\]

We define the birational maps (``{\it{mutation maps}}'')
\[
\mu_k: \cX_{\vb{s}} \dashrightarrow \cX_{\vb{s}'},
 \quad
\mu_k: \cA_{\vb{s}} \dashrightarrow \cA_{\vb{s}'}.
\]
via pull-back of functions
\begin{equation*} 
    \mu_k^* \lrp{X^n} = X^n (1+ X^{e_{k;\vb{s}}} )^{-\lrc{n,e_{k;\vb{s}}}d_k}  \text{ for } n \in N, 
\end{equation*}
\begin{equation*}  
    \mu_k^* \lrp{A^m} = A^m (1+ A^{v_{k;\vb{s}}} )^{-\langle d_k e_{k;\vb{s}}, m \rangle}  \text{ for } m \in M^{\circ},
\end{equation*}
where $v_{k;\vb{s}} = \lrc{e_{k;\vb{s}}, \ \cdot\ } \in M^\circ$. 
In terms of cluster variables, these mutation maps may be rewritten as
\eq{ 
    \mu_k^*\lrp{X_{i;\vb{s}'}}= 
    \begin{cases}
        X_{k;\vb{s}}^{-1} & \text{for } i=k\\
        X_{i;\vb{s}} \lrp{1 + X_{k;\vb{s}}^{-\sgn(\epsilon_{ik})}}^{-\epsilon_{ik}} & \text{for } i \neq k
    \end{cases} 
}{eq:muX}
and
\eq{ 
    \mu_k^*\lrp{A_{i;\vb{s}'}}= 
    \begin{cases}
        A_{k;\vb{s}}^{-1}\lrp{ \prod_{j: \epsilon_{kj}>0} A_{j;\vb{s}}^{\epsilon_{kj}} + \prod_{j: \epsilon_{kj}<0}A_{j;\vb{s}}^{-\epsilon_{kj}} } & \text{for } i=k\\
        A_{i;\vb{s}} & \text{for } i \neq k
    \end{cases}. 
}{eq:muA}
The $\cA$ and $\cX$ {\it{cluster varieties}} associated to $\Gamma$ and the mutation equivalence class $[\seed]$ of $\seed$ are the schemes
\eqn{\cA_{\Gamma,[\seed]} := \bigcup_{\seed \in [\seed]} \cA_\seed / \sim, }
where tori are glued via the birational maps \eqref{eq:muA},
and
\eqn{\cX_{\Gamma,[\seed]} := \bigcup_{\seed \in [\seed]} \cX_\seed / \sim, }
where tori are glued via the birational maps \eqref{eq:muX}.
See \cite[Proposition 2.4]{GHK_birational} and the discussion following it for further details.
\begin{remark}
When there is no risk of confusion, we will generally drop the subscripts $\Gamma$ and $[\seed]$ simply write $\cA$ and $\cX$.
\end{remark}
The $\cA$- and $\cX$-{\it{cluster algebras}} are the spaces of regular functions $\Gamma(\cA, \mathcal{O}_{\cA})$ and $\Gamma(\cX, \mathcal{O}_{\cX})$ on these cluster varieties.

\begin{remark}
There is a notion of duality for the fixed data $\Gamma$ and seed $\seed$.
The {\it{Langlands dual fixed data}} $\Gamma^\vee$ is defined by setting
\eqn{I^\vee := I, \qquad \Iuf^\vee := \Iuf, \qquad N^\vee = N^\circ, \qquad \lrc{ \ \cdot\ ,\ \cdot\ }^\vee := \frac{1}{\lcm\lrp{d_i:i\in I}} \lrc{ \ \cdot\ ,\ \cdot\ },}  
\eqn{\lrp{N^\vee}^\circ := \lcm\lrp{d_i:i\in I} \cdot N, \; \text{ and} \qquad d_j^\vee:= \frac{\lcm\lrp{d_i:i\in I}}{d_j}.}
The {\it{Langlands dual seed}} is $\seed^\vee:= \lrp{d_i e_{i;\seed}: i\in I}$. 
Naturally, this Langlands dual data determines a new pair of cluster varieties: $\cX_{\Gamma^\vee,[\seed^\vee]}$ and $\cA_{\Gamma^\vee,[\seed^\vee]}$ are said to be the {\it Fock-Goncharov duals} of $\cA_{\Gamma,[\seed]}$ and $\cX_{\Gamma,[\seed]}$.
See \cite[Section~1.2.10]{FG_cluster_ensembles}.
\end{remark}

For both $\cA$ and $\cX$, the mutation relations may be generalized to allow coefficients in a semi-field.
This was treated for $\cA$ in \cite{FZ_clustersI} and for $\cX$ in \cite{BFMNC}.
Importantly, the $\cA$ and $\cX$ mutations with coefficients remain dual to one another in a precise sense: see \cite[Section~3.3]{BFMNC}.
An instance of this construction of particular interest-- and the setting of the article-- is the case of {\it{principal coefficients}}.
We review cluster varieties with principal coefficients in the following subsection.

\subsubsection{Cluster varieties with principal coefficients}\label{sec:prin}

We first treat principal coefficients for $\cA$-varieties following \cite{GHK_birational}.
Let $\Gamma$ be the fixed data and $\seed$ the seed data of a (coefficient-free\footnote{We distinguish frozen variables and coefficients, so $I$ may have frozen directions.  The distinction is not so important for $\cA$, but it is significant for $\cX$ and affects the notion of cluster duality.}) pair of cluster varieties $\lrp{\cA,\cX}$.
Now consider the lattice and skew-symmetric bilinear form 
\[
\widetilde{N}:=N\oplus M^{\circ}, \quad \{\lrp{n_1,m_1},\lrp{n_2,m_2} \}_{\prin}
:=\{n_1,n_2\}+\langle n_1,m_2\rangle - \langle n_2,m_1 \rangle.
\]
Take $\widetilde{I} = I \sqcup I $ with $\widetilde{I}_{\uf} $ the inclusion of $\Iuf$ into the first copy of $I$, $\widetilde{N}_{\uf} = \Nuf \subset \widetilde{N}$, $\widetilde{N}^{\circ}=N^{\circ} \oplus M$, 
and $\widetilde{d}_i = d_i$ for $i$ in either copy of $I$.
This defines new fixed data $\widetilde{\Gamma}$.
Meanwhile, from the seed $\vb{s}$ we obtain a seed $\widetilde{\vb{s}}$ as $\lrp{(e_{i;\vb{s}},0): i \in I}\cup \lrp{(0,f_{i;\vb{s}}): i \in I}$. 
This data defines an $\cA$-variety denoted $\cAp$ closely related to the original $\cA$-variety.
In fact, denoting ${\widetilde{A}}^{(0,f_{i;\vb{s}})}$ by $t_i$, 
$\cAp$ may be viewed as a family of deformations of $\cA$ over 
$\Spec\lrp{\C\lrb{t_i^{\pm 1}: i\in I}}$.
Moreover a choice of seed $\vb{s}$ determines an extension of $\cAp$ to a family $\cAps{\vb{s}}$ over $\Spec\lrp{\C\lrb{t_i: i\in I}}$.
Explicitly, mutation of $\cA$-variables takes the form
\eq{
\mu_k^* \lrp{\widetilde{A}_{i;\vb{s'}}} = 
    \begin{cases}
        \widetilde{A}_{k;\vb{s}}^{-1}\lrp{\vb{t}^{[\cv_{k;\vb{s}}]_+}\prod_{j: \epsilon_{kj}>0}\widetilde{A}_{j;\vb{s}}^{\epsilon_{kj}} + \vb{t}^{[-\cv_{k;\vb{s}}]_+}\prod_{j: \epsilon_{kj}<0}\widetilde{A}_{j;\vb{s}}^{-\epsilon_{kj}}} & \text{if } i = k\\
        \widetilde{A}_{i;\vb{s}} & \text{if } i\neq k
    \end{cases}.
}{eq:Aprinmu}
Here $\cv_{k;\vb{s}}$ is a {\it{$\cv$-vector}}.
If $I_1$ and $I_2$ are the first and second copies of $I$ in $\widetilde{I}$, then $\cv_{k;\vb{s}}$ is the $k^{\text{th}}$ row of the submatrix $\widetilde{\epsilon}_{I_1 \times I_2}$ of 
the exchange matrix $\widetilde{\epsilon}$.
A major result in cluster theory is that $\cv$-vectors are {\it{sign coherent}}, meaning all entries are either non-negative or non-positive. See \cite[Corollary~5.5]{GHKK}.
In writing $[\cv_{k;\vb{s}}]_+$ in \eqref{eq:Aprinmu}, we use this sign coherence. Explicitly, given a sign coherent element $\vb{x}$ of $\R^n$, we define 
\eqn{ [\vb{x}]_+ := \begin{cases}
    \vb{x} & \text{if the entries of $\vb{x}$ are positive}\\
    0 & \text{otherwise.}
\end{cases} 
}

Next, the $\cX$-mutation with principal coefficients defining the corresponding family of deformations $\Xfam$ of $\cX$ from \cite{BFMNC} is given by
\eq{
\mu_k^*\lrp{\Xt_{i;\vb{s}'}} = 
    \begin{cases} 
        \Xt_{i;\vb{s}}^{-1} &\text{for } i=k\\
        \Xt_{i;\vb{s}}\lrp{\vb{t}^{[\sgn(\epsilon_{ik})\vb{c}_{k;\vb{s}}]_+} + \vb{t}^{[-\sgn(\epsilon_{ik})\vb{c}_{k;\vb{s}}]_+}\Xt_{k;\vb{s}}^{-\sgn(\epsilon_{ik})}}^{-\epsilon_{ik}} &\text{for } i\not =k
    \end{cases}.
}{eq:Xfammu}
Note that $\Xfam$ is {\emph{not}} the scheme $\cX_{\mathrm{prin}}$ of \cite{GHKK}.
The difference stems from treatment of coefficients and frozen variables.
In the skew-symmetric case,\footnote{When the exchange matrix is only skew-symmetrizable, we also need to take the Langlands dual data to obtain the dual cluster variety with principal coefficients.}
$\Xfam$ is dual to $\cAp$ as cluster varieties over $\C\lrb{t_i^{\pm 1}:i \in I}$,
while $\cX_{\mathrm{prin}}$ is dual to $\cAp$ as cluster varieties over $\C$.  
See Section~\ref{sec:q-pstar} for a more in depth discussion of this distinction.

\begin{remark}
As mentioned previously, the mutation relations with principal coefficients \eqref{eq:Aprinmu} and \eqref{eq:Xfammu} are both special cases of more general mutation relations with coefficients in an arbitrary semifield $\PP$.
The $\cA$ mutation with coefficients in $\PP$ was described by Fomin and Zelevinsky in \cite{FZ_clustersI}, 
while the $\cX$-mutations with coefficients in $\PP$ was given in \cite{BFMNC}.
For reference, these more general relations take the form:
\eq{
\mu_k^* \lrp{\widetilde{A}_{i;\vb{s'}}} = 
    \begin{cases}
        \widetilde{A}_{k;\vb{s}}^{-1}\lrp{p_k^{+}\prod_{j: \epsilon_{kj}>0}\widetilde{A}_{j;\vb{s}}^{\epsilon_{kj}} + p_k^{-}\prod_{j: \epsilon_{kj}<0}\widetilde{A}_{j;\vb{s}}^{-\epsilon_{kj}}} & \text{if } i = k\\
        \widetilde{A}_{i;\vb{s}} & \text{if } i\neq k
    \end{cases}
}{eq:Acoefmu}
and
\eq{
\mu_k^*\lrp{\Xt_{i;\vb{s}'}} = 
    \begin{cases} 
        \Xt_{i;\vb{s}}^{-1} &\text{for } i=k\\
        \Xt_{i;\vb{s}}\lrp{p_k^{ \lrbb{\epsilon_{ik} } } + p_k^{\lrbb{-\epsilon_{ik} } }\Xt_{k;\vb{s}}^{-\sgn(\epsilon_{ik})}}^{-\epsilon_{ik}} &\text{for } i\not =k
    \end{cases},
}{eq:Xcoefmu}
where
\eqn{
p^+:=\dfrac{p}{p\oplus 1}, 
\;\;\;\; 
p^-:=\dfrac{1}{p\oplus 1}, 
\; \mathrm{and} \;\;\;\;
p^{\lrbb{x}}:= \begin{cases} p^{-} &\text{ if } x<0,\\
1 &\text{ if } x=0,\\
p^+ &\text{ if } x>0.\end{cases}
}

\end{remark}

\subsection{Scattering diagrams}\label{sec:scat}

The geometry and combinatorics of cluster varieties and their associated cluster algebras can be encoded using scattering diagrams, as formulated in \cite{GHKK}.
In this section, we review the definition of scattering diagrams over skew-symmetric Lie algebras as in \cite{KS14}, \cite{GHKK}, \cite{dtpaper}, \cite{davison2019strong}.
We then use this machinery to discuss quantum scattering diagrams in Section~\ref{sec:q-scat}.

Consider a finite rank lattice $L$ with a skew-symmetric form $\{ \cdot, \cdot \}$.
Let $L^{\vee}$ be the dual lattice. 
Choose a strictly convex rational polyhedral cone $\mathcal{C}$ in $L_{\R} = L \otimes \R$. 
Let $L^{\oplus} = \mathcal{C} \cap L$ and $L^+ = L^{\oplus} \setminus \{ 0\}$. 
Now let $\fg$ be an $L^+$-graded Lie algebra: 
\eq{
\fg = \bigoplus_{n \in L^+} \fg_n, \qquad \lrb{\fg_{n_1},\fg_{n_2}} \subset \fg_{n_1+n_2}.
}{eq:Lie} 
Assume $\fg$ is skew-symmetric with respect to $\{ \cdot , \cdot \} $, meaning $[ \fg_{n_1} , \fg_{n_2} ] = 0$ whenever $\{ n_1, n_2 \} = 0$.
Choose a linear function $d: L \rightarrow \Z$ such that $d(n) >0$ for all $n \in L^+$. 
Then define $\displaystyle{\fg^{> k} := \bigoplus_{d(n)>k} \fg_n}$, which is a Lie subalgebra of $\fg$,
and let $\fg^{\leq k} = \fg / \fg^{> k}$. 
For each $n \in L^+$, define $\displaystyle{\fg_{n}^{\parallel}:=  \bigoplus_{k \geq 1} \fg_{kn}}$.
We have the corresponding nilpotent, pronilpotent, and abelian algebraic groups
$G_k:= \exp \fg^{\leq k}$, 
$G := \exp \fg = \varprojlim G_k$, 
and $G_{n}^{\parallel}:= \exp(\fg_n^{\parallel}) \subset G$.

\begin{definition}
A {\it{wall}} in $L^{\vee}_{\R}$ over $\fg$ is a pair $(\wall,g_{\wall})$ such that 
\begin{itemize}
    \item $g_{\wall}\in G_{n_{\wall}}^{\parallel}$ for some primitive $n_{\wall}\in L^+$.  
    \item $\wall \subset L^{\vee}_{\R}$ is a closed, convex (but not necessarily strictly convex), rational-polyhedral, codimension-one cone in $L^{\vee}_{\R}$ contained in $n_{\wall}^{\perp}$.  We call $\wall$ the {\it{support}} of the wall.
\end{itemize}
The wall $(\wall, g_{\wall})$ is said to be {\it{incoming}} if $\lrc{n_\wall, \ \cdot \ } \in \wall$.
Otherwise, it is said to be {\it{outgoing}}.

A {\it{scattering diagram}} $\scat$ over $\fg$ is a set of walls in $L^{\vee}_{\R}$ over $\fg$ such that for each $k >0$, there are only finitely many $(\wall,g_{\wall})\in \scat$ with $g_{\wall}$ not $1$ in $G_k$.   
The {\it{singular locus}} of the scattering diagram is defined as
	\eqn{\Sing (\scat) = {\bigcup_{(\wall,g) \in \scat} \partial \wall} \quad \cup {\bigcup_{\substack{(\wall_1,g_1), (\wall_2,g_2) \in \scat \\  \dim \wall_1 \cap \wall_2 = n-2}} \wall_1 \cap \wall_2}.} 
\end{definition}

Now consider a smooth immersion $\gamma: [0,1] \rightarrow L^{\vee}_{\R} \setminus \Sing(\scat)$, with endpoints not contained in the support of $\scat$, that only crosses walls transversely.
For each $k>0$, $\gamma$ will cross only a finite number $s_k$ of walls $(\wall,g_{\wall})$ with $g_{\wall}$ not $1$ in $G_k$, and we label them in order of crossing by
$(\wall_i, g_{\wall_i})$, $i = 1 , \dots, s_k$. 
Let $t_i$ be the time at which $\gamma$ crosses the wall $\wall_i$.
For this $k$, we define
\[
\fp^k_{\gamma} = g_{\wall_{s_k}}^{\sgn (\langle n_{\wall_{s_k}}, -\gamma'(t_{s_k}) \rangle)}  \cdots 
g_{\wall_{1}}^{\sgn (\langle n_{\wall_1}, -\gamma'(t_1) \rangle)}, 
\]
and taking the limit $k \to \infty$ we define the {\it{path-ordered product}} 
$	\fp_{\gamma} =\displaystyle{ \lim_{k \rightarrow \infty} \fp ^k_{\gamma} \in {G}}$. 
Note that if $(\wall_1,g_{\wall_1})$ and $(\wall_2,g_{\wall_2})$ are two walls of $\scat$ with $\text{codim}_{L^{\vee}_{\R}}(\wall_1\cap \wall_2)=1$, then by skew symmetry of $\fg$ with respect to $\{ \cdot , \cdot \} $, we have $[g_{\wall_1},g_{\wall_2}]=0$.
As a result, the previous definition does not depend on how we order walls which are crossed simultaneously.

Two scattering diagram $\scat$ and $\scat'$ are said to be \emph{equivalent} if $\fp_{\gamma, \scat} = \fp_{\gamma, \scat'}$ for all paths $\gamma$ for which both are defined.
A scattering diagram is \emph{consistent} if $\fp_{\gamma, \scat}$ only depends on the endpoints of $\gamma$ for any path $\gamma$ for which $\fp_{\gamma, \scat}$ is defined.

Versions of the following fundamental result have been proved in various contexts and degrees generality by various authors.

\begin{prop}[\cite{KS06}, \cite{GS11},  \cite{KS14}, \cite{GHKK}, \cite{Bri17}] \thlabel{thm:ks}
Let $\fg$ be a skew-symmetric $L^+$-graded Lie algebra, and let $\scat_{\mathrm{in}}$ be a finite scattering diagram 
whose walls are of the form $(n_i^{\perp},g_i)$ for  primitive $n_i\in L^+$. 
Then there is a unique-up-to-equivalence scattering diagram $\scat$ such that $\scat$ is consistent, $\scat \supset \scat_{\mathrm{in}}$, and 
$\scat \setminus \scat_{\mathrm{in}}$ contains only outgoing walls.
\end{prop}

Briefly, \cite{KS06} treats the two dimensional case, \cite{GS11} deals with scattering diagrams of arbitrary dimension in a logarithmic geometry context, \cite{KS14} elaborates the construction of general scattering diagrams over skew-symmetric Lie algebras and is the source for the formulation of \thref{thm:ks} we have provided, \cite{GHKK} develops scattering diagrams for cluster varieties, and \cite{Bri17} considers scattering diagrams over skew-symmetric Lie algebras as well and takes Hall algebras to obtain stability scattering diagrams.

\subsubsection{Cluster scattering diagrams}\label{sec:cluster-scat}

Restricting to the cluster setting and following \cite[Section~1]{GHKK},
consider the fixed data $\Gamma$ reviewed in Section~\ref{normalcluster}.
We need to use $\Gamma$ and a choice of seed $\seed= \lrp{e_i: i \in I}$ to construct a graded Lie algebra as in \eqref{eq:Lie}.

Naturally, the lattice $N$, which comes equipped with the skew-form $\lrc{\cdot , \cdot }$,  replaces the lattice $L$ from the general discussion. 
Now define
\eq{N^+_\seed := \lrc{ \sum_{i\in \Iuf} a_i e_i :  a_i\geq 0, \sum_{i\in \Iuf} a_i >0  } .}{eq:N+} 
This plays the role of $L^+$ from the general discussion. 

Define the map 
\eqn{
p_1^*: \Nuf &\rightarrow M^\circ\\  
 n &\mapsto \{ n,\ \cdot\ \}. 
}
In order to construct the cluster scattering diagram, \cite{GHKK} require $p_1^*$ to be injective.\footnote{
This is known as the {\it{injectivity assumption}}.
While it may seem limiting at first glance, this always holds after introducing principal coefficients, which we reviewed in Section~\ref{sec:prin}.
\cite{GHKK} use $\cAp$ to draw conclusions about $\cA$ (which is a fiber of $\cAp$) and $\cX$ (which is a quotient of $\cAp$ by a torus action).}
Using the injectivity assumption, we can choose a strictly convex top-dimensional
cone $\sigma \subseteq M^\circ_{\R}$
with associated monoid $P:= \sigma \cap M^{\circ}$ such that $p_1^*(e_i) \in P \setminus\{0\}$ for all $i \in \Iuf$.
Denote by $\widehat{\C[P]}$ the completion of $\C[P]$ with respect to the monomial ideal generated by $P\setminus\lrc{0}$.
Now consider the Lie algebra $\C [P] \otimes_{\Z} N^{\circ}$, with Lie bracket given by
\eqn{\lrb{A^{m_1}\otimes n_1, A^{m_2}\otimes n_2} = A^{m_1+m_2} \otimes \lrp{\lra{n_1,m_2}n_2 - \lra{n_2,m_1}n_1 }.}
Let it act on $\C [P]$ as follows:
$(f \otimes n ) (A^m) = { \langle n , m \rangle } f A^m$. 
Define $\fg_n $ to be the one-dimensional subspace of $\C [P] \otimes_{\Z} N^{\circ}$ spanned by $A^{p_1^*(n)}\otimes n$. 
Then the Lie subalgebra $\fg\subset \C [P] \otimes_{\Z} N^{\circ}$ is defined to be 
\eqn{\fg := \bigoplus_{n \in N^+_\seed} \fg_n.} 
Observe that $\fg$ is skew-symmetric with respect to $\lrc{\cdot,\cdot}$.
We can again choose a linear function $d:N \to \Z$ such that $d(n) >0$ for all $n\in N^+_\seed$.
We then define $\fg^{>k}$, $\fg^{\leq k}$, $\fg_{n}^{\parallel}$, and the corresponding Lie groups $G_k$, $G$, and $G_n^\parallel$ precisely as before.

All that remains is to give an initial scattering diagram.
For $n\in N^+_\seed$ and $f\in \widehat{\C[P]}$ of the form $${f= 1 + \sum_{k=1}^{\infty} c_k A^{k p_1^*(n)}},$$
define $\fp_f \in G_n^\parallel$ by
\eqn{\fp_f\lrp{A^m} = f^{\lra{n',m}} A^m,}
where $n'$ is the generator of $\R_{\geq 0} \cdot n \cap N^\circ$.
Assuming the injectivity assumption is satisfied (so the scattering diagram may be defined), 
the initial walls of the cluster scattering diagram are 
\eq{
\lrc{\lrp{ e_i^{\perp} , \fp_{f_{\wall_i}} } :\  i \in \Iuf,\, f_{\wall_i} = 1+A^{ p_1^*\left( e_i \right)}}.
}{eq:initial_scat}
For instance, the standard example is the scattering diagram associated to the $A_2$ quiver $v_1 \rightarrow v_2$.
This is shorthand for the following the skew-symmetric type\footnote{Meaning, the positive integers $d_i$ are all 1.} data without frozen directions: 
\begin{itemize}
    \item $N\cong \Z^2$,
    \item $\seed=\lrp{e_{1;\seed},e_{2;\seed}}$,
    \item $\lrc{e_{1;\seed},e_{2;\seed}} =1$.
\end{itemize}   
The initial scattering diagram described in \eqref{eq:initial_scat} is shown on the left side of Figure~\ref{fig:A2Scat},
with the associated consistent scattering diagram of \thref{thm:ks} on the right.


\noindent
\begingroup
\captionsetup{type=figure}
\begin{center}
\begin{tikzpicture}[scale=.9]

    \def\l{3}
    \def\d{1}
    \def\op{.5}

    \path (-\l,0) coordinate (3) --++ (\l,0) coordinate (0) --++ (\l,0) coordinate (1);
    \path (0,\l) coordinate (2) --++ (0,-2*\l) coordinate (4);
    \path (0,0) --++ (-3*\d,2*\d) coordinate (5);

    \draw[thick,->] (3) -- (1);
    \draw[thick,->] (2) -- (4);

    \node at (-.65*\l,-.5) {$1+A_{1}^{-1}$};
    \node at (.3*\l,.85*\l) {$1+A_2$};

\begin{scope}[xshift=8cm]

    \def\l{3}
    \def\d{1}
    \def\op{.5}

    \path (-\l,0) coordinate (3) --++ (\l,0) coordinate (0) --++ (\l,0) coordinate (1);
    \path (0,\l) coordinate (2) --++ (0,-2*\l) coordinate (4);
    \path (0,0) --++ (\l,-\l) coordinate (5) node [pos=.5, sloped, above]  {$1+A_1^{-1} A_2$};

    \draw[thick,->] (3) -- (1);
    \draw[thick,->] (2) -- (4);
    \draw[thick,->] (0) -- (5);

    \node at (-.65*\l,-.5) {$1+A_{1}^{-1}$};
    \node at (.3*\l,.85*\l) {$1+A_2$};
    
\end{scope}
\end{tikzpicture}.
\captionof{figure}{\label{fig:A2Scat} The initial $A_2$ scattering diagram is shown on the left, and the associated consistent scattering diagram on the right.} 
\end{center}
\endgroup


As the injectivity assumption is always satisfied after introducing principal coefficients, we can always work on the level of $\cAp$ to write:
\[
\scat_{\mathrm{in}, \seed}^{\cAp} := \lrc{\left( (e_i, 0)^{\perp} , \fp_{f_{\wall_i}} \right) : \  i \in \Iuf,\,  f_{\wall_i}= 1+\At^{ p_1^*\left( (e_i,0) \right) }}. 
\]

By \thref{thm:ks}, there exists a consistent scattering diagram $\scat^{\cAp}$ containing $\scat_{\mathrm{in}, \seed}^{\cAp} $. 
This is the cluster scattering diagram for $\cAp$. 
\cite[Construction~7.11]{GHKK} gives precisely how this is related to $\cA$ and $\cX$, allowing the recovery of a scattering diagram for $\cX$ as a slice of the $\cAp$ scattering diagram and a similar combinatorial structure\footnote{In general the ``scattering functions'' of this structure may not have the same form as scattering functions of a scattering diagram.} for $\cA$ whose support is a quotient of the support of the $\cAp$ scattering diagram.  
An overview appears in \cite[Section 2.2.1 and 2.2.2]{CMNcpt}.

The following result about cluster scattering diagrams plays an important role in this paper.
\begin{theorem}[{\cite[Lemma~2.10]{GHKK}}]
Let $\Gamma$ be fixed data for which the injectivity assumption is satisfied, fix an initial seed $\seed$,
let $\scat$ be the associated cluster scattering diagram,
and let $M^\circ_\R$ be the ambient space of the support of $\scat$.
Then $\scat$ induces a chamber decomposition of a full dimensional subset of $M^\circ_\R$,
yielding a (usually infinite) simplicial generalized\footnote{The cones are convex, but not strictly convex if there are frozen directions.} fan $\cc$  whose maximal cones are in 1-1 correspondence with clusters. 
\end{theorem}

\subsubsection*{Fan structure and cluster varieties}

The generalized fan $\cc$ is known as the {\it{cluster complex}}. 
If there are no frozen directions, the rays of $\cc$ correspond precisely to cluster variables.
Each ray is generated by the {\it{$\gv$-vector}} of a cluster variable, a notion which we will now review.
Let $p_2^*: N \to M^\circ/ \Nuf^\perp$ be the lattice map defined by 
\eqn{p_2^*(n) = \left.\lrc{n, \ \cdot\ }\right|_{\Nuf\cap N^\circ},}
and denote the projection $ M^\circ \to M^\circ/ \Nuf^\perp$ by $\pi$.
Now let $p^*:N \to M^\circ $ be any lattice map satisfying
\begin{itemize}
    \item $\left.p^*\right|_{\Nuf} = p_1^*$;
    \item $\pi \circ p^* = p_2^*$.
\end{itemize}
The inclusion 
\eqn{
N^\circ &\hookrightarrow N^\circ \oplus M\\
n &\mapsto \lrp{n, p^*(n)}
}
induces a $T_{N^\circ}$ action on $\cAp$.
After fixing an initial seed $\vb{s_0}$, every cluster variable $A_{i;\vb{s}}$ on $\cA$ has a canonical extension $\widetilde{A}_{i;\vb{s}}$ to $\cAp$,
and $\widetilde{A}_{i;\vb{s}}$ is an eigenfunction of this $T_{N^\circ}$ action.
The {\it{$\gv$-vector}} of $A_{i;\vb{s}}$ at $\vb{s_0}$ is the weight of $\widetilde{A}_{i;\vb{s}}$ under this $T_{N^\circ}$ action. See \cite[Section~5]{GHKK} for further details.

Even if the injectivity assumption is not satisfied for $\Gamma$, it is for $\widetilde{\Gamma}$ and projecting from $\widetilde{M}^\circ_\R$ to $M^\circ_\R$ once again yields a simplicial generalized fan $\cc$  (\cite[Theorem~2.13]{GHKK}).
This is still called the {\it{cluster complex}}.
We obtain an honest simplicial fan (as opposed to generalized fan) by defining the maximal cone associated to a cluster to be the $\R_{\geq 0}$-span of the $\gv$-vectors of the variables (both frozen and unfrozen) of that cluster.  
In this way each cone is contained in a chamber of $\cc$, and the integral points of maximal cones correspond the monomials in the variables of the corresponding cluster.
We call the resulting fan $\Sigma$ the {\it{$\gv$-vector fan}}.
We can avoid certain redundancies in the atlas of $\cA$ by indexing torus charts by maximal cones of $\Sigma$ rather than seeds, as multiple seeds may yield the same unlabeled cluster coordinates.
In this language, we have
\eqn{
\cA_{\Gamma,\Sigma}= \bigcup_{\cham \in \Sigma(\max)} \Spec\lrp{\C\lrb{\lrp{\cham \cap M^\circ}^\gp } } / \sim,
}
where $\lrp{\cham \cap M^\circ}^\gp$ is the group completion of the monoid $\cham \cap M^\circ$.

While $\Sigma$ is clearly closely related to $\cA_{\Gamma,\Sigma}$,
it also has a strong connection to $\cX_{\Gamma^\vee,\Sigma}$.
Observe that $\Sigma$ is a fan in $M^\circ_\R$, and $M^\circ$ is the cocharacter lattice of cluster tori in $\cX_{\Gamma^\vee,\Sigma}$.
In fact, $\cX$-cluster variables on $\cX_{\Gamma^\vee,\Sigma}$ can be associated to $\cv$-vectors (Section~\ref{sec:prin}, below \eqref{eq:Aprinmu}) in much the same way that $\cA$-variables on $\cA_{\Gamma,\Sigma}$ are associated to $\gv$-vectors. (See \cite[Definition~5.10]{BFMNC}.)
These $\cv$-vectors $\lrc{\cv_{i;\seed}:i \in I}$ associated to a cluster $\lrc{X_{i;\seed^\vee}:i \in I^\vee = I}$ generate the dual cone to $\cham:= \Sp_{\R_{\geq 0}} \lrc{\gv_{i;\seed}:i\in I }$. (See \cite[Theorem~1.2]{NZ}.) 
We can once again index by maximal cones of $\Sigma$ rather than seeds to avoid redundancies in the atlas of $\cX_{\Gamma^\vee,\Sigma}$.
In these terms, $\cX_{\Gamma^\vee,\Sigma}$ is the scheme
\eqn{
\cX_{\Gamma^\vee,\Sigma} = \bigcup_{\cham \in \Sigma(\max)} \Spec\lrp{\C\lrb{\lrp{\cham^\vee \cap N^\circ}^\gp } } / \sim,
}
where $\cham^\vee$ is the dual cone of $\cham$ and $\lrp{\cham^\vee \cap N^\circ}^\gp$ is the group completion of the monoid $\cham^\vee \cap N^\circ$.
Observe that this description is highly reminiscent of the construction of toric varieties in terms of a fan.
There are two key differences: the gluing maps and the group completion.
The gluing maps are what put us in the cluster world to begin with, and if we were to replace the mutation maps with toric gluing maps the end result would simply be a torus.
That is to say, given a fan $\Sigma\in M_\R$, if we use the toric gluing maps $\mu_k^*(X^n) = X^n$, we trivially have the equality
\eqn{
T_M:=\Spec\lrp{\C[N]} = \bigcup_{\cham \in \Sigma(\max)} \Spec\lrp{\C\lrb{\lrp{\cham^\vee \cap N}^\gp } } /\sim.
}
This suggests a very natural partial compactification of $\cX_{\Gamma^\vee,\Sigma}$, analogous to a toric variety partially compactifying a torus:
\eqn{
\Xsp_{\Gamma^\vee,\Sigma} = \bigcup_{\cham \in \Sigma(\max)} \Spec\lrp{\C\lrb{\cham^\vee \cap N^\circ} } / \sim.
}
The scheme $\Xsp_{\Gamma^\vee,\Sigma}$ is known as the {\it{special completion}}  of $\cX_{\Gamma^\vee,\Sigma}$, and it was introduced by Fock and Goncharov in \cite{FG_X}.
\begin{remark}
Note that Langlands duality is an involution (up to canonical isomorphism), so if we were to start with $\Gamma^\vee$ we would obtain the special completion $\Xsp_{\Gamma,\Sigma}$.
We will generally make use of this fact to avoid continually writing ``Langlands dual'' and ${}^\vee$.
\end{remark}

\begin{remark}\thlabel{rem:cham}
From Section~\ref{sec:Xfamq} on we will make use of the discussion in this section to index by maximal cones $\cham$ of $\Sigma$ rather than by seeds $\seed$. For instance, we will denote an $\cX$-cluster on $\Xfam_q$ by $\lrc{\Xt_{i;\cham}:i \in I}$. 
Note that for this to be possible, if $\seed_0$ is the initial seed and $\mu^q_{(\seed_0,\seed)}$ and $\mu^q_{(\seed_0,\seed')}$ are a pair of mutation sequences for $\seed,\seed'\in[\seed_0]$, we must have that $\mu^q_{(\seed_0,\seed)}\lrp{\Xt_{i;\seed}} = \mu^q_{(\seed_0,\seed')}\lrp{\Xt_{j;\seed'}}$ whenever $\seed$ and $\seed'$  correspond to the same maximal cone $\cham$ and $\cv_{i;\seed}=\cv_{j;\seed'}$.
This is known to hold for $\cX$, $\Xfam$, and $\cX_q$.
We will show that it holds for $\Xfam_q$ as well in \thref{prop:q-periods} and \thref{cor:cocycle}. For notational simplicity, we will freely use \thref{prop:q-periods} and \thref{cor:cocycle} before proving them.
\end{remark}


\subsection{Quantum cluster algebras} \label{sec:quantum}

Here we review the quantization of $\cX$ by Fock and Goncharov (\cite{FG_cluster_ensembles}), 
and the quantization of $\cA$ by Berenstein and Zelevinsky (\cite{BZquantum}).
$\cX_q$ (respectively $\cA_q$) is a noncommutative $q$-deformation of $\cX$ (respectively $\cA$), where tori are replaced by quantum tori (which are $*$-algebras) and the mutation maps are $*$-algebra isomorphisms of the noncommuative fraction fields of the quantum tori.

\subsubsection{Quantum \texorpdfstring{$\cX$}{X}-varieties}
\label{sec:FG-q}

We again begin with fixed data $\Gamma$ and a seed $\seed$ (Section~\ref{normalcluster}).
Let $d:=\lcm\lrp{{d_i: i\in I}}$, and consider the noncommutative torus 
\eqn{\mathcal{T}_{N} :=\C[q^{\pm \frac{1}{d}}]\lra{X^n\, : \, n\in N,\, X^n \cdot X^{n'} = q^{\lrc{ n, n'}} X^{n+n'}}.}
Thus, we have
\begin{equation} \label{eq:qcommute}
    q^{\{ n', n\}} X^n \cdot X^{n'} = q^{\{ n, n'\}} X^{n'} \cdot X^n. 
\end{equation}
Observe that
$\mathcal{T}_N$ is a $*$-algebra over $\C[q^{\pm\frac{1}{d}}]$, 
with the involutive antiautomorphism $* : \mathcal{T}_{N} \rightarrow \mathcal{T}_{N}$ defined by 
\begin{equation} \label{eq:*map}
    *(X^n) = X^n, \qquad *(q) = q^{-1}. 
\end{equation}

\subsubsection*{Quantum dilogarithm}
To construct a non-commutative deformation of the space $\cX$, the {\it{quantum space}} $\cX_q$, Fock and Goncharov \cite{FG_cluster_ensembles} consider the \emph{quantum dilogarithm} defined by 
\eqn{
\vb{\Psi}_{q}(x) = \prod_{\ell=1}^{\infty} \lrp{1 + q^{2\ell-1} x}^{-1}.
}
Write 
\begin{equation} \label{eq:li}
    \mbox{Li}_{2}\lrp{x; q}  
:= \sum_{\ell \geq 1} \frac{x^{\ell}}{\ell(q^{\ell}-q^{-\ell})}.
\end{equation}
One can check that
\begin{equation} \label{eq:psi}
    \vb{\Psi}_{q}\lrp{ x}  = \exp \left( -\mbox{Li}_{2}\lrp{-x; q} \right) .
\end{equation}
See \cite[Section~3.2]{FG_cluster_ensembles} for further details.

For each seed $\seed$ we will have a copy of $\mathcal{T}_N$, which we denote $\mathcal{T}_{N;\seed}$.
Since $\mathcal{T}_{N;\seed}$ is a {\it{right Ore domain}}\footnote{A domain $A$ is a {\it{right Ore domain}} if for each nonzero elements $x,y\in A$, there exist $r,s\in A$ such that $xr=ys\neq 0$.}, we can take its right noncommutative fraction field $\TT_{N; \seed}$.
The \emph{quantum mutation map} $\mu_k^q$, $k \in \Iuf$, is the isomorphism ${\mu_k^q : \TT_{N;\seed'} \rightarrow \TT_{N;\seed}}$ given by conjugation by $\vb{\Psi}_{q^{1/d_k}}\lrp{X_{k;\seed}}$, 
\[
\mu_k^q := \vb{Ad}_{\vb{\Psi}_{q^{1/d_k}}\lrp{X_{k;\seed}}}.
\] 
Fock and Goncharov \cite{FG_cluster_ensembles} showed that $\mu_k^q$ is a $\ast$-algebras homomorphism, which follows from the fact that the equality $\vb{\Psi}_{q^{-1}}\lrp{x}=\vb{\Psi}_q\lrp{x}^{-1}$ holds.
They then decompose $\mu_k^q$ and express it in terms of cluster variables.
That is, they write  
\[\mu^q_k = \mu^{\sharp}_k \circ \mu_k',\] where $\mu^{\sharp}_k := \vb{Ad}_{\vb{\Psi}_{q^{1/d_k}}\lrp{X_k}}: \TT_{N;\seed} \rightarrow \TT_{N;\seed}$, and $\mu'_k: \TT_{N;\seed'} \rightarrow \TT_{N;\seed}$ relates the coordinates associated to the basis $\seed'$ of $N$ with the coordinates associated to the basis $\seed$. 
Explicitly, 
\eqn{ \mu_k':X_{i;\seed'}=X^{e_{i;\seed'}} \mapsto  X^{e_{i;\seed'}} = X^{e_{i;\seed}+[\epsilon_{ik}]_+e_{k;\seed}} = q^{-\hat{\epsilon}_{ik}[\epsilon_{ik}]_+}X^{e_{i;\seed}}X^{[\epsilon_{ik}]_+ e_{k;\seed}} =
q^{-\hat{\epsilon}_{ik}[\epsilon_{ik}]_+}X_{i;\seed}X^{[\epsilon_{ik}]_+}_{k;\seed}.}

One can compute $\mu_k^{\sharp}$ more explicitly as
\begin{equation}
X_i \mapsto
\begin{cases} \label{eq:fg}
X_i(1+q^{1/{d_k}}X_k )(1+q^{3/{d_k}}X_k ) \cdots (1+q^{(2 | \epsilon_{ik}|-1)/{d_k}}X_k ) & 
\epsilon_{ik}\leq 0,\\
X_i(1+q^{-1/{d_k}}X_k )^{-1}(1+q^{-3/{d_k}}X_k )^{-1} \cdots (1+q^{(1-2 |\epsilon_{ik}|)/{d_k}}X_k )^{-1} & \epsilon_{ik} \geq 0. 
\end{cases}
\end{equation}
Finally, in terms of cluster coordinates, we have
\eq{
\mu_k^q \lrp{X_{i;\seed'}} = 
\begin{cases}
    X_{i;\seed}^{-1} & \text{if } i=k\\
    X_{i;\seed} \lrp{ \prod_{\ell=1}^{\lrm{\epsilon_{ik}}} \lrp{ 1 + q^{(2 \ell - 1)/d_k} X_{k;\seed}^{-\sgn{\epsilon_{ik}} } }}^{-\sgn (\epsilon_{ik})} & \text{if } i\neq k.
\end{cases}
}{eq:muq}

\subsubsection{Quantum \texorpdfstring{$\cA$}{A}-cluster algebras} \label{sec:quantumA}

Now we recall Berenstein and Zelevinsky's treatment of quantum $\cA$-cluster algebras from \cite{BZquantum}.
We illustrate how one can translate between the two frameworks in Section~\ref{sec:relate}. 

Let $\Lambda$ be an $\lrm{I}\times \lrm{I}$ skew-symmetric integer matrix,
and let $\widetilde{B}$ be an $\lrm{I}\times \lrm{\Iuf}$ integer matrix.
The pair $\lrp{\Lambda, \widetilde{B}}$ is called \emph{compatible} if
$
\Bt^T \Lambda = 
\begin{pmatrix}
    D & 0
\end{pmatrix},
$
for a diagonal matrix $D$.
Berenstein and Zelevinsky show (\cite[Proposition 3.3]{BZquantum}) that the pair $(\Lambda, \widetilde{B})$ being compatible implies $\widetilde{B}$ is of full rank and the principal $\lrm{\Iuf} \times \lrm{\Iuf} $ submatrix $B$ is skew-symmetrizable by noting $DB = \widetilde{B}^T \Lambda \widetilde{B}$.
A compatible pair is the input for the Berenstein-Zelevinsky quantum $\cA$-cluster algebra construction.\footnote{Note in particular that while we can always quantize an $\cX$-variety, if the exchange matrix $\widetilde{B}$ (which is $\epsilon_{\Iuf \times I}^t$ in the Fock-Goncharov formalism) is not full rank, then $\cA$ does not admit a quantization by this procedure. }

We again consider a quantum torus $*$-algebra.
In Section~\ref{sec:FG-q}, the cluster coordinates on $\mathcal{T}_{N;\seed}$ were exponentials of the elements of the basis $\seed = \lrc{e_{i;\seed}: i\in I }$ of $N$, and the $q$-commutation relations were defined in terms of the skew form $\lrc{\cdot, \cdot}$ on $N$.
In this case the cluster variables are exponentials of $\lrc{f_{i;\seed}: i \in I}$, a basis for $M^\circ$.
We view $\Lambda$ as a skew form on $M^\circ$ and define $\mathcal{T}_{M^\circ}$ as 
\eqn{\mathcal{T}_{M^\circ}:=\C[q^{\pm \frac{1}{2}}]\lra{A^m \, :\, m\in M^\circ,\, A^m \cdot A^{m'}= q^{\frac{1}{2}\Lambda\lrp{m,m'}} A^{m+m'}} .}
This is a $*$-algebra over $\C[q^{\pm \frac{1}{2}}]$, with $*(A^m) = A^m$ and $*(q)= q^{-1}$. 
Again, we have a copy $\mathcal{T}_{M^\circ;\seed}$ for each seed $\seed$, and we denote the noncommutative fraction field of $\mathcal{T}_{M^\circ;\seed}$ by $\TT_{M^\circ;\seed}$.
We denote $A^{f_{i;\seed}} \in \mathcal{T}_{M^\circ;\seed}$ by $A_{i;\seed}$.
The quantum $\cA$-mutation defining $\cA_q$ is the $*$-algebra isomorphism $\TT_{M^\circ;\seed'} \to \TT_{M^\circ;\seed}$ given by
\eqn{ 
    \mu_k^q\lrp{A_{i;\vb{s}'}}= 
    \begin{cases}
        A^{-f_{k;\vb{s}} + \sum_{j: \epsilon_{kj}>0 } \epsilon_{kj} f_{j;\vb{s}} } +  A^{-f_{k;\vb{s}} - \sum_{j: \epsilon_{kj}<0 } \epsilon_{kj} f_{j;\vb{s}} } & \text{for } i=k\\
        A_{i;\vb{s}} & \text{for } i \neq k
    \end{cases}. 
}
See \cite[Proposition~4.9]{BZquantum}.

\subsection{Quantum scattering diagrams}\label{sec:q-scat}

In this section, we will outline how to construct the quantum version of scattering diagrams for $\cX$ as in \cite[arXiv~version~1, Construction~1.31]{GHKK} and \cite{mandel2015refined}.
We have outlined how to obtain consistent scattering diagrams from skew-symmetric Lie algebras in Section \ref{sec:scat}. 
Thus we only need to merge the quantum algebras into the formulation. 
Gross-Hacking-Keel-Kontsevich define the following Lie bracket on $\mathcal{T}_{N}$:
\eqn{\lrb{X,Y} = Y X - XY.}
It is straightforward to check that this indeed satisfies the properties of a Lie bracket using \eqref{eq:qcommute}.
$\widehat{\fg}$ will be a Lie subalgebra of a localization of $\mathcal{T}_N$ with this Lie bracket.
To define it, Gross-Hacking-Keel-Kontsevich first take a multiplicative subset $S$ of $\C[q^{\pm 1/d}] $, defined as the complement of the union of the two prime ideals $(q^{1/d}-1)$ and $(q^{1/d}+1)$.
They then take $\C_q$ to be the localization of $\C[q^{\pm 1/d}] $ by $S$: $\C_q:= S^{-1}\C[q^{\pm 1/d}]$.
Then $\widehat{\fg}$ is defined to be the free $\C_q$-submodule of 
\eqn{\C(q^{\frac{1}{d}})\lra{X^n\, : \, n\in N,\, X^n \cdot X^{n'} = q^{\lrc{ n, n'}} X^{n+n'}}    }
with basis $\left\{ \left. \widehat{X}^n := \frac{X^n}{q-q^{-1}} \; \right| \; n \in N^+_\seed \right\}$.
(See \eqref{eq:N+} for the definition of $N^+_\seed$.)
This is in fact a Lie subalgebra.
Observe that
\eq{
[ \widehat{X}^n, \widehat{X}^{n'}] = \frac{q^{\{n', n\}}-q^{\{n, n'\}}}{q-q^{-1}}\widehat{X}^{n+n'},
}{eq:LieBracket}
and suppose  $ \{n', n\} >0 $. 
Then
\eqn{
\frac{q^{\{n', n\}}-q^{\{n, n'\}}}{q-q^{-1}} 
= \frac{q^{\{n, n'\}}}{q^{-1}}
\left(\frac{1+q^{2/d}+ \cdots + q^{2(\{n', n\}-1)}}{
1+q^{2/d}+ \cdots + q^{2(d-1)/d}}
\right)} 
is in $\C_q$ as neither $q^{1/d}-1$ nor $q^{1/d}+1$ divide $1+q^{2/d}+ \cdots + q^{2(d-1)/d}$.
The case $ \{n', n\} < 0 $ is similar, and the case $  \{n', n\} =0 $ is clear.
Observe that $\widehat{\fg}$ is a skew-symmetric (with respect to $\lrb{\cdot,\cdot}$) $N^+_\seed$-graded Lie algebra, and in order to apply  \thref{thm:ks} all that remains is to provide the initial scattering diagram.  

Choose the monoid $P \subset N$ as in the classical cluster scattering diagrams construction. 
Then elements of $\widehat{G}= \exp ( \widehat{\fg})$ act on $\widehat{\C [q^{\pm 1/d}][P]}$ by conjugation, i.e. $\exp(g): X^n \mapsto \exp(-g) X^n \exp (g)$. 
Recall the function $\mbox{Li}_2 (x;q) $ defined as in \eqref{eq:li}.
Note that in particular $\mbox{Li}_2 (-X^{e_i};q^{1/d_i}) \in \varprojlim \widehat{\fg} / \widehat{\fg}^{>k} $. 
Indeed the conjugation by $\vb{\Psi}_{q^{1/d_i}}$ can be expressed as 
\begin{align*}
   & \vb{\Psi}_{q^{1/d_i}} \lrp{X^{e_i}}^{-1} X^{e_j} \vb{\Psi}_{q^{1/d_i}} \lrp{X^{e_i}} \\ 
   =&
\begin{cases}
(1+q^{1/{d_i}}X^{e_i} )(1+q^{3/{d_i}}X^{e_i} ) \cdots (1+q^{(2 \{ e_j, e_i \}d_i -1)/{d_i}}X^{e_i} )X^{e_j} & \{ e_j , e_i \} >0 \\
(1+q^{-1/{d_i}}X^{e_i} )^{-1}(1+q^{-3/{d_i}}X^{e_i} )^{-1} \cdots (1+q^{(2 \{ e_j, e_i \}d_i +1)/{d_i}}X^{e_i} )^{-1}X^{e_j} & \{ e_j , e_i \} \leq 0 
\end{cases}.
\end{align*}

In fact, by comparing with \eqref{eq:fg}, one can see this corresponds to the inverse of the automorphism $\mu_k^{\sharp}$ in Fock and Goncharov's setting.

Now consider the scattering diagram
\[
\widehat{\scat}_{\text{in}, \seed} :=
\left\{ \left( v_i^{\perp},  \vb{\Psi}_{ q^{1/d_i}}\lrp{ X^{e_i} } \right)  \mid i \in \Iuf  \right\}, 
\]
where $v_i = p_1^*(e_i)$. 
Then by \thref{thm:ks}, there is a unique (up to equivalence) consistent scattering diagram $\widehat{\scat}_{\seed}$ containing $\widehat{\scat}_{\text{in}, \seed} $, with only outgoing walls in $\widehat{\scat}_{\seed}\setminus \widehat{\scat}_{\text{in}, \seed} $.

Given a compatible pair, one can similarly define the quantum $\cA$ scattering diagram using the $\Lambda$ as the skew-symmetric form of the construction.
In this case the initial scattering diagram is
\eqn{\widehat{\scat}_{\text{in},\seed}^{\cA}:= \lrc{\left.\lrp{e_i^{\perp}, \vb{\Psi}_{q^{1/{d_i}}} \lrp{A^{v_i} }}\right| i \in \Iuf}.
}
Details of this construction may be found in \cite[\S4.2]{mandel2015refined} and \cite[\S4.3]{davison2019strong}.
We will not repeat the particulars here, but we will consider a quantum $\cA$ scattering diagram as an example in the next subsection.  

Assuming that there exists a compatible pair for $\cA_q$, one can construct a compatible pair for $\cA_{\prin,q}$ such that $\cA_q$ is a fiber of $\cA_{\prin,q}$ and quantum theta functions on $\cA_q$ are restrictions of quantum theta functions on $\cA_{\prin,q}$. See \cite[Lemma~4.5]{davison2019strong}.
However, in a subtle difference from the classical case, if a compatible pair does not exist for the $\cA$-algebra, one cannot view $\cA_q$ as a fiber of $\cA_{\prin,q}$ and restrict the quantum $\cAp$ theta functions to obtain quantum $\cA$ theta functions-- which do not exist.  
The issue is that in this case the new ``principal coefficient'' variables that have been introduced\footnote{They must be treated as {\emph{variables}} rather than {\emph{coefficients}} if the $\cAp$ algebra is to have a compatible pair when $\cA$ does not.}
will not commute with the original cluster variables.
To recover the $\cA$-algebra, we would have to specialize the principal coefficient variables to $1$,
at which point their failure to commute with the original cluster variables would yield relations of the form $A_i = q^{c} A_i$ for some $c\neq 0$.
Related notions will appear in \S\ref{sec:q-pstar}.

\subsubsection{Non-positivity for theta functions} \label{sec:non-pos}

While the construction of quantum scattering diagrams is very similar to the classical construction, there is an important difference between the classical and quantum settings which we would like to highlight here. 
In \cite{Davison}, Davison proved the {\it{positivity conjecture}} for skew-symmetric quantum $\cA$-cluster algebras: 
{\it{every cluster variable is a Laurent polynomial in the monomials of any other cluster, with coefficients in $\Z_{\geq 0}[q^{\pm \frac{1}{2}}]$}}.
He has since taken this further with Mandel. 
In \cite{davison2019strong} they show the {\it{strong positivity conjecture}} for quantum theta bases for skew-symmetric type quantum cluster algebras, meaning that the structure constants for decomposing products of quantum theta functions in this setting are also Laurent polynomials in the quantum parameter with positive integer coefficients. 

\cite{greedy} indicates that the rank 2 (classical) theta basis is the same as the greedy basis. Meanwhile, 
Lee-Li-Rupel-Zelevinsky develop rank 2 quantum greedy bases in \cite{lee2016existence} and give an example of a skew-symmetrizable quantum cluster algebra where quantum positivity fails for the quantum greedy basis in \cite{Lee9712}.
We repeat this example, but use scattering diagram machinery, thus treating the quantum theta basis.
In doing this, we find:

\begin{prop}
Quantum positivity fails for the quantum theta basis.
That is, for a quantum cluster algebra of skew-symmetrizable type, there may be quantum theta functions which, when expressed in terms of initial quantum cluster variables, have coefficients which are in $\Z[q^{\pm \frac{1}{2}}]$ but not $\Z_{\geq 0}[q^{\pm \frac{1}{2}}]$.
\end{prop}

Consider the quantization of the rank two cluster algebra $\cA(2,3)$.
In the Fock-Goncharov formulation, this is defined by the exchange matrix
$\epsilon = \begin{pmatrix}
0& -3  \\ 2 & 0
\end{pmatrix}$, 
i.e. 
$D = \begin{pmatrix}
2 & 0  \\ 0 & 3
\end{pmatrix}$
and 
${\hat{\epsilon} =
\begin{pmatrix}
0& -1  \\ 1 & 0
\end{pmatrix}}$.
In the formulation of Berenstein-Zelevinsky, it corresponds to the compatible pair $(\Lambda, B)$, where 
$\Lambda = \begin{pmatrix}
0 & 1\\
-1 & 0
\end{pmatrix}  $ and 
$B= \begin{pmatrix}
0& 2  \\ -3 & 0
\end{pmatrix}$, and hence
$D' = \begin{pmatrix}
3 & 0  \\ 0 & 2
\end{pmatrix}$.

Consider the seed $\{e_1 = (1,0), e_2 = (0,1) \}$, and the corresponding basis
$\{ f_1 = \frac{1}{2} e_1^*, f_2 = \frac{1}{3} e_2^*\}$ of $M^{\circ}$. 
For ease of comparison, we will adopt the notation of \cite{Lee9712} here, set $\qbz=v^{-2}$, and so consider the quantum torus with relation 
$A_2 A_1 = v^{2} A_1 A_2$.
Let us also note that, as we will see in \S\ref{sec:q-pstar}, to relate Berenstein-Zelevinsky conventions to the Fock-Goncharov conventions we have adopted, we should set $\qfg^{\frac{1}{d}}=\qbz^{-\frac{1}{2}}=v$. 
In this case the initial scattering diagram for the quantum $\cA$ variety is
$$\widehat{\scat}_{\text{in}}^\cA = \lrc{  \lrp{ e_1^\perp , \vb{\Psi}_{v^{3}}\lrp{ A_2^{-3}} }, \lrp{ e_2^\perp , \vb{\Psi}_{v^{2}}\lrp{ A_1^{2}} } } .$$ 
Observe that, for $u = u_1 f_1 + u_2 f_2$,
\eq{\vb{\Psi}_{v^{3}}\lrp{ A_2^{-3}}^{-1} A^u \vb{\Psi}_{v^{3}}\lrp{ A_2^{-3}}
= \lrp{ \prod_{\ell=1}^{\lrm{u_1}} \lrp{1+v^{\sgn(u_1) 3 \lrp{2 \ell -1}} A_2^{-3} }^{\sgn(u_1)} } A^u }{eq:g2}
and
\eq{\vb{\Psi}_{v^{2}}\lrp{ A_1^{2}}^{-1} A^u \vb{\Psi}_{v^{2}}\lrp{ A_1^{2}}
= \lrp{ \prod_{\ell=1}^{\lrm{u_2}} \lrp{1+v^{\sgn(u_2) 2 \lrp{2 \ell -1}} A_1^{2} }^{\sgn(u_2)} } A^u. }{eq:g1}
It is straightforward to check that the initial wall functions reproduce the quantum mutation formulae for the initial seed of this quantum cluster algebra.

Now set $L:= p_1^*(N)$, set $L^+:=p_1^*\lrp{N_\seed^+}\setminus\lrc{0}$, and let $d:L\to \Z$ be given by $\lra{e_1-e_2,\ \cdot\ }$.
Then $d\lrp{p_1^*(e_1)} = d\lrp{p_1^*(e_2)} = 1$.
Now set $P:= \lrp{\Sp_{\R_{\geq 0}} \lrc{p_1^*(e_1),p_1^*(e_2)} } \bigcap M^\circ = \Sp_{\Z_{\geq 0}} \lrc{f_1, -f_2}$,  
and let $J$ be the monomial ideal in $\C[P]$ generated by $L^+$.
Observe that $\widehat{\scat}_{\text{in}}^\cA $ is consistent for $G_1$.
To see the failure of quantum positivity in this example, we only need to compute the consistent scattering diagram for $G_2$, denoted $\widehat{\scat}_{2}^\cA$.

Consider the following loop $\gamma$:


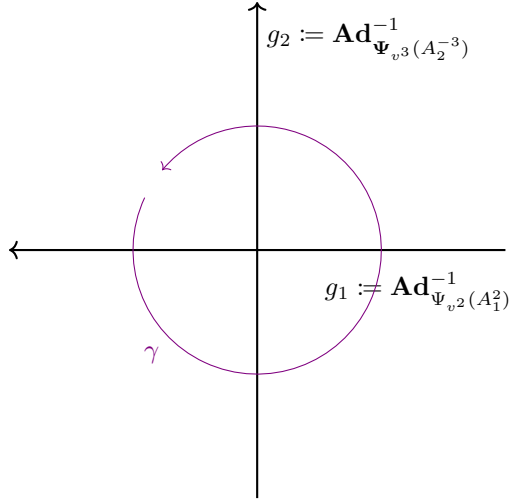
\noindent
\begin{center}
\begin{tikzpicture}[scale=1.1]

    \def\l{3}
    \def\d{1}
    \def\op{.5}

    \path (-\l,0) coordinate (3) --++ (\l,0) coordinate (0) --++ (\l,0) coordinate (1);
    \path (0,\l) coordinate (2) --++ (0,-2*\l) coordinate (4);
    \path (0,0) --++ (-3*\d,2*\d) coordinate (5);

    \draw[thick,->] (1) -- (3);
    \draw[thick,->] (4) -- (2);

    \node at (.65*\l,-.5) {$g_1 \coloneqq \vb{Ad}_{\Psi_{v^2}(A_1^2)}^{-1}$};
    \node at (.45*\l,.85*\l) {$g_2 \coloneqq \vb{Ad}_{\vb{\Psi}_{v^3}(A_2^{-3})}^{-1}$};

    \draw[color=violet, ->] ([shift={(155:.5*\l)}] 0,0) arc (155:500:.5*\l);
    \node [color=violet] at (225:.6*\l) {$\gamma$}; 

\end{tikzpicture}.
\end{center}

For $\widehat{\scat}_{\text{in}}^\cA $, the wall crossing automorphism associated to $\gamma$ is $\fp_\gamma^1 = g_2 \circ g_1^{-1} \circ g_2^{-1} \circ g_1$.
To obtain $\widehat{\scat}_{2}^\cA$, we add walls such that $\fp_\gamma^2$ is the identity in $G_2$.
But for $u = u_1 f_1 + u_2 f_2$, $\lrp{\fp_\gamma^1}^{-1} =g_1^{-1} \circ g_2  \circ g_1 \circ g_2^{-1} $ is given by\footnote{See Appendix~\ref{app:wallx} for this computation.
While the calculation is a bit lengthy, the referee has pointed out to us that the same result can be obtained with considerably less effort using \cite[Lemma~2.9]{mandel2015refined}, or the new techniques of the pentagon relation for quantum dilogarithm elements in \cite{Nak22}. In fact, a more general non-positivity result may be found in \cite[Theorem~5.6]{Nak22}.
}
\eq{\lrp{\fp_\gamma^1}^{-1} (A^u)= 
\Biggl(& 1 +
    \biggl(
        \sgn(u_1) \lrp{ v^{-4} + 1 + v^4 } \sum_{\ell = 1}^{\lrm{u_1}} v^{\sgn(u_1) 3 (2 \ell -1)}
        +
        \sgn(u_2) \lrp{ v^{-3} + v^3 } \sum_{\ell = 1}^{\lrm{u_2}} v^{\sgn(u_2) 2 (2 \ell -1)}\\
        &+
        \sgn(u_1) \sgn(u_2) \lrp{ v^{6} - v^{-6} } \sum_{\ell_1 = 1}^{\lrm{u_1}} \sum_{\ell_2 = 1}^{\lrm{u_2}} v^{\sgn(u_1) 3 (2 \ell_1 -1)+\sgn(u_2) 2 (2 \ell_2 -1)}
    \biggr) A^{2 f_1 - 3 f_2} + \cdots 
\Biggr) A^u.
}{eq:g3}
Note that there is only one term whose exponent vector $m$ satisfies $d(m)=2$.
All terms encompassed by the ellipses are higher order, i.e. have exponent vectors $m'$ with $d(m') > 2$. 
As a result, $\widehat{\scat}_{2}^\cA$ has only one new wall.
Let $g_3 \in G_2$ denote the map
\eqn{A^u \mapsto \Biggl(& 1 +
    \biggl(
        \sgn(u_1) \lrp{ v^{-4} + 1 + v^4 } \sum_{\ell = 1}^{\lrm{u_1}} v^{\sgn(u_1) 3 (2 \ell -1)}
        +
        \sgn(u_2) \lrp{ v^{-3} + v^3 } \sum_{\ell = 1}^{\lrm{u_2}} v^{\sgn(u_2) 2 (2 \ell -1)}\\
        &+
        \sgn(u_1) \sgn(u_2) \lrp{ v^{6} - v^{-6} } \sum_{\ell_1 = 1}^{\lrm{u_1}} \sum_{\ell_2 = 1}^{\lrm{u_2}} v^{\sgn(u_1) 3 (2 \ell_1 -1)+\sgn(u_2) 2 (2 \ell_2 -1)}
    \biggr) A^{2 f_1 - 3 f_2} + \cdots 
\Biggr) A^u.}
The new wall is $\lrp{\wall_3 , g_3}$, where $\wall_3:= \R_{\geq 0} \cdot \lrp{-2 f_1 +3 f_2}$. The scattering diagram $\widehat{\scat}_{2}^\cA$ is shown in Figure~\ref{fig:Scat2}.
The failure of quantum positivity in this example stems from this wall function having coefficients in $\Z[v^{\pm 1}]$ rather than $\Z_{\geq 0}[v^{\pm 1}]$.


\noindent
\begingroup
\captionsetup{type=figure}
\begin{center}
\begin{tikzpicture}[scale=1.1]

    \def\l{3}
    \def\d{1}
    \def\op{.5}

    \path (-\l,0) coordinate (3) --++ (\l,0) coordinate (0) --++ (\l,0) coordinate (1);
    \path (0,\l) coordinate (2) --++ (0,-2*\l) coordinate (4);
    \path (0,0) --++ (-2*\d,3*\d) coordinate (5);

    \draw[thick,->] (1) -- (3);
    \draw[thick,->] (4) -- (2);
    \draw[thick,->] (0) -- (5); 

    \node at (.65*\l,-.5) {$g_1 \coloneqq \vb{Ad}_{\Psi_{v^3}(A_1^2)}^{-1}$};
    \node at (.45*\l,.85*\l) {$g_2 \coloneqq \vb{Ad}_{\vb{\Psi}_{v^2}(A_2^{-3})}^{-1}$};

    \node at (-.5*\l,.6*\l) {$g_3$};

\end{tikzpicture}.
\captionof{figure}{\label{fig:Scat2} The scattering diagram $\widehat{\scat}_{2}^\cA$. } 
\end{center}
\endgroup


While the theta basis is indexed by $\gv$-vectors, the greedy basis is indexed by $\vb{d}$-vectors.
A piecewise linear map sending the $\gv$-vector of a classical theta function to its $\vb{d}$-vector (the theta function now being viewed as a greedy basis element) is given in \cite{greedy}.
Due to a convention difference (amounting to a transpose of the exchange matrix) we need a slightly modified version of the map from \cite{greedy}.
For the rank two cluster algebra $\cA(b,c)$, define 
\eqn{T: M^\circ_\R &\to M^\circ_\R\\
m &\mapsto \begin{cases}
m & \text{if } m_1 \geq 0\\
m + (0,c m_1) & \text{if } m_1 \leq 0
\end{cases}.
}
Then $T(\gv(\tf_m)) = -d(\tf_m)$ by a similar argument as in \cite[Remark~4.5]{greedy}. 

In \cite{Lee9712}, they considered the greedy basis element corresponding to the $\vb{d}$-vector $(3,4)$:
\[
X[3,4] = \sum_{p,q \geq 0} e(p,q) X^{(2p-3, 3q-4)}. 
\]
In particular, $e(2,1) = v^2-1+v^{-2}$. 
This term $e(2,1)$ is the coefficient of $X^{(1,-1)}$, or in the Fock-Goncharov notation we have adopted $A^{f_1-f_2}$.

Note that $T(-3,5) = (-3, -4)$. Then the $\vb{d}$-vector $(3,4)$ corresponds to the $\gv$-vector $(-3,5)$ in our setting.
Now we compute the coefficient of $A^{f_1-f_2}$ in $\tf_{(-3,5)}$, expressed in terms of the initial cluster.
We pick a basepoint $Q$ in the positive quadrant and consider all possible broken lines whose initial decorating monomial is $A^{-3 f_1 +5 f_2}$, final decorating monomial is a scalar multiple of $A^{f_1-f_2}$, and endpoint is $Q$.
There is in fact only one, shown in Figure~\ref{fig:brokenline}.
We obtain precisely the same non-positive coefficient $\lrp{v^{-2}-1+v^2}$ of $A^{f_1-f_2}$ for the quantum theta function $\tf_{-3 f_1 +5 f_2}$ as Lee-Li-Rupel-Zelevinsky do for the quantum greedy basis element $X[3,4]$.


\noindent
\begingroup
\captionsetup{type=figure}
\begin{center}
\begin{tikzpicture}[scale=1.5]

    \def\l{3}
    \def\d{1.5}
    \def\x{.33}
    \def\y{.51}
    \def\op{.3}

    \path (-\l,0) coordinate (3) --++ (\l,0) coordinate (0) --++ (\l,0) coordinate (1);
    \path (0,1.5*\l) coordinate (2) --++ (0,-2*\l) coordinate (4);

    \path[name path = bbox] (-\l,-.5*\l) -- (\l,-.5*\l) -- (\l,1.5*\l) -- (-\l,1.5*\l) -- cycle;

    \path (0,0) --++ (-2*\d,3*\d) coordinate (5);

    \draw[thick,->, opacity =\op] (1) -- (3);
    \draw[thick,->, opacity =\op] (4) -- (2);
    \draw[name path=wall3, thick,->, opacity =\op] (0) -- (5); 

    \node [opacity =\op] at (.9*\l,-.3) {$g_1$};
    \node [opacity =\op] at (.1*\l,1.4*\l) {$g_2$};

    \node [opacity =\op] at (-.95*\l,1.25*\l) {$g_3$};

    \node [circle, fill, inner sep = 1.5pt, color = blue] (Q) at (\x,\x) {};
    \coordinate (x1) at (2*\x,0);
    \coordinate (x2) at (0,-2*\x);
    \path [name path=bl] (x2) --++ (-5*\x,10*\x); 

    \path [name intersections={of=bl and wall3,by=x3}];

    \path [name path=ray] (x3) --++(-3*\y,5*\y);
    \path [name intersections={of=ray and bbox,by=in}];

    \draw [thick, color=blue] (Q.center) -- (x1) -- (x2) -- (x3)  node [color=blue, pos = .5, sloped, below] {$\lrp{v^{-2}-1+v^2} A^{-f_1 +2 f_2}$} -- (in) node [color=blue, pos = .5, sloped, above] {$A^{-3f_1 +5 f_2}$};

    \node [color=blue] at (1.5,-.7) {$\lrp{v^{-2}-1+v^2} A^{-f_1 - f_2}$}; 

    \node [color=blue] at (1.7,.3) {$\lrp{v^{-2}-1+v^2} A^{f_1 - f_2}$}; 
    
    \node [color=blue] at (.2,.55) {$Q$};
    
\end{tikzpicture}.
\captionof{figure}{\label{fig:brokenline} The broken line giving the $A^{f_1-f_2}$ summand of $\tf_{-3f_1+5f_2}$. 
While we have drawn this atop $\widehat{\scat}^\cA_2$, it is not difficult to see that this is the only broken line contributing to any order.} 
\end{center}
\endgroup



\section{The quantum space \texorpdfstring{$\Xfsp_q$}{Xfamq}}\label{sec:Xfamq}

In Section \ref{sec:FG-q}, we recall the quantization of a cluster Poisson variety $\cX$, which directly extends to the quantization for its special completion $\Xsp$, by describing the quantum mutation map $\mu_k^q$ as a composition $\mu_k^\sharp \circ \mu_k' $. 
To establish the quantum version of the $\Xfsp$ family deforming $\Xsp$, we introduce coefficients to both $\mu_k^\sharp$ and $\mu_k'$.
Consider the ring $R:=\C\lrb{t_i: i\in I}$.
We are going to work over the $*$-ring $R[q^{\pm \frac{1}{d}}]$, where $d:=\lcm\lrc{d_i:i\in I}$ and the involution map $*$ is defined as in \eqref{eq:*map}.

Geometrically, we want to associate a noncommutative affine scheme to each maximal cone $\cham$ of the $\gv$-vector fan $\Sigma$ and glue them via a quantum version of \eqref{eq:Xfammu}, or a version of \eqref{eq:muq} with coefficients. Algebraically, each maximal cone $\cham$ determines a $*$-algebra $A_{\cham}$ over $R[q^{\pm \frac{1}{d}}]$, a noncommutative polynomial ring given by
\eqn{ A_{\cham}:= R\lrb{q^{\pm \frac{1}{d}}}\lra{\Xt_{\cham}^n: n \in N,\  \Xt_{\cham}^{n} \Xt_{\cham}^{n'} 
= 
q^{\lrc{n,n'}} \Xt_{\cham}^{n+n'}}.}
In particular, if we denote $\widehat{\epsilon}_{ij}=\{e_i,e_j\}$ we have
\eqn{ q^{- \widehat{\epsilon}_{ij}} \Xt_{i;\cham} \Xt_{j;\cham} 
= 
q^{- \widehat{\epsilon}_{ji}} \Xt_{j;\cham} \Xt_{i;\cham} .}

The quantum mutation with coefficients will be a $*$-algebra isomorphism of the noncommutative fraction fields of these $*$-algebras.\footnote{We review the notion of a noncommutative fraction field in the following Subsection \ref{sec:qAff} and show that it can be applied here.}
Thinking of the geometric interpretation, we will express the noncommutative fraction field of $A_\cham$ by $\sK\lrp{\A^{\lrm{I}}_{M;\cham,q}\lrp{R}}$.

\subsection{Quantum affine algebra} \label{sec:qAff}
We begin by reviewing some properties of the spectrum of a not-necessarily-commutative ring. 
For further details about noncommutative Noetherian rings, we refer the reader to \cite{GW} and \cite{McRob}.

\begin{definition}
Let $A$ be a ring. 
\begin{itemize}
       \item A proper two-sided ideal $P$ of $A$ is {\it{prime}} if whenever $I$ and $J$ are two-sided ideals of $A$ with $IJ\subseteq P$, either $I\subseteq P$ or $J\subseteq P$. \footnote{This definition is equivalent to the usual one in the commutative case, see \cite[Proposition~3.1]{GW}.}
    \item The {\it{prime spectrum}} $\mathrm{Spec}(A)$ is the set of all two-sided prime ideals of $A$.
\end{itemize}
\end{definition}

The prime spectrum $\mathrm{Spec}(A)$ is a topological space with the {\it patch topology}: the closed sets are the subsets $X$ of $\mathrm{Spec}(A)$ such that any prime ideal of $A$ which is the intersection of some primes in $X$ belongs to $X$ (see \cite[4.6.14]{McRob}).

It is convenient to think of $A_{\cham}$ as an {\it{iterated skew polynomial ring}}. 
We illustrate this concept in the case of adjoining three noncommuting variables.
The general case is simply an iteration of this process. See \cite[Chapter~1]{GW} for a more complete treatment of iterated skew polynomial rings. 
Consider the $R[q^{\pm \frac{1}{d}}]$-automorphism $\alpha_2$ of $R[q^{\pm \frac{1}{d}}][\Xt_{1;\cham}]$ given by 
\[\alpha_2\lrp{\Xt_{1;\cham}}=q^{-2\widehat{\epsilon}_{12}}\Xt_{1;\cham}.\]
The ring $R[q^{\pm \frac{1}{d}}][\Xt_{1;\cham}][\Xt_{2;\cham};\alpha_2]$ is the noncommutative polynomial ring where $\Xt_{2;\cham}\Xt_{1;\cham}=\alpha_2\lrp{\Xt_{1;\cham}}\Xt_{2;\cham}$. 
Now consider the $R[q^{\pm \frac{1}{d}}]$-automorphism $\alpha_3$ of $R[q^{\pm \frac{1}{d}}][\Xt_{1;\cham}][\Xt_{2;\cham};\alpha_2]$ given by 
\[\alpha_3\lrp{\Xt_{1;\cham}}=q^{-2\widehat{\epsilon}_{13}}\Xt_{1;\cham},
 \text{ and} \quad  \alpha_3\lrp{\Xt_{2;\cham}}=q^{-2\widehat{\epsilon}_{23}}\Xt_{2;\cham}.
\]
Then the ring $R[q^{\pm \frac{1}{d}}][\Xt_{1;\cham}][\Xt_{2;\cham};\alpha_2][\Xt_{3;\cham};\alpha_3]$ is the noncommutative polynomial ring satisfying \eqn{\Xt_{2;\cham}\Xt_{1;\cham}=\alpha_2\lrp{\Xt_{1;\cham}}\Xt_{2;\cham}\ , \qquad \Xt_{3;\cham}\Xt_{1;\cham}=\alpha_3\lrp{\Xt_{1;\cham}}\Xt_{3;\cham}\ , \text{ and } \qquad \Xt_{3;\cham}\Xt_{2;\cham}=\alpha_3\lrp{\Xt_{2;\cham}}\Xt_{3;\cham}.} 
With this description, we can assert that $A_{\cham}$ is a Noetherian domain.

\begin{prop}\thlabel{prop:Noether}
The ring $A_\cham$ is a Noetherian non-commutative domain.
\end{prop}
\begin{proof}
By induction on $\lrm{I}$. Note that $R[q^{\pm \frac{1}{d}}]$ is a (commutative) Noetherian domain, so $R[q^{\pm \frac{1}{d}}][\Xt_{1;\cham}]$ also is a (commutative) Noetherian domain. The skew Hilbert basis theorem (see \cite[Theorem~1.14]{GW})
implies that $R[q^{\pm \frac{1}{d}}][\Xt_{1;\cham}][\Xt_{2;\cham};\alpha_2]$ is a Noetherian ring.
By construction it is also a domain.
Assume the claim holds for all $i< \lrm{I}$.

Now, since $R[q^{\pm \frac{1}{d}}]\langle\Xt_{1;\cham},\dots,\Xt_{\lrm{I};\cham}\rangle/\sim$ is equal to $R[q^{\pm \frac{1}{d}}][\Xt_{1;\cham}][\Xt_{2;\cham};\alpha_2]\cdots[\Xt_{\lrm{I}-1;\cham};\alpha_{\lrm{I}-1}][\Xt_{\lrm{I};\cham};\alpha_{\lrm{I}}]$, by the induction hypothesis $R[q^{\pm \frac{1}{d}}][\Xt_{1;\cham}][\Xt_{2;\cham};\alpha_2]\cdots[\Xt_{\lrm{I}-1;\cham};\alpha_{\lrm{I}-1}]$ is a Noetherian ring, then by the skew Hilbert basis theorem we concluded that $R[q^{\pm \frac{1}{d}}][\Xt_{1;\cham}][\Xt_{2;\cham};\alpha_2]\cdots[\Xt_{\lrm{I}-1;\cham};\alpha_{\lrm{I}-1}][\Xt_{\lrm{I};\cham};\alpha_{\lrm{I}}]$ is a Noetherian ring. Finally, note that since by iteration $R[q^{\pm \frac{1}{d}}][\Xt_{1;\cham}][\Xt_{2;\cham};\alpha_2]\cdots[\Xt_{\lrm{I}-1;\cham};\alpha_{\lrm{I}-1}]$ is a domain, we obtain that $A_{\mathcal{G}}$ is also a domain.
\end{proof}

By \cite[Theorem~1]{Goldie}, a ring without zero divisors which satisfies the ascending chain condition for right ideals has a right noncommutative fraction field.
Combining with \thref{prop:Noether}, an immediate corollary is that we can take the noncommutative fraction field of $A_{\cham}$. 
We denote the right noncommutative fraction field of $A_\cham$ by $\sK\lrp{\A^{\lrm{I}}_{M;\cham,q}\lrp{R}}$.

\begin{remark}
By the universal property of right rings of fractions, the right noncommutative fraction field of $A_\cham$ is canonically isomorphic to the right noncommutative fraction field of any localization of $A_\cham$.  
This will allow us to glue quantum affine schemes by noncommutative ``birational'' maps $\mu_{k,\vb{t};\cham}^q$ in Section~\ref{sec:qmut}.
\end{remark}

We can consider the multiplicatively closed set $S$ consisting of the powers of the cluster variables $\lrc{\Xt_{i;\cham}: i \in I}$ of $A_{\cham}$. 
Observe that the cluster variables are all normal elements of $A_{\cham}$, meaning $\Xt_{i;\cham}A_{\cham}=A_{\cham}\Xt_{i;\cham}$. Then, we construct the noncommutative localization of $A_\cham$ by $S$ as
$\mathcal{T}_{\cham}= R[q^{\pm \frac{1}{d}}] \lra{\Xt_{\cham}^n: n \in N}/\sim$ 
with the same relations as before: 
\eqn{\Xt_{\cham}^{n} \Xt_{\cham}^{n'} 
= 
q^{\lrc{n,n'}} \Xt_{\cham}^{n+n'}.}
The ring $\mathcal{T}_{\cham}$ is the {\it{quantum torus algebra over $R[q^{\pm \frac{1}{d}}]$}}.

\begin{prop}
There is a 1-1 correspondence between the elements in $\mathrm{Spec}\lrp{A_{\cham}}$ that do not contain the element $\Xt_{i;\cham}$ for all $i=1,\dots,\lrm{I}$ and the elements of $\mathrm{Spec}\lrp{\mathcal{T}_{\cham}}$.
\end{prop}
\begin{proof}
Since $A_{\cham}$ is a Noetherian ring, the results follows by \cite[Proposition~2.1.16]{McRob}.
\end{proof}

While the points of $\Spec\lrp{\mathcal{T}_{\cham}}$ are not so straightforward to relate to the points of the spectrum of a commutative ring-- that is, to points of a classical affine scheme--
the situation becomes simpler if we localize our base ring appropriately.

\begin{definition}
Define the localized base ring $R_q$ by
\eqn{R_q:=R[q^{\pm \frac{1}{d}}]_{\{q^{\frac{k}{d}} - q^{-\frac{k}{d}} \,:\, k\in \Z_{>0} \}},} 
and define the {\it{quantum torus algebra over $R_q$}} by
$\mathcal{T}_{\cham;q}:= R_q \lra{\Xt_{\cham}^n: n \in N}/\sim$ 
with the same relations as before: 
\eqn{\Xt_{\cham}^{n} \Xt_{\cham}^{n'} 
= 
q^{\lrc{n,n'}} \Xt_{\cham}^{n+n'}.}
\end{definition}

Based on the arguments given in \cite[Proposition~11.2]{BZquantum} and \cite[Lemma~11.33]{Lorenz}, we describe the elements of $\Spec\lrp{\mathcal{T}_{\cham;q}}$ in terms of the prime ideals of a commutative ring.\footnote{In the first version of this paper we erroneously stated \thref{prop:ideals} for $\Spec\lrp{\mathcal{T}_{\cham}}$ rather than $\Spec\lrp{\mathcal{T}_{\cham;q}}$.  Thanks to the anonymous referee for pointing out a mistake in our original statement and proof.}

\begin{prop}\thlabel{prop:ideals}
Let $\mathcal{Z}$ be the center of the localized quantum torus algebra $\mathcal{T}_{\cham;q}$. There is a 1-1 correspondence between the ideals of $\mathcal{Z}$ and the two-sided ideals of $\mathcal{T}_{\cham;q}$ given by extension from $\mathcal{Z}$ to $\mathcal{T}_{\cham;q}$, with inverse given by contraction. 
In particular, this gives a bijection between the sets  $\mathrm{Spec}\lrp{\mathcal{Z}}$ and  $\mathrm{Spec}\lrp{\mathcal{T}_{\cham;q}}$.
\end{prop}
\begin{proof}

Define the lattice map
\eqn{p^*: N &\to M^\circ\\
n &\mapsto \lrc{n, \ \cdot\ }.\footnotemark }
\footnotetext{A more general $p^*$ map was defined in Section~\ref{sec:cluster-scat}.} 
Now for any $m \in p^*(N)$ consider the set
\eqn{\mathcal{T}_{\cham, m;q}=\left\{ f \in \mathcal{T}_{\cham;q} \,:\, \lrp{\Xt_{\cham}^{\cv_i}}f\lrp{{\Xt_{\cham}}^{\cv_i}}^{-1}=q^{2 \lra{\cv_i,m}}f \;\mathrm{for}\; i \in I \right\}. }
(Note that the pairing $ \lra{\cv_i,m}$ takes values in $\frac{1}{d} \Z$.) 
Next, observe that $\mathcal{T}_{\cham,m;q}$ is an $R_q$-module with generating set $\left\{ \Xt_{\cham}^n \,:\, -p^*(n) = m \right\}$.
Indeed,  if $\Xt_{\cham}^n$ is such that $-p^*(n)=m$, then
\eqn{\lrp{\Xt_{\cham}^{\cv_i}}\Xt_{\cham}^n\lrp{\Xt_{\cham}^{\cv_i}}^{-1} 
&=q^{-2\lrc{n,\cv_i}}\Xt_{\cham}^n \Xt_{\cham}^{\cv_i} \lrp{\Xt_{\cham}^{\cv_i}}^{-1} \\
&=q^{-2\lra{\cv_{i},p^*(n)}}\Xt_{\cham}^n \\
&=q^{2\lra{\cv_{i},m}}\Xt_{\cham}^n.}
Now, let $f\in \mathcal{T}_{\cham,m;q}$. 
Since $f=\sum_n \lambda_{n}\Xt_{\cham}^n$ and by hypothesis
\eqn{q^{2\lra{\cv_i,m}}f
&=\Xt_{\cham}^{\cv_i} \lrp{\sum \lambda_{n}\Xt_{\cham}^n}\lrp{\Xt_{\cham}^{\cv_i}}^{-1} \\
&=\sum \lambda_{n}\Xt_{\cham}^{\cv_i}\Xt_{\cham}^n\lrp{\Xt_{\cham}^{\cv_i}}^{-1}\\
&=\sum \lambda_{n} q^{2\lra{\cv_i,-p^*(n)}}\Xt_{\cham}^n,
}
then it must hold that $\lra{\cv_i,-p^*(n)}=\lra{\cv_i,m}$ for each index in the sum. 
Consequently,  $\mathcal{T}_{\cham;q}$ is graded by $p^*(N)$: $\displaystyle{\mathcal{T}_{\cham;q}=\bigoplus_{m\in p^*(N)} \mathcal{T}_{\cham,m;q}}$. 
Using this grading, it is immediate that $\mathcal{Z}=\mathcal{T}_{\cham,0;q}$ and also that, if $I$ is an ideal of $\mathcal{Z}$, then the contraction in $\mathcal{Z}$ of the extension $I\mathcal{T}_{\cham;q}=\mathcal{T}_{\cham;q}I$ is equal to $I$. That is, extension followed by contraction is the identity on ideals of $\mathcal{Z}$.

Now, let $J$ be a two-sided ideal of $\mathcal{T}_{\cham;q}$. 
Clearly, $\lrp{J\cap\mathcal{Z}}\mathcal{T}_{\cham;q}\subseteq J$. 
Suppose $S:=J \setminus \lrp{J\cap\mathcal{Z}}\mathcal{T}_{\cham;q}$ is non-empty and let $f = \sum_{m\in p^*(N)} z_m \Xt^{\varphi(m)}_{\cham}$, with $\varphi$ any section of $p^*$, be an element of $S$ with the minimal number of non-zero coefficients $z_m \in \mathcal{Z}$.
As each $\Xt^{\varphi(m)}_{\cham}$ is a unit in $ \mathcal{T}_{\cham;q}$, we are free to assume the constant term $z_0$ is non-zero.
Say $f= z_0 + z_{m_1} \Xt^{\varphi(m_1)}_{\cham} + \cdots + z_{m_r} \Xt^{\varphi(m_r)}_{\cham}$.
Note that $r\geq 1$-- otherwise $f \in \lrp{J\cap\mathcal{Z}}\mathcal{T}_{\cham;q}$.
Now fix some $n\in N$ and consider the commutator
\eqn{\lrb{\Xt^{n}_{\cham}, f } &= \lrb{\Xt^{n}_{\cham},z_0} + \lrb{\Xt^{n}_{\cham},z_{m_1} \Xt^{\varphi(m_1)}_{\cham}} + \cdots + \lrb{\Xt^{n}_{\cham},z_{m_r} \Xt^{\varphi(m_r)}_{\cham}}\\
&= z_{m_1} \lrb{\Xt^{n}_{\cham}, \Xt^{\varphi(m_1)}_{\cham}} + \cdots + z_{m_r} \lrb{\Xt^{n}_{\cham}, \Xt^{\varphi(m_r)}_{\cham}}\\
&= \lrp{q^{\lrc{n,\varphi(m_1)}} - q^{-\lrc{n,\varphi(m_1)}}} z_{m_1} \Xt^{n+\varphi(m_1)}_{\cham} + \cdots +  \lrp{q^{\lrc{n,\varphi(m_r)}} - q^{-\lrc{n,\varphi(m_r)}}} z_{m_r} \Xt^{n+\varphi(m_r)}_{\cham} .}
Observe that for each $m_i$, we can choose $n$ such that $\lrc{n,\varphi(m_i)}$ is non-zero.
As $J$ is a two-sided ideal of $\mathcal{T}_{\cham;q}$ and $f \in J$, the commutator $f':=\lrb{\Xt^{n}_{\cham}, f }$ must be in $J$ also.
Since $f'$ has fewer non-zero coefficients than $f$, it must lie in $\lrp{J\cap \mathcal{Z}}\mathcal{T}_{\cham;q}$.
We can write
\eqn{f'= \sum_{m \in p^*(N)}w_{m} \Xt^{\varphi(m)}_{\cham} }
for some $w_{m} \in \lrp{J\cap \mathcal{Z}}$.
Then
\eqn{ \sum_{m \in p^*(N)}w_{m} \Xt^{\varphi(m)}_{\cham}  =  \lrp{q^{\lrc{n,\varphi(m_1)}} - q^{-\lrc{n,\varphi(m_1)}}} z_{m_1} \Xt^{n+\varphi(m_1)}_{\cham} + \cdots +  \lrp{q^{\lrc{n,\varphi(m_r)}} - q^{-\lrc{n,\varphi(m_r)}}} z_{m_r} \Xt^{n+\varphi(m_r)}_{\cham} }
and by linear independence of $\lrc{\Xt^{n}_{\cham}: n\in N}$, the non-zero summands on the left correspond precisely to non-zero summands on the right.
Let $w_{m}\Xt^{\varphi(m)}$ be some such non-zero summand, say $w_{m}\Xt^{\varphi(m)}_{\cham} = \lrp{q^{\lrc{n,\varphi(m_i)}} - q^{-\lrc{n,\varphi(m_i)}}} z_{m_i} \Xt^{n+\varphi(m_i)}_{\cham}$.
Then $\Xt^{\varphi(m)}_{\cham}$ and $\Xt^{n+\varphi(m_i)}_{\cham}$ differ by a unit in $\mathcal{Z}$, namely a factor of $\Xt^{\varphi(m) - \lrp{n+\varphi(m_i)}}_{\cham}$, and
\eqn{w_m = \lrp{q^{\lrc{n,\varphi(m_i)}} - q^{-\lrc{n,\varphi(m_i)}}} z_{m_i} \Xt^{\varphi(m) - \lrp{n+\varphi(m_i)}}_{\cham}.}
So, $w_{m}$ is the product of $z_{m_i}$ and a unit in $\mathcal{Z}$.
It follows that $z_{m_i}$ is itself in $\lrp{J\cap \mathcal{Z}}$ since $w_m$ is.
This holds for all $m_i$ so $f$ must be in $\lrp{J\cap \mathcal{Z}}\mathcal{T}_{\cham;q}$, contradicting the assumption.
So, $J = \lrp{J\cap\mathcal{Z} } \mathcal{T}_{\cham;q}$ and contraction followed by extension is the identity on two-sided ideals of $\mathcal{T}_{\cham;q}$. 

For the next claim, using the fact that the contraction and extension give the bijection between the sets of two-sided  ideals of $\mathcal{T}_{\cham;q}$ and the ideals of $\mathcal{Z}$, it is enough to prove the following statements:
\begin{enumerate}
    \item \label{it:non-com-to-com}If $P\in\Spec\lrp{\mathcal{T}_{\cham;q}}$, then $P\cap\mathcal{Z}\in\Spec\lrp{\mathcal{Z}}$, and
    \item \label{it:com-to-non-com}If $Q\in\Spec\lrp{\mathcal{Z}}$, then $Q\mathcal{T}_{\cham;q}\in\Spec\lrp{\mathcal{T}_{\cham;q}}$.
\end{enumerate}
For (\ref{it:non-com-to-com}), let $P\in\Spec\lrp{\mathcal{T}_{\cham;q}}$. Consider ideals $I_1$ and $I_2$ of $\mathcal{Z}$ such that $I_1I_2\subseteq P\cap\mathcal{Z}$. Then, using the fact that contraction followed by extension is the identity,
\eqn{
\lrp{I_1\mathcal{T}_{\cham;q}}\lrp{I_2\mathcal{T}_{\cham;q}}
=\lrp{I_1I_2}\mathcal{T}_{\cham;q}
\subseteq 
\lrp{P\cap\mathcal{Z}}\mathcal{T}_{\cham;q}
=P.
}
The hypothesis $P$ prime implies that $I_1\mathcal{T}_{\cham;q}\subseteq P$ or $I_2\mathcal{T}_{\cham;q}\subseteq P$. Hence, \eqn{
\lrp{I_1\mathcal{T}_{\cham;q}}\cap\mathcal{Z}\subseteq P\cap\mathcal{Z}
\;\;\;\mbox{or}\;\;\;
\lrp{I_2\mathcal{T}_{\cham;q}}\cap\mathcal{Z}\subseteq P\cap\mathcal{Z}
}
and since extension followed by contraction is the identity, we conclude that $I_1\subseteq P\cap\mathcal{Z}$ or $I_2\subseteq P\cap\mathcal{Z}$. So, $P\cap\mathcal{Z}$ is a prime ideal of $\mathcal{Z}$.

For (\ref{it:com-to-non-com}), let $Q\in\Spec\lrp{\mathcal{Z}}$. Consider two-sided ideals $J_1$ and $J_2$ of $\mathcal{T}_{\cham;q}$ such that $J_1J_2\subseteq Q\mathcal{T}_{\cham;q}$. Since extension followed by contraction is the identity,
\eqn{
\lrp{J_1\cap\mathcal{Z}} \lrp{J_2\cap\mathcal{Z}} 
\subseteq
\lrp{J_1J_2}\cap\mathcal{Z}
\subseteq
\lrp{Q\mathcal{T}_{\cham;q}}\cap\mathcal{Z}
=Q.
}
From the hypothesis $Q$ prime it follows that 
\eqn{
\lrp{J_1\cap\mathcal{Z}}\subseteq Q
\;\;\;\mbox{or}\;\;\;
\lrp{J_2\cap\mathcal{Z}}\subseteq Q.
}
Finally, since contraction followed by extension is the identity, we conclude that $J_1\subseteq Q\mathcal{T}_{\cham;q}$ or $J_2\subseteq Q\mathcal{T}_{\cham;q}$. Hence, $Q\mathcal{T}_{\cham;q}$ is a prime ideal of $\mathcal{T}_{\cham;q}$.
\end{proof}

\subsection{Quantum mutation of the \texorpdfstring{$\Xfsp$}{X} family}\label{sec:qmut}

We now describe explicitly the quantum mutation map with coefficients $\mu_{k,\vb{t};\cham}^q$.
It is analogous to  \cite[Section~3.3]{FG_cluster_ensembles} but includes coefficients. For simplicity, we denote $q_k:=q^{1/d_k}$.
First, we consider the quantum dilogarithm with coefficients $\vb{t}$. 
\eq{
\vb{\Psi}_{q_k,\vb{t}}\lrp{\Xt_{k;\cham}}:=
\vb{\Psi}_{q_k} \lrp{\frac{\vb{t}^{\lrb{\cv_{k;\cham}}_+}}{\vb{t}^{\lrb{-\cv_{k;\cham}}_+}} \Xt_{k;\cham}}
}{eq:tdilog}
We then define $\mu_{k,\vb{t};\cham}^\sharp$ to be the automorphism of $\sK\lrp{\A^{\lrm{I}}_{M;\cham,q}\lrp{R}}$ given by conjugation by $\vb{\Psi}_{q_k,\vb{t}}\lrp{\Xt_{k;\cham}}$:
\eq{\mu_{k,\vb{t};\cham}^\sharp \lrp{x} := \vb{\Psi}_{q_k} \lrp{\frac{\vb{t}^{\lrb{\cv_{k;\cham}}_+}}{\vb{t}^{\lrb{-\cv_{k;\cham}}_+}} \Xt_{k;\cham}} x\   \vb{\Psi}_{q_k} \lrp{\frac{\vb{t}^{\lrb{\cv_{k;\cham}}_+}}{\vb{t}^{\lrb{-\cv_{k;\cham}}_+}} \Xt_{k;\cham}}^{-1}. }{eq:musharpt}

Next, we define
\eqn{
\mu_{k,\vb{t};\cham}'\lrp{\Xt^v_{\cham'}} = \vb{t}^{-\lrc{v,e_{k;\cham}}d_k \lrb{-\cv_{k;\cham}}_+} \Xt^v_{\cham}.
} 
Up to a certain scale factor,
$\mu_{k,\vb{t};\cham}':\sK\lrp{\A^{\lrm{I}}_{M;\cham',q}\lrp{R}}\rightarrow \sK\lrp{\A^{\lrm{I}}_{M;\cham,q}\lrp{R}}$ just re-expresses the coordinates $\lrc{\Xt_{i;\cham'}= \Xt^{e_{i;\cham'}}:i \in I}$ obtained by exponentiating the basis elements of $\vb{s}'$ in terms of the coordinates $\lrc{\Xt_{i;\cham}= \Xt^{e_{i;\cham}}:i \in I}$ obtained by exponentiating the basis elements of $\vb{s}$.
In terms of the two sets of cluster coordinates, $\mu_{k,\vb{t};\cham}'$ has form:
\eq{
\mu_{k,\vb{t};\cham}'\lrp{\Xt_{i;\cham'}} = \begin{cases} \Xt_{i;\cham}^{-1} &\text{ if } i=k,\\[5pt]
\vb{t}^{-\epsilon_{ik} \lrb{-\cv_{k;\cham}}_+} q^{-\widehat{\epsilon}_{ik}[\epsilon_{ik}]_+} \Xt_{i;\cham}\Xt_{k;\cham}^{[\epsilon_{ik}]_+} &\text{ if } i\not =k.\end{cases}
}{eq:muprimet}

\begin{lemma}\thlabel{qmusharpAlt}
The automorphism $\mu_k^{\sharp}$ is given in cluster coordinates by
\eq{
\mu_{k,\vb{t};\cham}^{\sharp}\lrp{\Xt_{i;\cham}} = \begin{cases}
\displaystyle{\Xt_{i;\cham}
\prod_{\ell=1}^{|\epsilon_{ik}|} \lrp{1 + \frac{\vb{t}^{\lrb{\cv_{k;\cham}}_+}}{\vb{t}^{\lrb{-\cv_{k;\cham}}_+}} q_k^{2\ell-1} \Xt_{k;\cham}}}
&\text{ if } \epsilon_{ik}\leq 0,\\[10pt]
\displaystyle{\Xt_{i;\cham}
\lrp{\prod_{\ell=1}^{\epsilon_{ik}} \lrp{1 + \frac{\vb{t}^{\lrb{\cv_{k;\cham}}_+}}{\vb{t}^{\lrb{-\cv_{k;\cham}}_+}} q_k^{1-2\ell} \Xt_{k;\cham}}}^{-1}} 
&\text{ if } \epsilon_{ik}\geq 0.\end{cases}
}{eq:musharpcluster}
\end{lemma}

\begin{proof}
This is \cite[Lemma~3.4]{FG_cluster_ensembles} with $X_k$ replaced by $\displaystyle{\frac{\vb{t}^{\lrb{\cv_{k;\cham}}_+}}{\vb{t}^{\lrb{-\cv_{k;\cham}}_+}} \Xt_{k;\cham}}$.
We reproduce the argument here for the reader's convenience.
Recall that 
$q^{-\widehat{\epsilon}_{ki}}\Xt_{k;\cham}\Xt_{i;\cham} =q^{-\widehat{\epsilon}_{ik}}\Xt_{i;\cham}\Xt_{k;\cham}$, 
so 
\eqn{\Xt_{k;\cham}\Xt_{i;\cham} &= q^{-2\widehat{\epsilon}_{ik}}\Xt_{i;\cham}\Xt_{k;\cham}\\
&= q_k^{-2\epsilon_{ik}}\Xt_{i;\cham}\Xt_{k;\cham}.}
Then writing $\vb{\Psi}_{q_k} (x)$ as $\displaystyle{\sum_{\ell=0}^\infty c_{\ell} x^{\ell}}$, we have
\eqn{\vb{\Psi}_{q_k} \lrp{\frac{\vb{t}^{\lrb{\cv_{k;\cham}}_+}}{\vb{t}^{\lrb{-\cv_{k;\cham}}_+}} \Xt_{k;\cham}} \Xt_{i;\cham} 
&= \sum_{\ell=0}^\infty c_{\ell} \lrp{\frac{\vb{t}^{\lrb{\cv_{k;\cham}}_+}}{\vb{t}^{\lrb{-\cv_{k;\cham}}_+}} \Xt_{k;\cham}}^{\ell} \Xt_{i;\cham}\\
&= \sum_{\ell=0}^\infty c_{\ell}  \Xt_{i;\cham} \lrp{ q_k^{-2 \epsilon_{ik}} \frac{\vb{t}^{\lrb{\cv_{k;\cham}}_+}}{\vb{t}^{\lrb{-\cv_{k;\cham}}_+}} \Xt_{k;\cham}}^{\ell} \\
&= \Xt_{i;\cham} \vb{\Psi}_{q_k} \lrp{q_k^{-2 \epsilon_{ik}} \frac{\vb{t}^{\lrb{\cv_{k;\cham}}_+}}{\vb{t}^{\lrb{-\cv_{k;\cham}}_+}} \Xt_{k;\cham}}.}
So,
\eq{ \mu_{k,\vb{t};\cham}^{\sharp}\lrp{\Xt_{i;\cham}} 
&= \vb{\Psi}_{q_k} \lrp{\frac{\vb{t}^{\lrb{\cv_{k;\cham}}_+}}{\vb{t}^{\lrb{-\cv_{k;\cham}}_+}} \Xt_{k;\cham}} \Xt_{i;\cham}\,  \vb{\Psi}_{q_k} \lrp{\frac{\vb{t}^{\lrb{\cv_{k;\cham}}_+}}{\vb{t}^{\lrb{-\cv_{k;\cham}}_+}} \Xt_{k;\cham}}^{-1}\\
&= \Xt_{i;\cham}\, \vb{\Psi}_{q_k} \lrp{q_k^{-2 \epsilon_{ik}} \frac{\vb{t}^{\lrb{\cv_{k;\cham}}_+}}{\vb{t}^{\lrb{-\cv_{k;\cham}}_+}} \Xt_{k;\cham}} \vb{\Psi}_{q_k} \lrp{\frac{\vb{t}^{\lrb{\cv_{k;\cham}}_+}}{\vb{t}^{\lrb{-\cv_{k;\cham}}_+}} \Xt_{k;\cham}}^{-1} .}{eq:sharpqcom}
Next, we use the quantum dilogarithm difference relation \cite[Equation~52]{FG_cluster_ensembles}:
\eq{\vb{\Psi}_{q}\lrp{q^2 x} = \lrp{1+ q x}\vb{\Psi}_{q}(x), \qquad \vb{\Psi}_{q}\lrp{q^{-2} x} = \lrp{1+ q^{-1} x}^{-1}\vb{\Psi}_{q}(x) .}{eq:qdiff}
Combining \eqref{eq:qdiff} and \eqref{eq:sharpqcom} yields the claim.
\end{proof}

We now define the quantum $\cX$-mutation with coefficients $\mu_{k,\vb{t};\cham}^q: \sK\lrp{\A^{\lrm{I}}_{M;\cham',q}\lrp{R}} \to \sK\lrp{\A^{\lrm{I}}_{M;\cham,q}\lrp{R}} $  to be the composition $\mu_{k,\vb{t};\cham}^q:= \mu_{k,\vb{t};\cham}^\sharp \circ \mu_{k,\vb{t};\cham}'$.

\begin{prop}\thlabel{prop:qdmutation}
The quantum $\cX$-mutation with coefficients $\mu_{k,\vb{t};\cham}^q: \sK\lrp{\A^{\lrm{I}}_{M;\cham',q}\lrp{R}} \to \sK\lrp{\A^{\lrm{I}}_{M;\cham,q}\lrp{R}} $ is given in cluster coordinates by
\eq{
\mu_{k,\vb{t};\cham}^q \lrp{\Xt_{i;\cham'}} = \begin{cases} \Xt_{i;\cham}^{-1} &\text{ if } i=k,\\
& \\
\Xt_{i;\cham} \lrp{\displaystyle\prod_{\ell=1}^{|\epsilon_{ik}|} \lrp{\vb{t}^{[\sgn{\lrp{\epsilon_{ik}}}\vb{c}_{k;\cham}]_+}+\vb{t}^{[-\sgn{\lrp{\epsilon_{ik}}}\vb{c}_{k;\cham}]_+} q_k^{2\ell-1} \Xt_{k;\cham}^{-\sgn{\lrp{\epsilon_{ik}}}}}}^{-\sgn{\lrp{\epsilon_{ik}}}}  &\text{ if } i\not =k .\end{cases}
}{eq:qdmutation}
\end{prop}

\begin{proof}
We use \thref{qmusharpAlt} to explicitly compute the composition $\mu_{k,\vb{t};\cham}^q:= \mu_{k,\vb{t};\cham}^\sharp \circ \mu_{k,\vb{t};\cham}'$.
First, if $i=k$ the result is immediate.
Next, if $i\neq k$ but $\epsilon_{ik} = 0$, \eqref{eq:muprimet} reduces to 
$\displaystyle{\mu_{k,\vb{t};\cham}' \lrp{\Xt_{i;\cham'}} = \Xt_{i;\cham}}$ and the product in \eqref{eq:musharpcluster} is empty.
So in this case $\displaystyle{\mu_{k,\vb{t};\cham}^q\lrp{\Xt_{i;\cham'}}= \Xt_{i;\cham}}$ as claimed.
Two cases remain:

{\bf Case 1}: $\epsilon_{ik}<0$.
\begin{align*}
&\mu_{k,\vb{t};\cham}^{\sharp}\lrp{\vb{t}^{-\epsilon_{ik} \lrb{-\cv_{k;\cham}}_+} q^{-\widehat{\epsilon}_{ik}[\epsilon_{ik}]_+}  \Xt_{i;\cham}\Xt_{k;\cham}^{[\epsilon_{ik}]_+} }\\
=& \vb{t}^{-\epsilon_{ik} \lrb{-\cv_{k;\cham}}_+} \mu_{k,\vb{t};\cham}^{\sharp}\lrp{ \Xt_{i;\cham}}\\
=& \vb{t}^{\lrm{\epsilon_{ik}} \lrb{-\cv_{k;\cham}}_+} \Xt_{i;\cham}
\prod_{\ell=1}^{|\epsilon_{ik}|} \lrp{1 + \frac{\vb{t}^{\lrb{\cv_{k;\cham}}_+}}{\vb{t}^{\lrb{-\cv_{k;\cham}}_+}} q_k^{2\ell-1} \Xt_{k;\cham}}
\\
=& \Xt_{i;\cham}
\prod_{\ell=1}^{|\epsilon_{ik}|} \lrp{\vb{t}^{\lrb{-\cv_{k;\cham}}_+} + \vb{t}^{\lrb{\cv_{k;\cham}}_+} q_k^{2\ell-1} \Xt_{k;\cham}}.
\end{align*}

{\bf Case 2}: $\epsilon_{ik}>0$. Recall that by definition $\epsilon_{ik}=\{e_i,e_k\}d_k$, $q_k=q^{1/d_k}$ and $\widehat{\epsilon}_{ik}=\{e_i,e_k\}$.
\eqn{
&\mu_{k,\vb{t};\cham}^{\sharp}\lrp{\vb{t}^{-\epsilon_{ik} \lrb{-\cv_{k;\cham}}_+}  q^{-\widehat{\epsilon}_{ik}[\epsilon_{ik}]_+} \Xt_{i;\cham}\Xt_{k;\cham}^{[\epsilon_{ik}]_+} }\\
=& \vb{t}^{-\epsilon_{ik} \lrb{-\cv_{k;\cham}}_+} q^{-\widehat{\epsilon}_{ik}\, \epsilon_{ik}}  \mu_{k,\vb{t};\cham}^{\sharp}\lrp{ \Xt_{i;\cham}\Xt_{k;\cham}^{\epsilon_{ik}} }\\
=& \vb{t}^{-\epsilon_{ik} \lrb{-\cv_{k;\cham}}_+} q^{-\widehat{\epsilon}_{ik}\, \epsilon_{ik}}  
\Xt_{i;\cham}\Xt_{k;\cham}^{\epsilon_{ik}} 
\lrp{\prod_{\ell=1}^{\epsilon_{ik}} \lrp{1 + \frac{\vb{t}^{\lrb{\cv_{k;\cham}}_+}}{\vb{t}^{\lrb{-\cv_{k;\cham}}_+}} q_k^{1-2\ell} \Xt_{k;\cham}}}^{-1}\\
=& 
\Xt_{i;\cham}
\lrp{\prod_{\ell=1}^{\epsilon_{ik}} q^{\widehat{\epsilon}_{ik}}\Xt_{k;\cham}^{-1} \lrp{\vb{t}^{\lrb{-\cv_{k;\cham}}_+}  + \vb{t}^{\lrb{\cv_{k;\cham}}_+} q_k^{1-2\ell} \Xt_{k;\cham}}}^{-1}\\
=& \Xt_{i;\cham}
\lrp{\prod_{\ell=1}^{\epsilon_{ik}} q_k^{\epsilon_{ik}} q_k^{1-2\ell} \lrp{ \vb{t}^{\lrb{-\cv_{k;\cham}}_+} q_k^{2\ell-1} \Xt_{k;\cham}^{-1} + \vb{t}^{\lrb{\cv_{k;\cham}}_+} }}^{-1}\\
=& \Xt_{i;\cham}
\lrp{ q_k^{\epsilon_{ik}^2} q_k^{-\epsilon_{ik}^2} \prod_{\ell=1}^{\epsilon_{ik}}\lrp{ \vb{t}^{\lrb{-\cv_{k;\cham}}_+} q_k^{2\ell-1} \Xt_{k;\cham}^{-1} + \vb{t}^{\lrb{\cv_{k;\cham}}_+} }}^{-1}\\
=& \Xt_{i;\cham}
\lrp{\prod_{\ell=1}^{\epsilon_{ik}}\lrp{\vb{t}^{\lrb{\cv_{k;\cham}}_+} + \vb{t}^{\lrb{-\cv_{k;\cham}}_+} q_k^{2\ell-1} \Xt_{k;\cham}^{-1} }}^{-1}.
}
This proves the claim.
\end{proof}

\begin{cor}\thlabel{cor:recover}
 Setting $\vb{t}=\vb{1}$ we recover the coefficient free quantum mutation formula \eqref{eq:muq}. Setting $q=1$, we recover the degeneration formula \eqref{eq:Xfammu}. 
\end{cor}

\begin{prop}\thlabel{prop:*hom}
The quantum $\cX$-mutation with coefficients $\mu_{k;\vb{t};\cham}^q$ is a $*$-homomorphism.
\end{prop}
\begin{proof}
Since by construction $\mu_{k;\vb{t};\cham}^q=\mu_{k;\vb{t};\cham}^{\sharp}\circ\mu_{k;\vb{t};\cham}'$, it is enough to prove that the homomorphisms $\mu_{k;\vb{t};\cham}^{\sharp}$ and $\mu_{k;\vb{t};\cham}'$ are $*$-homomorphisms. 

Using the fact that the quantum dilogarithm satisfies the equality $\Psi_{q_k^{-1}}\lrp{x}=\Psi_{q_k}(x)^{-1}$ \cite[Equation~53]{FG_cluster_ensembles}, it follows that
\eqn{
*\lrp{\mu_{k;\vb{t};\cham}^{\sharp}\lrp{\Xt_{i;\cham}}}
&= \vb{\Psi}_{q_k^{-1}} \lrp{\frac{\vb{t}^{\lrb{\cv_{k;\cham}}_+}}{\vb{t}^{\lrb{-\cv_{k;\cham}}_+}} \Xt_{k;\cham}}^{-1} \Xt_{i;\cham}\,  \vb{\Psi}_{q_k^{-1}} \lrp{\frac{\vb{t}^{\lrb{\cv_{k;\cham}}_+}}{\vb{t}^{\lrb{-\cv_{k;\cham}}_+}} \Xt_{k;\cham}} \\
&=\vb{\Psi}_{q_k} \lrp{\frac{\vb{t}^{\lrb{\cv_{k;\cham}}_+}}{\vb{t}^{\lrb{-\cv_{k;\cham}}_+}} \Xt_{k;\cham}} \Xt_{i;\cham}\,  \vb{\Psi}_{q_k} \lrp{\frac{\vb{t}^{\lrb{\cv_{k;\cham}}_+}}{\vb{t}^{\lrb{-\cv_{k;\cham}}_+}} \Xt_{k;\cham}}^{-1}
=\mu_{k;\vb{t};\cham}^{\sharp}\lrp{\Xt_{i;\cham}}.
}
For $\mu_{k;\vb{t};\cham}'$, the statement is clear if $i=k$ or $i\neq k$ and $\epsilon_{ik}=0$. Now, assume that $i\neq k$ and $\epsilon_{ik}\neq 0$. We consider two cases:

{\bf Case 1:} $\epsilon_{ik}<0$. 
\eqn{
*\lrp{ \mu_{k;\vb{t};\cham}'\lrp{\Xt_{i;\cham}}}
= *\lrp{\vb{t}^{-\epsilon_{ik} \lrb{-\cv_{k;\cham}}_+}  \Xt_{i;\cham}}
= \vb{t}^{-\epsilon_{ik} \lrb{-\cv_{k;\cham}}_+}  \Xt_{i;\cham}
= \mu_{k;\vb{t};\cham}'\lrp{\Xt_{i;\cham}}.
}

{\bf Case 2:} $\epsilon_{ik}>0.$ Recall that $\Xt_{k;\cham}\Xt_{i;\cham}=q^{-2\widehat{\epsilon}_{ik}}\Xt_{i;\cham}\Xt_{k;\cham}$.
\eqn{
*\lrp{ \mu_{k;\vb{t};\cham}'\lrp{\Xt_{i;\cham}}}
&=*\lrp{\vb{t}^{-\epsilon_{ik} \lrb{-\cv_{k;\cham}}_+} q^{-\widehat{\epsilon}_{ik}\epsilon_{ik}} \Xt_{i;\cham}\Xt_{k;\cham}^{\epsilon_{ik}}} \\
&=\vb{t}^{-\epsilon_{ik} \lrb{-\cv_{k;\cham}}_+} q^{\widehat{\epsilon}_{ik}\epsilon_{ik}} \Xt_{k;\cham}^{\epsilon_{ik}}\Xt_{i;\cham} \\
&=\vb{t}^{-\epsilon_{ik} \lrb{-\cv_{k;\cham}}_+} q^{\widehat{\epsilon}_{ik}\epsilon_{ik}}q^{-2\widehat{\epsilon}_{ik}\epsilon_{ik}} \Xt_{i;\cham}\Xt_{k;\cham}^{\epsilon_{ik}} \\
&=\vb{t}^{-\epsilon_{ik} \lrb{-\cv_{k;\cham}}_+} q^{-\widehat{\epsilon}_{ik}\epsilon_{ik}} \Xt_{i;\cham}\Xt_{k;\cham}^{\epsilon_{ik}}\\
&=\mu_{k;\vb{t};\cham}'\lrp{\Xt_{i;\cham}}.
}
This completes the proof. 
\end{proof}

\begin{definition}\thlabel{def:quantum-family}
The scheme $\Xfsp_{\Sigma,q}$ over $R$ is obtained by gluing the affine patches $$\A^{\lrm{I}}_{M;\cham,q}\lrp{R}:= \Spec\lrp{R[q^{\pm \frac{1}{d}}]\lra{\Xt_{\cham}^n: n \in (\cham^\vee\cap N) } / \sim}$$ via the mutations given in \eqref{eq:qdmutation}, where $\cham$ varies over the maximal cones of $\Sigma$. Similarly, we obtain the scheme $\Xfam_{\Sigma,q}$ by gluing the affine patches $\Spec\lrp{R[q^{\pm \frac{1}{d}}]\lra{\Xt_{\cham}^{n}: n \in N } / \sim}$. 
When there is no risk of confusion, we will often drop the subscript $\Sigma$ from our notation, writing simply $\Xfsp_q$ or $\Xfam_q$. 
\end{definition}

\begin{remark}
In \cite{BFMNC}, the family $\Xfsp_{\Gamma,[\seed]}$ is viewed as a toric degeneration of the special completion $\Xsp_{\Gamma,[\seed]}$ of $\cX_{\Gamma,[\seed]}$ to the toric variety defined by the underlying fan of the cluster complex in the $\cA_{\Gamma^\vee,[\seed^\vee]}$ scattering diagram.
Note that here we obtain a quantum version of this toric degeneration, where characters only $q$-commute.
\end{remark}

Before proceeding, let us describe the gluing in \thref{def:quantum-family} in greater detail.
We will focus on $\Xfsp_{q}$; the case of $\Xfam_{q}$ involves only minor changes.
This discussion closely follows the analogous classical discussion in \cite{BFMNC}. 
Let $\cham$ and $\cham'$ be adjacent maximal cones of $\Sigma$, sharing a facet contained in $\cv_{k;\cham}^\perp = \cv_{k;\cham'}^\perp$. 
Define 
\eq{A_{\cham,q} := R[q^{\pm \frac{1}{d}}]\lra{\Xt_{k;\cham}^{\pm}, \Xt_{i\neq k;\cham}, \lrp{\displaystyle\prod_{\ell=1}^{|\epsilon_{ik}|} \lrp{\vb{t}^{[\sgn{\lrp{\epsilon_{ik}}}\vb{c}_{k;\cham}]_+}+\vb{t}^{[-\sgn{\lrp{\epsilon_{ik}}}\vb{c}_{k;\cham}]_+} q_k^{2\ell-1} \Xt_{k;\cham}^{-\sgn{\lrp{\epsilon_{ik}}}}}}^{-1}} / \sim  }{eq:glue-ring} 
and define $A_{\cham',q}$ analogously.
Observe that $\sgn(\epsilon_{ik}) \cv_{k;\cham} = \sgn(\epsilon_{ik}') \cv_{k;\cham'}$ and $\mu_{k,\vb{t};\cham}^q(\Xt_{k;\cham'}^{-\sgn(\epsilon_{ik}')}) = \Xt_{k;\cham}^{-\sgn(\epsilon_{ik})}$.
Armed with these identities, we apply $\mu_{k,\vb{t};\cham}^q$ to the generators of $A_{\cham',q}$ and find: 
\eq{\mu_{k,\vb{t};\cham}^q(\Xt_{k;\cham'}^{\pm}) = \Xt_{k;\cham}^{\mp},}{eq:muqXk}
\eq{\mu_{k,\vb{t};\cham}^q(\Xt_{i;\cham'}) =\Xt_{i;\cham} \lrp{\displaystyle\prod_{\ell=1}^{|\epsilon_{ik}|} \lrp{\vb{t}^{[\sgn{\lrp{\epsilon_{ik}}}\vb{c}_{k;\cham}]_+}+\vb{t}^{[-\sgn{\lrp{\epsilon_{ik}}}\vb{c}_{k;\cham}]_+} q_k^{2\ell-1} \Xt_{k;\cham}^{-\sgn{\lrp{\epsilon_{ik}}}}}}^{-\sgn{\lrp{\epsilon_{ik}}}} }{eq:muqXi}
for $i\neq k$, 
and
\eq{\mu_{k,\vb{t};\cham}^q\lrp{\displaystyle\prod_{\ell=1}^{|\epsilon_{ik}'|} \lrp{\vb{t}^{[\sgn{\lrp{\epsilon_{ik}'}}\vb{c}_{k;\cham'}]_+}+\vb{t}^{[-\sgn{\lrp{\epsilon_{ik}'}}\vb{c}_{k;\cham'}]_+} q_k^{2\ell-1} \Xt_{k;\cham'}^{-\sgn{\lrp{\epsilon_{ik}'}}}}}^{-1} \\
= \lrp{\displaystyle\prod_{\ell=1}^{|\epsilon_{ik}|} \lrp{\vb{t}^{[\sgn{\lrp{\epsilon_{ik}}}\vb{c}_{k;\cham}]_+}+\vb{t}^{[-\sgn{\lrp{\epsilon_{ik}}}\vb{c}_{k;\cham}]_+} q_k^{2\ell-1} \Xt_{k;\cham}^{-\sgn{\lrp{\epsilon_{ik}}}}}}^{-1}.}{eq:muqstuff}
The right hand sides of \eqref{eq:muqXk} and \eqref{eq:muqstuff} are explicitly within the generating set defining $A_{\cham,q}$.
Meanwhile, if $\epsilon_{ik}<0$ we may generate $\Xt_{i\neq k;\cham}$ using the right hand sides of \eqref{eq:muqXi} and \eqref{eq:muqstuff}.
On the other hand, if $\epsilon_{ik}>0$ the term within the outer parentheses on the right hand side of \eqref{eq:muqXi} is an element of $R[q^{\pm \frac{1}{d}}, \Xt_{k,\cham}^{\pm}]$, and we may generate $\Xt_{i\neq k;\cham}$ using the right hand sides of \eqref{eq:muqXi} and \eqref{eq:muqXk}.
Moreover, we see that applying $\mu_{k,\vb{t};\cham}^q$ to each generator of $A_{\cham',q}$ yields an element of $A_{\cham,q}$.
It follows that $\mu_{k,\vb{t};\cham}^q:  \sK\lrp{\A^{\lrm{I}}_{M;\cham',q}\lrp{R}} \to \sK\lrp{\A^{\lrm{I}}_{M;\cham,q}\lrp{R}}  $
restricts to a $*$-algebra isomorphism $\mu_{k,\vb{t};\cham}^q:  A_{\cham',q} \to A_{\cham,q} $.
The affine patches $\A^{\lrm{I}}_{M;\cham,q}\lrp{R}$ and $ \A^{\lrm{I}}_{M;\cham',q}\lrp{R}$ are glued via the induced isomorphism of their open subvarieties $\Spec\lrp{A_{\cham,q}} \stackrel{\sim}{\to} \Spec\lrp{A_{\cham',q}}$.

One final detail remains in order to verify that the gluing of \thref{def:quantum-family} makes sense.
We have described how adjacent affine patches glue, and to assure that all affine patches can be coherently glued as described we must verify that this gluing is independent of mutation sequence used as discussed in \thref{rem:cham}. To do so and following the argument in \cite{BFMNC}, we will prove that the periodicities in $\Xsp_{q}$ coincide with the ones in $\Xfsp_{q}$.
The independence of mutation sequence is known to hold in the former case.

Let $F$ be an element in $\mathbb{C}[q^{\pm \frac{1}{d}}]\lra{X_{\cham}^n: n \in N } / \sim$. In a natural way, we can consider the coefficients of $F$ as elements in $R[q^{\pm \frac{1}{d}}]$ and then, we have an induced element $F^R$ in $R[q^{\pm \frac{1}{d}}]\lra{X_{\cham}^n: n \in N } / \sim$. Using this extension, we are able to recover the quantum $\cX$-mutation with coefficients \eqref{eq:qdmutation}.

\begin{prop}
\thlabel{prop:q-periods}
With notation as above, the following equality holds
\eq{ \mu_{\vb{t};(\seed_0, \seed)}^q \lrp{\Xt_{i;\seed} }= 
\frac{\lrp{\mu_{(\seed_0, \seed)}^q \lrp{X_{i;\seed}}}^R \lrp{ t_1 X_{1;\seed_0}^R,\dots, t_{\lrm{I}} X_{\lrm{I};\seed_0}^R } }{\vb{t}^{\cv_{i;\seed}}}, }{eq:periodicity} 
where $\mu_{\vb{t};(\seed_0, \seed)}^q$ denotes the pullback from $\A^{\lrm{I}}_{M;\seed,q}\lrp{R}$ to $\A^{\lrm{I}}_{M;\seed_0,q}\lrp{R}$  along a given sequence of mutations and similarly $\mu_{(\seed_0, \seed)}^q $ denotes the corresponding pullback from $\A^{\lrm{I}}_{M;\seed,q}\lrp{\C}$ to $\A^{\lrm{I}}_{M;\seed_0,q}\lrp{\C}$.
As a consequence, we obtain the same periodicities for the cluster variables of $\Xfsp_{q}$ and $\Xsp_{q}$. That is,
\begin{center}
$\mu_{\vb{t};(\seed_0, \seed)}^q \lrp{\Xt_{i;\seed}}=\mu_{\vb{t};(\seed_0, \seed')}^q \lrp{\Xt_{j;\seed'}}$ if and only if $\mu_{(\seed_0, \seed)}^q \lrp{X_{i;\seed}}=\mu_{(\seed_0, \seed')}^q \lrp{X_{{j;\seed'}}}$. 
\end{center}
\end{prop}

\begin{proof}
First, if $\seed=\seed_0$, $\mu_{\vb{t};(\seed_0, \seed)}^q \lrp{\Xt_{i;\seed} } = \Xt_{i;\seed_0} = X_{i;\seed_0}^R$ and we have 
\eqn{
\frac{\lrp{\mu_{(\seed_0, \seed)}^q \lrp{X_{i;\seed}}}^R \lrp{ t_1 X_{1;\seed_0}^R,\dots, t_{\lrm{I}} X_{\lrm{I};\seed_0}^R } }{\vb{t}^{\cv_{i;\seed}}} =
\frac{t_i X_{i;\seed_0}^R }{t_i} = X_{i;\seed_0}^R .}
So the result holds for all variables of the initial cluster.
Assume for induction that it holds for all variables of $\seed$, and let $\seed'$ be a seed adjacent to $\seed$.
Let $\seed$ and $\seed'$ be related by mutation in direction $k$.
For $i=k$, we have \eq{\mu_{\vb{t};(\seed_0,\seed')}^q \lrp{\Xt_{k;\seed'}}  =\mu_{\vb{t};(\seed_0,\seed)}^q \circ \mu_{k,\vb{t};\seed}^q \lrp{\Xt_{k;\seed'}} = \mu_{\vb{t};(\seed_0,\seed)}^q\lrp{\Xt_{k;\seed}^{-1}}= \mu_{\vb{t};(\seed_0,\seed)}^q\lrp{\Xt_{k;\seed}}^{-1}}{eq:k-period-1} 
and similarly by \eqref{eq:muq}
\eq{
\frac{\lrp{\mu_{(\seed_0, \seed')}^q \lrp{X_{k;\seed'}}}^R \lrp{ t_1 X_{1;\seed_0}^R,\dots, t_{\lrm{I}} X_{\lrm{I};\seed_0}^R } }{\vb{t}^{\cv_{k;\seed'}}} &=
\frac{\lrp{\mu_{(\seed_0, \seed)}^q \circ\mu_{k;\seed}^q \lrp{X_{k;\seed'}}}^R \lrp{ t_1 X_{1;\seed_0}^R,\dots, t_{\lrm{I}} X_{\lrm{I};\seed_0}^R } }{\vb{t}^{\cv_{k;\seed'}}}\\ 
&=
\frac{\lrp{\mu_{(\seed_0, \seed)}^q  \lrp{X_{k;\seed}^{-1}} }^R \lrp{ t_1 X_{1;\seed_0}^R,\dots, t_{\lrm{I}} X_{\lrm{I};\seed_0}^R } }{\vb{t}^{-\cv_{k;\seed}}}\\
&= 
\lrp{\frac{\lrp{\mu_{(\seed_0, \seed)}^q  \lrp{X_{k;\seed}} }^R \lrp{ t_1 X_{1;\seed_0}^R,\dots, t_{\lrm{I}} X_{\lrm{I};\seed_0}^R }} {\vb{t}^{\cv_{k;\seed}}}}^{-1}
}{eq:k-period-2}
By the induction hypothesis, \eqref{eq:k-period-1} and \eqref{eq:k-period-2} agree, as desired.
Next, we use \eqref{eq:qdmutation} with $i\neq k$ to expand
$\mu_{\vb{t};(\seed_0,\seed')}^q \lrp{\Xt_{i;\seed'}}  =\mu_{\vb{t};(\seed_0,\seed)}^q \circ \mu_{k,\vb{t};\seed}^q \lrp{\Xt_{i;\seed'}} $:
\eq{ &\mu_{\vb{t};(\seed_0,\seed)}^q\lrp{\Xt_{i;\seed} \lrp{\displaystyle\prod_{\ell=1}^{|\epsilon_{ik}|} \lrp{\vb{t}^{[\sgn{\lrp{\epsilon_{ik}}}\vb{c}_{k;\seed}]_+}+\vb{t}^{[-\sgn{\lrp{\epsilon_{ik}}}\vb{c}_{k;\seed}]_+} q_k^{2\ell-1} \Xt_{k;\seed}^{-\sgn{\lrp{\epsilon_{ik}}}}}}^{-\sgn{\lrp{\epsilon_{ik}}}}}\\
=& 
\mu_{\vb{t};(\seed_0,\seed)}^q\lrp{ \vb{t}^{-\epsilon_{ik} [\sgn{\lrp{\epsilon_{ik}}}\vb{c}_{k;\seed}]_+}  \Xt_{i;\seed} \lrp{\displaystyle\prod_{\ell=1}^{|\epsilon_{ik}|} \lrp{1+\vb{t}^{-\sgn{\lrp{\epsilon_{ik}}}\vb{c}_{k;\seed}} q_k^{2\ell-1} \Xt_{k;\seed}^{-\sgn{\lrp{\epsilon_{ik}}}}}}^{-\sgn{\lrp{\epsilon_{ik}}}}}\\
=& 
\mu_{\vb{t};(\seed_0,\seed)}^q\lrp{ \vb{t}^{\cv_{i;\seed}-\cv_{i;\seed'}}  \Xt_{i;\seed} \lrp{\displaystyle\prod_{\ell=1}^{|\epsilon_{ik}|} \lrp{1+ q_k^{2\ell-1} \lrp{\vb{t}^{\vb{c}_{k;\seed}} \Xt_{k;\seed}}^{-\sgn{\lrp{\epsilon_{ik}}}}}}^{-\sgn{\lrp{\epsilon_{ik}}}}}\\
=& 
\vb{t}^{\cv_{i;\seed}-\cv_{i;\seed'}} \mu_{\vb{t};(\seed_0,\seed)}^q\lrp{  \Xt_{i;\seed} }
\lrp{\displaystyle\prod_{\ell=1}^{|\epsilon_{ik}|} \lrp{1+ q_k^{2\ell-1} \lrp{\vb{t}^{\vb{c}_{k;\seed}}  \mu_{\vb{t};(\seed_0,\seed)}^q\lrp{\Xt_{k;\seed}}}^{-\sgn{\lrp{\epsilon_{ik}}}}}}^{-\sgn{\lrp{\epsilon_{ik}}}}
}{eq:i-period-1} 
By the induction hypothesis, we can rewrite \eqref{eq:i-period-1} as 
\eqn{&\vb{t}^{-\cv_{i;\seed'}}\lrp{\mu_{(\seed_0,\seed)}^q\lrp{  X_{i;\seed} }}^R\lrp{t_1 X_{1;\seed_0}^R, \dots, t_{\lrm{I}} X_{\lrm{I};\seed_0}^R}\\
\times & \lrp{\displaystyle\prod_{\ell=1}^{|\epsilon_{ik}|} \lrp{1+ q_k^{2\ell-1} \lrp{\mu_{(\seed_0,\seed)}^q\lrp{X_{k;\seed}}}^R\lrp{t_1 X_{1;\seed_0}^R, \dots, t_{\lrm{I}} X_{\lrm{I};\seed_0}^R}^{-\sgn{\lrp{\epsilon_{ik}}}}}}^{-\sgn{\lrp{\epsilon_{ik}}}} 
}
which is precisely
\eqn{\frac{\lrp{\mu_{(\seed_0,\seed)}^q\circ \mu_{k;\seed}^q \lrp{  X_{i;\seed'} }}^R\lrp{t_1 X_{1;\seed_0}^R, \dots, t_{\lrm{I}} X_{\lrm{I};\seed_0}^R}}{\vb{t}^{\cv_{i;\seed'}}} = \frac{\lrp{\mu_{(\seed_0,\seed')}^q\lrp{  X_{i;\seed'} }}^R\lrp{t_1 X_{1;\seed_0}^R, \dots, t_{\lrm{I}} X_{\lrm{I};\seed_0}^R}}{\vb{t}^{\cv_{i;\seed'}}}
}
as claimed.

Next, we prove that the periodicities coincide. If $\mu_{\vb{t};(\seed_0, \seed)}^q \lrp{\Xt_{i;\seed}}=\mu_{\vb{t};(\seed_0, \seed')}^q \lrp{\Xt_{j;\seed'}}$, then taking 
the limit $\vb{t} \rightarrow 1$
we obtain that $\mu_{(\seed_0, \seed)}^q \lrp{X_{i;\seed}}=\mu_{(\seed_0, \seed')}^q \lrp{X_{j;\seed'}}$. Now, assume that $\mu_{(\seed_0, \seed)}^q \lrp{X_{i;\seed}}=\mu_{(\seed_0, \seed')}^q \lrp{X_{j;\seed'}}$. By \cite[Proposition~3.4]{KN11} we have that the periodicities in $\cX_q$ coincide with the ones in $\cX$ and  \cite[Theorem~2.34]{BFMNC} implies that $\cX$ shares the same periodicities with $\Xfam$. In particular, from the proof of \cite[Theorem~2.34]{BFMNC} we have that $\vb{t}^{\cv_{i;\seed}}=\vb{t}^{\cv_{j;\seed'}}$. 
Using \eqref{eq:periodicity}, the equalities $\mu_{(\seed_0, \seed)}^q \lrp{X_{i;\seed}}=\mu_{(\seed_0, \seed')}^q \lrp{X_{j;\seed'}}$ and $\vb{t}^{\cv_{i;\seed}}=\vb{t}^{\cv_{j;\seed'}}$ imply that $\mu_{\vb{t};(\seed_0, \seed)}^q \lrp{\Xt_{i;\seed}}=\mu_{\vb{t};(\seed_0, \seed')}^q \lrp{\Xt_{j;\seed'}}$.
\end{proof}

\begin{cor}\thlabel{cor:cocycle}
The $\cX$-clusters for $\Xfam_{q}$ correspond to maximal cones $\cham$ of $\Sigma$, with each variable indexed by a generator of the dual cone.
The pullback of each cluster variable to the initial affine patch is independent of the chosen mutation sequence.
That is, $\mu_{\vb{t};(\seed_0, \seed)}^q \lrp{\Xt_{i;\seed}}=\mu_{\vb{t};(\seed_0, \seed')}^q \lrp{\Xt_{j;\seed'}}$ for $\Xfam_{q}$ if and only if $\seed$ and $\seed'$ are seeds corresponding to the same maximal cone $\cham$ of $\Sigma$ and $\cv_{i;\seed} = \cv_{j;\seed'} $.
\end{cor}

\begin{proof}
We have just shown in \thref{prop:q-periods} that $\mu_{\vb{t};(\seed_0, \seed)}^q \lrp{\Xt_{i;\seed}}=\mu_{\vb{t};(\seed_0, \seed')}^q \lrp{\Xt_{j;\seed'}}$ for $\Xfam_{q}$ if and only if $\mu_{(\seed_0, \seed)}^q \lrp{X_{i;\seed}}=\mu_{(\seed_0, \seed')}^q \lrp{X_{{j;\seed'}}}$ for $\cX_{q}$. 
But by \cite[Proposition~3.4]{KN11}, $\mu_{(\seed_0, \seed)}^q \lrp{X_{i;\seed}}=\mu_{(\seed_0, \seed')}^q \lrp{X_{{j;\seed'}}}$ for $\cX_{q}$ if and only if the classical version $\mu_{(\seed_0, \seed)}^* \lrp{X_{i;\seed}}=\mu_{(\seed_0, \seed')}^* \lrp{X_{{j;\seed'}}}$ holds for $\cX$.
By \cite[Theorem 5.9]{Nak21}, $\mu_{(\seed_0, \seed)}^* \lrp{X_{i;\seed}}=\mu_{(\seed_0, \seed')}^* \lrp{X_{{j;\seed'}}}$ for $\cX$ if and only if $\seed$ and $\seed'$ are seeds corresponding to the same maximal cone $\cham$ of $\Sigma$ and $\cv_{i;\seed} = \cv_{j;\seed'} $. So, $\mu_{\vb{t};(\seed_0, \seed)}^q \lrp{\Xt_{i;\seed}}=\mu_{\vb{t};(\seed_0, \seed')}^q \lrp{\Xt_{j;\seed'}}$ for $\Xfam_{q}$ if and only if $\seed$ and $\seed'$ are seeds corresponding to the same maximal cone $\cham$ of $\Sigma$ and $\cv_{i;\seed} = \cv_{j;\seed'} $.
\end{proof}

As alluded to in \thref{rem:cham}, we are now justified in indexing $\cX$-clusters by maximal cones $\cham$ of $\Sigma$ rather than by seeds $\seed$ mutation equivalent to $\seed_0$.

\begin{example} In this example, we compare the different cluster varieties associated to the $A_2$ quiver. We illustrated the scattering diagram for this example in Figure~\ref{fig:A2Scat}.  This may be compared to \cite[Table~2]{BFMNC} and \cite[Tables~1-4]{FZ_clustersIV}, where to make the comparison set $\epsilon_\seed=B_\seed^T$. 
In Table~\ref{tab:a2_type}, we compare the usual $\cX$-variety with the corresponding quantum variety $\cX_q$ obtained using the quantum mutation formula in \cite{FG_cluster_ensembles} and in Table~\ref{tab:a2_deg_quant}, we compare the family $\Xfsp$ with the quantum family $\Xfsp_q$ using the mutation formula \eqref{eq:qdmutation}. Here, we index with seeds in order to have a direct comparison with \cite[Table~2]{BFMNC} and \cite[Tables~1-4]{FZ_clustersIV}.  \\[1ex]

\captionsetup{type=table}
\renewcommand{\arraystretch}{1.5}
\begin{minipage}{\textwidth}
\begin{center}
\begin{tabularx}{.97\linewidth}{|c|c|c|c|c|c|}
\cline{1-6}
& &  \multicolumn{2}{|c|}{$\cX$} & \multicolumn{2}{|c|}{$\cX_q$}  \\
\cline{1-6}
$\seed$ & $\epsilon_\seed$ &  $X_{1;\seed}$ & $X_{2;\seed}$ & $X_{1;\seed}$ & $X_{2;\seed}$  \\
\cline{1-6}
    0 & 
    $\left(\begin{smallmatrix} 0 &-1 \\ 1 &0 \end{smallmatrix}\right)$ &
    $X_1$ &
    $X_2$ &
    $X_1$ &
    $X_2$ \\
\cline{1-6}
$\updownarrow \mu_2$ \\
\cline{1-6}
    1 & 
    $\left(\begin{smallmatrix} 0 &1 \\ -1 &0 \end{smallmatrix}\right)$ &
    $X_1\lrp{1+X_2}$&
    $\frac{1}{X_2}$&
    $X_1\lrp{1+qX_2}$&
    $X_2^{-1}$\\
\cline{1-6} 
$\updownarrow \mu_1$ \\
\cline{1-6}
    2 & 
    $\left(\begin{smallmatrix} 0 &-1 \\ 1 &0 \end{smallmatrix}\right)$ &
    $\frac{1}{X_1\lrp{1+X_2}}$&
    $\frac{X_1X_2+X_1+1}{X_2}$&
    $\lrp{1+qX_2}^{-1}X_1^{-1}$&
    $\lrp{X_1^{-1}X_2}^{-1}\lrp{X_1^{-1}+q\lrp{1+qX_2}}$   \\
    \cline{1-6} 
$\updownarrow \mu_2$ \\
\cline{1-6}
    3 & 
    $\left(\begin{smallmatrix} 0 &1 \\ -1 &0 \end{smallmatrix}\right)$ &
    $\frac{X_1+1}{X_1X_2}$&
    $\frac{X_2}{X_1X_2+X_1+1}$&
    $X_2^{-1}\lrp{1+q^{-1}X_1^{-1}}$&
    $\lrp{X_1^{-1}+q\lrp{1+qX_2}}^{-1}\lrp{X_1^{-1}X_2}$\\
\cline{1-6} 
$\updownarrow \mu_1$ \\
\cline{1-6}
    4 & 
    $\left(\begin{smallmatrix} 0 &-1 \\ 1 &0 \end{smallmatrix}\right)$ &
    $\frac{X_1X_2}{X_1+1}$&
    $\frac{1}{X_1}$&
    $\lrp{1+q^{-1}X_1^{-1}}^{-1}X_2$&
    $X_1^{-1}$\\    
\cline{1-6} 
$\updownarrow \mu_2$ \\
\cline{1-6}
    5 & 
    $\left(\begin{smallmatrix} 0 &1 \\ -1 &0 \end{smallmatrix}\right)$ &
    $X_2$&
    $X_1$&
    $X_2$&
    $X_1$\\    
\cline{1-6}    
\end{tabularx}
\end{center}
\end{minipage}
\captionof{table}{\label{tab:a2_type} Comparison between the usual cluster variety $\cX$ and the quantum cluster variety $\cX_q$ in type $A_2$.}

\newpage
\begin{landscape}
\captionsetup{type=table}
\renewcommand{\arraystretch}{1.5}
\begin{minipage}[c]{\textwidth}
\begin{center}
\begin{tabularx}{.8\linewidth}{|c|c|c|c|c|c|c|}
\cline{1-7}
& & & \multicolumn{2}{|c|}{$\Xfsp$} & \multicolumn{2}{|c|}{$\Xfsp_q$} \\
\cline{1-7}
$\seed$ & $\epsilon_\seed$ & $\mathbf{C}_\seed$ & $\Xt_{1;\seed}$ & $\Xt_{2;\seed}$ & $\Xt_{1;\seed}$ & $\Xt_{2;\seed}$  \\
\cline{1-7}
    0 & 
    $\left(\begin{smallmatrix} 0 &-1 \\ 1 &0 \end{smallmatrix}\right)$ &
    $\left(\begin{smallmatrix} 1 &0 \\ 0 &1 \end{smallmatrix}\right)$ &
    $X_1$ &
    $X_2$ &
    $X_1$ &
    $X_2$ \\
\cline{1-7}
$\updownarrow \mu_2$ \\
\cline{1-7}
    1 & 
    $\left(\begin{smallmatrix} 0 &1 \\ -1 &0 \end{smallmatrix}\right)$ &
    $\left(\begin{smallmatrix} 1 &0 \\ 0 &-1 \end{smallmatrix}\right)$ &
    $X_1\lrp{1+t_2X_2}$&
    $\frac{1}{X_2}$&
    $X_1\lrp{1+t_2qX_2}$& 
    $X_2^{-1}$\\
\cline{1-7} 
$\updownarrow \mu_1$ \\
\cline{1-7}
    2 & 
    $\left(\begin{smallmatrix} 0 &-1 \\ 1 &0 \end{smallmatrix}\right)$ &
    $\left(\begin{smallmatrix} -1 &0 \\ 0 &-1 \end{smallmatrix}\right)$ & 
    $\frac{1}{X_1\lrp{1+t_2X_2}}$&
    $\frac{t_1t_2X_1X_2+t_1X_1+1}{X_2}$&
    $\lrp{1+t_2qX_2}^{-1}X_1^{-1}$&
    $\lrp{X_1^{-1}X_2}^{-1}\lrp{X_1^{-1}+t_1q\lrp{1+t_2qX_2}}$    \\
    \cline{1-7} 
$\updownarrow \mu_2$ \\
\cline{1-7}
    3 & 
    $\left(\begin{smallmatrix} 0 &1 \\ -1 &0 \end{smallmatrix}\right)$ &
    $\left(\begin{smallmatrix} -1 &0 \\ -1 &-1 \end{smallmatrix}\right)$& 
    $\frac{t_1X_1+1}{X_1X_2}$&
    $\frac{X_2}{t_1t_2X_1X_2+t_1X_1+1}$&
    $X_2^{-1}\lrp{t_1+q^{-1}X_1^{-1}}$&
    $(X_1^{-1}+t_1q(1+t_2qX_2))^{-1}(X_1^{-1}X_2)$ \\
\cline{1-7} 
$\updownarrow \mu_1$ \\
\cline{1-7}
    4 & 
    $\left(\begin{smallmatrix} 0 &-1 \\ 1 &0 \end{smallmatrix}\right)$ &
    $\left(\begin{smallmatrix} 1 &-1 \\ 1 &0 \end{smallmatrix}\right)$ &
    $\frac{X_1X_2}{t_1X_1+1}$&
    $\frac{1}{X_1}$&
    $\lrp{t_1+q^{-1}X_1^{-1}}^{-1}X_2$&
    $X_1^{-1}$\\    
\cline{1-7} 
$\updownarrow \mu_2$ \\
\cline{1-7}
    5 & 
    $\left(\begin{smallmatrix} 0 &1 \\ -1 &0 \end{smallmatrix}\right)$ &
    $\left(\begin{smallmatrix} 0 &1 \\ 1 &0 \end{smallmatrix}\right)$ &
    $X_2$&
    $X_1$&
    $X_2$&
    $X_1$\\    
\cline{1-7}    
\end{tabularx}
\end{center}
\end{minipage}
\captionof{table}{\label{tab:a2_deg_quant}Comparison between the family $\Xfsp$ and the quantum family $\Xfsp_q$ in type $A_2$. \\Here  $\mathbf{C}_\seed$ is the matrix of $\cv$-vectors for the cluster.}
\end{landscape}

\end{example}
\newpage


\section{Relation to Berenstein-Zelevinsky quantization of \texorpdfstring{$\cAp$}{Aprin}}\label{sec:q-pstar} \label{sec:relate}

In this section we relate quantum $\cX$-mutation with coefficients (Equation \eqref{eq:qdmutation}) with Berenstein-Zelevinsky's quantum $\cA$-mutation and Fomin-Zelevinsky's $\cA$-mutation with coefficients.
We are particularly interested in the case of principal coefficients, so this will be our focus.\footnote{We will however allow arbitrary frozen variables.  We discuss the distinction we make between coefficients and frozen variables later in this section.}

The first task is simply to translate between the Berenstein-Zelevinsky framework and the Fock-Goncharov framework (employed in this paper).
Let us briefly recall some notions from \cite{BZquantum} which we reviewed in greater detail in Section \ref{sec:quantumA}. 
Let $\lrp{ \Lambda, \Bt}$ be a compatible pair, i.e. 
$
\Bt^T \Lambda = 
\begin{pmatrix}
    D' & 0
\end{pmatrix},
$
where
$D'$ is the skew-symmetrizing matrix.
Then $D'B$ is skew-symmetric.
Note that all matrices in this setup are integral.

First note that $\epsilon_{\ubI} = \Bt^T$ when comparing \cite{FG_cluster_ensembles} to \cite{BZquantum}.
Rather than multiplying $\epsilon$ by a diagonal matrix to obtain an integral skew symmetric matrix, Fock and Goncharov multiply a rational skew symmetric matrix $\epshat$ by a diagonal integral matrix $D$  to obtain $\epsilon$:
\eqn{
\epsilon = \epshat\, D.
}
Then
\eqn{
\epsilon_\ubu = \epshat_\ubu D_\ubu.
}
We have that
\eqn{
\lrp{D'B}^T 
=& \epsilon_\ubu D'\\
=& \epshat_\ubu D_\ubu D' 
}
is skew symmetric,
and we can relate the two setups by taking $D'$ to be $\alpha D^{-1}_\ubu$ for some $\alpha \in \Z$ with $\alpha\, \epshat_\ubu$ integral. 
Since $\epsilon$ is integral and $\epsilon = \epshat\, D$, it follows that $d\, \epshat$ is integral as well, where $d = \lcm\lrp{d_i: i \in I}$.
We can take $D'= d\, D_\ubu^{-1}$-- i.e. we can take $D'$ to be the unfrozen part of the Langlands dual of $D$ to give a dictionary between the Berenstein-Zelevinsky and Fock-Goncharov prescriptions.\footnote{Note that this particular appearance of Langlands duality is simply to translate between differing conventions. 
The $\cA$- and $\cX$-varieties discussed in this section are associated to {\emph{the same}} fixed data $\Gamma$, rather than Langlands dual fixed data $\Gamma$ and $\Gamma^\vee$. It is important to keep this in mind in \thref{prop:quantum-pstar}.}
To be clear, $D'$ does not need to be given in this way to have a compatible pair and corresponding Berenstein-Zelevinsky quantization of the $\cA$ variety.
However, choosing a compatible pair with $D'= d\, D_\ubu^{-1}$ will give a Berenstein-Zelevinsky quantization of the $\cA$ variety that is closely related to the Fock-Goncharov quantization of the $\cX$ variety, as illustrated in \thref{prop:quantum-pstar}.

Before relating quantum $\cX$-mutation with coefficients to quantum $\cA$-mutation with coefficients, we need to clarify a potential point of confusion.
As in \cite{BFMNC}:

\begin{center}
{\emph{We distinguish between frozen variables and coefficients.}}    
\end{center}

Geometrically, a coefficient is a parameter on the base of a family while a frozen variable is a coordinate on the fibers of the family.
Coefficients are in the base ring while frozen variables are in an algebra over the base ring.
Classically, $\cA$-mutation treats frozen variables and coefficients in exactly the same way, so these concepts are generally identified in the literature.
However, $\cX$-mutation is genuinely affected the distinction, and in turn so is the notion of a dual cluster variety.
For example, if we have a cluster variety $\cA$ of dimension $n$ over $\C$,
then $\cAp$ can be viewed as either a $2n$-dimensional $\cA$-cluster variety over $\C$ or an
$n$-dimensional $\cA$-cluster variety over $\C\lrb{t_1,\dots,t_n}$.
In the former case, the dual is a $2n$-dimensional $\cX$-cluster variety over $\C$  ($\cX_{\mathrm{prin}}$ in \cite{GHKK}), while in the latter case the dual is an $n$-dimensional $\cX$-cluster variety over $\C\lrb{t_1,\dots,t_n}$ ($\Xfam$ in \cite{BFMNC}).
The mutation formulas for $\cX_{\mathrm{prin}}$ and $\Xfam$ differ.

In the quantum setting, the distinction can also affect $\cA$-mutation.
The reason deals with the relation between lattices and cluster variables.
Classically, it is convenient to describe mutation with coefficients by extending the lattice and skew form to incorporate the coefficients.
On the other hand, we could just as well keep the original lattice and skew form when adding coefficients and still obtain the same mutation formula without referencing any extended lattice, which is essentially how mutation with coefficients in a semifield was originally framed in \cite{FZ_clustersI}, \cite{FZ_clustersIV}.
By contrast, in the quantum setting existence of a compatible pair implies that $\epsilon_\ubI$ is full-rank. 
If the submatrix $\epsilon_\ubu$ is {\emph{not}} full-rank, then there will be some frozen $\cA$-variable that does not commute with a mutable $\cA$-variable-- $\Lambda(f_i,f_j) \neq 0$ for some mutable index $i$ and frozen index $j$.
So, if we extend the lattice to incorporate coefficients {\emph{as frozen variables}}, we can obtain coefficients that do not commute with variables. 
If, on the other hand, we keep the original lattice and distinguish between frozen variables and coefficients, by construction the coefficients will always commute with cluster variables while frozen variables may not.
The rank of the lattice here is the number of variables (both mutable and frozen) in each cluster.

With this in mind, we modify the quantum mutation formula of \cite{BZquantum} to include coefficients in a semifield $\PP$ as in \cite{FZ_clustersI}, \cite{FZ_clustersIV}, without altering the lattice and skew form of the coefficient-free case.
As usual, if $i\neq k$, $\mu_{k,\vb{t};\cham}^q \lrp{\At_{i;\cham'}} = \At_{i;\cham}$.
For $i=k$, we combine \cite[Proposition~4.9]{BZquantum} and  \cite[Equation~2.8]{FZ_clustersIV} to set
\eq{
\mu_{k,\vb{p};\cham}^q \lrp{\At_{i;\cham'}} = p_k^+ \At^{-f_{k;\cham} + \sum_{j: \epsilon_{kj}>0 } \epsilon_{kj} f_{j;\cham} } + p_k^- \At^{-f_{k;\cham} - \sum_{j: \epsilon_{kj}<0 } \epsilon_{kj} f_{j;\cham} }. 
}{eq:Aqmutsemifield}
In the special case of principal coefficients at $\cham$, $p_k^+ = \vb{t}^{\lrb{\cv_{k;\cham}}_+}$, $p_k^- = \vb{t}^{\lrb{-\cv_{k;\cham}}_+}$, and \eqref{eq:Aqmutsemifield} becomes
\eq{
\mu_{k,\vb{t};\cham}^q \lrp{\At_{i;\cham'}} = \vb{t}^{\lrb{\cv_{k;\cham}}_+} \At^{-f_{k;\cham} + \sum_{j: \epsilon_{kj}>0 } \epsilon_{kj} f_{j;\cham} } + \vb{t}^{\lrb{-\cv_{k;\cham}}_+} \At^{-f_{k;\cham} - \sum_{j: \epsilon_{kj}<0 } \epsilon_{kj} f_{j;\cham} }. 
}{eq:Aqmutprin}
We denote the corresponding Berenstein-Zelevinsky quantum $\cA$-variety with principal coefficients by $\mathscr{A}_{\prin, q}$, while writing $\cA_{\prin, q}$ for the quantization of $\cAp$ viewed as a scheme over $\C$.  

Next, recall the lattice maps $p^*:N\to M^\circ $ reviewed in Section~\ref{sec:cluster-scat}.
We show that the two quantizations with coefficients are compatible in the following sense:
\begin{prop}\thlabel{prop:quantum-pstar}
Assume $\mathscr{A}_{\prin, q}$ exists and that the data defining $\mathscr{A}_{\prin, q}$ and $\Xfam_q$ are related as described in the beginning of this section.
After making the identification\footnote{
This identification can be incorporated into the definition of the $*$-algebra homomorphism, writing ${p^*\lrp{q^{\frac{1}{d}}}= q^{-\frac{1}{2}}}$.} $\qfg^{\frac{1}{d}}= \qbz^{-\frac{1}{2}} $,
a $p^*$-map on the level of lattices induces a $*$-algebra homomorphism from the quantum $\cX$-torus algebra over $R_{\qfg}:= \C\lrb{t_i: i\in I}\lrb{\qfg^{\pm \frac{1}{d}}}$ to the quantum $\cA$-torus algebra over $R_{\qbz}:= \C\lrb{t_i: i\in I}\lrb{\qbz^{\pm \frac{1}{2}}}$  that commutes with mutation: 
\eqn{p^*\lrp{\mu_{k,\vb{t};\cham}^q \lrp{\Xt_{i;\cham'}}}=\mu_{k,\vb{t};\cham}^q \lrp{p^*\lrp{\Xt_{i;\cham'}}}.}
That is, there is a map of quantum cluster varieties $p:\mathscr{A}_{\mathrm{prin},q} \to \Xfam_q$ from the Berenstein-Zelevinsky quantum $\cA$-variety with principal coefficients to the Fock-Goncharov quantum $\cX$-variety with principal coefficients in the sense of \cite{BFMNC}.
\end{prop}

\begin{proof}
First, the identification of quantum parameters $\qfg^{\frac{1}{d}}= \qbz^{-\frac{1}{2}}$ comes from the following computation.
In the quantum $\cX$-torus algebra
\eqn{
 \Xt^{e_{i;\cham}} \Xt^{e_{j;\cham}}  =\qfg^{\epsilon_{ij}d_j^{-1}} \Xt^{e_{i;\cham}+e_{j;\cham}},
}
and in the quantum $\cA$-torus algebra
\eqn{
\At^{p^*(e_{i;\cham})}\At^{p^*(e_{j;\cham})} &= \qbz^{\frac{1}{2} \Lambda\lrp{p^*(e_{i;\cham}), p^*(e_{j  ;\cham})}} \At^{p^*(e_{i;\cham}+e_{j;\cham})}\\
&=\qbz^{-\frac{1}{2}\epsilon_{ij}d_j^\vee} \At^{p^*(e_{i;\cham}+e_{j;\cham})}.
}
In light of this, we want to identify $\qfg^{\frac{1}{d_j}}$ with $\qbz^{-\frac{1}{2}d_j^\vee}$ for all $j$, which is accomplished by setting $\qfg^{\frac{1}{d}}= \qbz^{-\frac{1}{2}}$.

Now if $i=k$, 
\eqn{p^*\lrp{\mu_{k,\vb{t};\cham}^q \lrp{\Xt_{i;\cham'}}} = p^*\lrp{\Xt_{k;\cham}^{-1}} = \At^{p^*(-e_{k;\cham})} .}
Similarly, 
\eqn{\mu_{k,\vb{t};\cham}^q \lrp{p^*\lrp{\Xt_{i;\cham'}}} = \mu_{k,\vb{t};\cham}^q \lrp{\At^{p^*(e_{k;\cham'})}} =  \At^{p^*(-e_{k;\cham})} .}
So it holds for $i=k$.

Assume $i\neq k$.  Then
\eq{
p^*\lrp{\mu_{k,\vb{t};\cham}^q \lrp{\Xt_{i;\cham'}}} 
=
\At^{p^*\lrp{e_{i;\cham}}} \lrp{\displaystyle\prod_{\ell=1}^{|\epsilon_{ik}|} \lrp{\vb{t}^{[\sgn{\lrp{\epsilon_{ik}}}\vb{c}_{k;\cham}]_+}+\vb{t}^{[-\sgn{\lrp{\epsilon_{ik}}}\vb{c}_{k;\cham}]_+} \qfg^{\frac{2\ell-1}{d_k}} \At^{p^*\lrp{-\sgn{\lrp{\epsilon_{ik}}} e_{k;\cham} } }}}^{-\sgn{\lrp{\epsilon_{ik}}}}
}{eq:pmu}
and
\eq{
\mu_{k,\vb{t};\cham}^q \lrp{p^*\lrp{\Xt_{i;\cham'}}} =& \mu_{k,\vb{t};\cham}^q  \lrp{ \At^{p^*\lrp{e_{i;\cham'}}}}\\
=& \qbz^{-\epsilon_{ik}' \frac{1}{2} \Lambda \lrp{ p^*\lrp{e_{i;\cham'}}, f_{k;\cham'}   } }  \At^{p^*\lrp{e_{i;\cham'}} - \epsilon_{ik}'f_{k;\cham'}} \mu_{k,\vb{t};\cham}^q  \lrp{\At^{\epsilon_{ik}'f_{k;\cham'}} }\\
=&  \qbz^{-\epsilon_{ik}' \frac{1}{2} \Lambda \lrp{ p^*\lrp{e_{i;\cham'}}, f_{k;\cham'}   } }  \At^{p^*\lrp{e_{i;\cham'}} - \epsilon_{ik}'f_{k;\cham'}}\\
&\lrp{\vb{t}^{\lrb{\cv_{k;\cham}}_+} \At^{-f_{k;\cham} + \sum_{j:\epsilon_{kj>0}} \epsilon_{kj} f_{j;\cham}} + \vb{t}^{\lrb{-\cv_{k;\cham}}_+} \At^{-f_{k;\cham} -  \sum_{j:\epsilon_{kj<0}} \epsilon_{kj} f_{j;\cham}}}^{\epsilon_{ik}'}\\
=& \At^{p^*(e_{i;\cham'})}\At^{-\epsilon_{ik}'f_{k;\cham'}}
\lrp{\vb{t}^{\lrb{\cv_{k;\cham}}_+} \At^{p^*(e_{k;\cham})+f_{k;\cham'}} + \vb{t}^{\lrb{-\cv_{k;\cham}}_+} \At^{f_{k;\cham'}}}^{\epsilon_{ik}'}
.
}{eq:muqkp-star}

The statement follows immediately if $\epsilon_{ik}=0$. Now take $\epsilon_{ik} < 0$. Then \eqref{eq:muqkp-star} becomes
\eq{
\mu_{k,\vb{t};\cham}^q \lrp{p^*\lrp{\Xt_{i;\cham'}}}
=& \At^{p^*(e_{i;\cham})}
    \prod_{\ell=1}^{\lrm{\epsilon_{ik}}}
        \lrp{
            \vb{t}^{\lrb{\cv_{k;\cham}}_+} \qbz^{\lrp{\lrm{\epsilon_{ik}}-\ell} \Lambda\lrp{-f_{k;\cham'} , p^*(e_{k;\cham}) } } \At^{-f_{k;\cham'} }\At^{p^*(e_{k;\cham})+f_{k;\cham'}} 
            + 
            \vb{t}^{\lrb{-\cv_{k;\cham}}_+}
            } \\
=& \At^{p^*(e_{i;\cham})}
    \prod_{\ell=1}^{\lrm{\epsilon_{ik}}}
        \lrp{
            \vb{t}^{\lrb{\cv_{k;\cham}}_+} \qbz^{-\lrp{\lrm{\epsilon_{ik}}-\ell}d_k^\vee } \At^{-f_{k;\cham'} }\At^{p^*(e_{k;\cham})+f_{k;\cham'}} 
            + 
            \vb{t}^{\lrb{-\cv_{k;\cham}}_+}
            } \\
=& \At^{p^*(e_{i;\cham})}
    \prod_{\ell=1}^{\lrm{\epsilon_{ik}}}
        \lrp{
            \vb{t}^{\lrb{\cv_{k;\cham}}_+} \qbz^{-\lrp{\lrm{\epsilon_{ik}}-\ell+\frac{1}{2}}d_k^\vee } \At^{p^*(e_{k;\cham})} 
            + 
            \vb{t}^{\lrb{-\cv_{k;\cham}}_+}
            } \\
=& \At^{p^*(e_{i;\cham})}
    \prod_{\ell=1}^{\lrm{\epsilon_{ik}}}
        \lrp{
            \vb{t}^{\lrb{\cv_{k;\cham}}_+} \qbz^{-\frac{2\ell-1}{2}d_k^\vee } \At^{p^*(e_{k;\cham})} 
            + 
            \vb{t}^{\lrb{-\cv_{k;\cham}}_+}
            }. 
}{eq:muqk-neg}
In the final equality, we have used the fact that all term to the right of the product symbol commute with one another to reorder the product. 

Meanwhile, \eqref{eq:pmu} becomes
\eq{p^*\lrp{\mu_{k,\vb{t};\cham}^q \lrp{\Xt_{i;\cham'}}} 
=
\At^{p^*\lrp{e_{i;\cham}}} \displaystyle\prod_{\ell=1}^{|\epsilon_{ik}|} \lrp{\vb{t}^{[-\vb{c}_{k;\cham}]_+}+\vb{t}^{[\vb{c}_{k;\cham}]_+} \qfg^{\frac{2\ell-1}{d_k}} \At^{p^*\lrp{ e_{k;\cham} } }}.}{eq:pmuneg}
After making the substitution $ \qfg^{\frac{1}{d}}= \qbz^{-\frac{1}{2}}$, we find
\eqn{p^*\lrp{\mu_{k,\vb{t};\cham}^q \lrp{\Xt_{i;\cham'}}} 
=
\At^{p^*\lrp{e_{i;\cham}}} \displaystyle\prod_{\ell=1}^{|\epsilon_{ik}|} \lrp{\vb{t}^{[-\vb{c}_{k;\cham}]_+}+\vb{t}^{[\vb{c}_{k;\cham}]_+} \qbz^{-\frac{2\ell-1}{2} d_k^\vee } \At^{p^*\lrp{ e_{k;\cham} } }},}
in agreement with \eqref{eq:muqk-neg}.

Next, take $\epsilon_{ik}>0$. 
Then \eqref{eq:muqkp-star} becomes
\eq{
\mu_{k,\vb{t};\cham}^q \lrp{p^*\lrp{\Xt_{i;\cham'}}}
=& \At^{p^*(e_{i;\cham'})}\lrp{\At^{-\epsilon_{ik}f_{k;\cham'}}}^{-1}
    \prod_{\ell=1}^{\epsilon_{ik} }
    \lrp{\vb{t}^{\lrb{\cv_{k;\cham}}_+} \At^{p^*(e_{k;\cham})+f_{k;\cham'}} + \vb{t}^{\lrb{-\cv_{k;\cham}}_+} \At^{f_{k;\cham'}}}^{-1}\\
=& \At^{p^*(e_{i;\cham'})}
    \lrp{
        \lrp{
            \vb{t}^{\lrb{\cv_{k;\cham}}_+}    
            \At^{p^*(e_{k;\cham})+f_{k;\cham'}} 
            + 
            \vb{t}^{\lrb{-\cv_{k;\cham}}_+} \At^{f_{k;\cham'}}
        } 
        \At^{-\epsilon_{ik}f_{k;\cham'}}}^{-1}\\
&    
    \prod_{\ell=1}^{\epsilon_{ik}-1}
        \lrp{\vb{t}^{\lrb{\cv_{k;\cham}}_+} \At^{p^*(e_{k;\cham})+f_{k;\cham'}} + \vb{t}^{\lrb{-\cv_{k;\cham}}_+} \At^{f_{k;\cham'}}}^{-1}\\
=& \At^{p^*(e_{i;\cham'})}
    \lrp{
        \At^{-\lrp{\epsilon_{ik}-1}f_{k;\cham'}}
        \lrp{
            \vb{t}^{\lrb{\cv_{k;\cham}}_+}  \qbz^{\Lambda\lrp{p^*(e_{k;\cham}), -\lrp{\epsilon_{ik}-\frac{1}{2}}f_{k;\cham'}}}  
            \At^{p^*(e_{k;\cham})} 
            + 
            \vb{t}^{\lrb{-\cv_{k;\cham}}_+}
        } 
        }^{-1}\\
&    
    \prod_{\ell=1}^{\epsilon_{ik}-1}
        \lrp{\vb{t}^{\lrb{\cv_{k;\cham}}_+} \At^{p^*(e_{k;\cham})+f_{k;\cham'}} + \vb{t}^{\lrb{-\cv_{k;\cham}}_+} \At^{f_{k;\cham'}}}^{-1}\\
=& \At^{p^*(e_{i;\cham'})}
        \lrp{
            \vb{t}^{\lrb{\cv_{k;\cham}}_+}  \qbz^{ \lrp{\epsilon_{ik}-\frac{1}{2}}d_k^\vee}  
            \At^{p^*(e_{k;\cham})} 
            + 
            \vb{t}^{\lrb{-\cv_{k;\cham}}_+}
        }^{-1}
        \lrp{\At^{-\lrp{\epsilon_{ik}-1}f_{k;\cham'}}}^{-1}\\
&    
    \prod_{\ell=1}^{\epsilon_{ik}-1}
        \lrp{\vb{t}^{\lrb{\cv_{k;\cham}}_+} \At^{p^*(e_{k;\cham})+f_{k;\cham'}} + \vb{t}^{\lrb{-\cv_{k;\cham}}_+} \At^{f_{k;\cham'}}}^{-1}\\
=& \At^{p^*(e_{i;\cham}+ \epsilon_{ik}e_{k;\cham} )}
    \prod_{\ell=1}^{\epsilon_{ik}}
        \lrp{
            \vb{t}^{\lrb{\cv_{k;\cham}}_+}  \qbz^{ \lrp{\epsilon_{ik}-\lrp{\ell-\frac{1}{2}}}d_k^\vee}  
            \At^{p^*(e_{k;\cham})} 
            + 
            \vb{t}^{\lrb{-\cv_{k;\cham}}_+}
        }^{-1}\\
=& \qbz^{ \frac{1}{2} \epsilon_{ik}^2 d_k^\vee} \At^{p^*(e_{i;\cham})} \At^{\epsilon_{ik} p^*(e_{k;\cham} )}
    \prod_{\ell=1}^{\epsilon_{ik}}
        \lrp{
            \vb{t}^{\lrb{\cv_{k;\cham}}_+}  \qbz^{\frac{2\ell-1}{2}d_k^\vee}  
            \At^{p^*(e_{k;\cham})} 
            + 
            \vb{t}^{\lrb{-\cv_{k;\cham}}_+}
        }^{-1}\\
=& \At^{p^*(e_{i;\cham})} 
    \prod_{\ell=1}^{\epsilon_{ik}}
        \lrp{
            \vb{t}^{\lrb{\cv_{k;\cham}}_+}  
            + 
            \vb{t}^{\lrb{-\cv_{k;\cham}}_+}
            \qbz^{-\frac{2\ell-1}{2}d_k^\vee}  
            \At^{p^*(-e_{k;\cham})} 
        }^{-1}.        
}{eq:muqk-pos}
To write the final equality, we use the identity
$\displaystyle{\sum_{\ell=1}^{r} (2\ell-1) = r^2}$.

Meanwhile, \eqref{eq:pmu} becomes
\eq{
p^*\lrp{\mu_{k,\vb{t};\cham}^q \lrp{\Xt_{i;\cham'}}} 
&=
\At^{p^*\lrp{e_{i;\cham}}} \lrp{\displaystyle\prod_{\ell=1}^{\epsilon_{ik}} \lrp{\vb{t}^{[\vb{c}_{k;\cham}]_+}+\vb{t}^{[-\vb{c}_{k;\cham}]_+} \qfg^{\frac{2\ell-1}{d_k}} \At^{p^*\lrp{-e_{k;\cham} } }}}^{-1}\\
&=
\At^{p^*\lrp{e_{i;\cham}}} \prod_{\ell=1}^{\epsilon_{ik}} \lrp{\vb{t}^{[\vb{c}_{k;\cham}]_+}+\vb{t}^{[-\vb{c}_{k;\cham}]_+} \qfg^{\frac{2\ell-1}{d_k}} \At^{p^*\lrp{-e_{k;\cham} } }}^{-1}.
}{eq:pmupos}
After making the substitution $ \qfg^{\frac{1}{d}}= \qbz^{-\frac{1}{2}}$, we find
\eqn{
p^*\lrp{\mu_{k,\vb{t};\cham}^q \lrp{\Xt_{i;\cham'}}} 
=
\At^{p^*\lrp{e_{i;\cham}}} \prod_{\ell=1}^{\epsilon_{ik}} \lrp{\vb{t}^{[\vb{c}_{k;\cham}]_+}+\vb{t}^{[-\vb{c}_{k;\cham}]_+} \qbz^{-\frac{2\ell-1}{2}d_k^\vee} \At^{p^*\lrp{-e_{k;\cham} } }}^{-1},
}
in agreement with \eqref{eq:muqk-pos}.  This completes the proof.
\end{proof}

\begin{remark}
To compare the Berenstein-Zelevinsky and Fock-Goncharov quantizations in the coefficient-free case rather than the principal coefficient case, we can simply take the coefficients $t_i$ to $1$ in \thref{prop:quantum-pstar}. 
\end{remark}

\begin{remark}
Here we have described how to relate two types of quantum cluster varieties with principal coefficients.
Another interesting direction is to relate a quantum cluster variety with principal coefficients to coefficient-free quantum cluster varieties.
Classically, $\cAp$ is intimately related to both an $\cA$ and $\cX$ cluster variety.
It is a deformation family of $\cA$ over a torus $T_M = \Spec(\C[N])$.
Meanwhile, there is an action of a torus $T_{N^\circ}=\Spec(\C[M^\circ])$ on $\cAp$, with $\cX$ the quotient of $\cAp$ under this $T_{N^\circ}$ action.
See \cite[Appendix~B]{GHKK}.
It is natural to wonder how this picture quantizes, and under what conditions.
This is worked out by Davison and Mandel in \cite[Lemma~4.1]{davison2019strong}.
\end{remark}


\section{Recovering the Poisson structure from the quantum structure}\label{sec:PoisFromQ}

The Poisson structure on $\cX$ can be realized as the semi-classical limit of $\cX_q$'s quantum structure, as described in \cite{FG_cluster_ensembles}.
Let us summarize as follows:

Recall from Section~\ref{sec:FG-q} that $\cX_q$ is a union of non-commutative tori $\mathcal{T}_{N;\seed}$, glued via birational quantum mutation maps $\mu_k^q  : \TT_{N;\seed'} \rightarrow \TT_{N;\seed}$, where $\TT_{N;\seed}$ is the noncommutative fraction field of $\mathcal{T}_{N;\seed}$ (Equation~\eqref{eq:muq}). 
On each $\mathcal{T}_{N;\seed}$ we have the relation
    \eqn{q^{-\{v_1,v_2\}} X^{v_1} X^{v_2} = X^{v_1+v_2}.}
 This induces a Poisson structure on the corresponding $\cX$-torus $\mathcal{T}_{N;\seed}$ by taking the {\it{semi-classical limit}} shown below.
    \eqn{ \lrc{X^{v_1}, X^{v_2}} :=& \lim_{q\to 1}  \frac{X^{v_1} X^{v_2} - X^{v_2}X^{v_1} }{q-1} \\
    =& \lim_{q\to 1}  \frac{q^{\{v_1,v_2\}} - q^{\{v_2,v_1\}} }{q-1} X^{v_1 + v_2}\\
    =& 2 \{v_1,v_2\}X^{v_1 + v_2}.
}

Since $\mu_k^q: \TT_{N; \vb{\seed'}} \to \TT_{N; \vb{\seed}}$ is a $*$\=/algebra homomorphism that restricts to the identity on $\C\lrb{q^{\pm \frac{1}{d}}}$, we have
    \eqn{ {\mu_k^q} \lrp{ \frac{X^{v_1}X^{v_2} -X^{v_2}X^{v_1}}{q-1}  }  =  \frac{ {\mu_k^q}\lrp{X^{v_1}}{\mu_k^q}\lrp{X^{v_2}} - {\mu_k^q}\lrp{X^{v_2}}{\mu_k^q}\lrp{X^{v_1}}}{q-1}  . }    
Since the usual mutation formula can be recovered by taking $q \to 1$, we have
    \eqn{{\mu_k^*} \lrp{ \lrc{X^{v_1},X^{v_2}}_{\seed'} } :=& {\mu_k^*} \lrp{ \lim_{q \to 1} \frac{X^{v_1}X^{v_2} -X^{v_2}X^{v_1}}{q-1}  } \\
    =& \lim_{q \to 1}  {\mu_k^q} \lrp{ \frac{X^{v_1}X^{v_2} -X^{v_2}X^{v_1}}{q-1}  }\\  
    =& \lim_{q \to 1} \frac{ {\mu_k^q}\lrp{X^{v_1}}{\mu_k^q}\lrp{X^{v_2}} - {\mu_k^q}\lrp{X^{v_2}}{\mu_k^q}\lrp{X^{v_1}}}{q-1}\\
    =&:  \lrc{{\mu_k^*}\lrp{X^{v_1}},{\mu_k^*}\lrp{X^{v_2}}}_{\seed} . }
Thus, the Poisson structures on tori $\mathcal{T}_{N; \seed}$ patch together giving a global Poisson structure on $\cX$.   

In this section we extend the above argument to $\Xfsp$ and its fibers $\Xsp_{\vb{t}}$.
We start by giving the key definitions.
In the following definitions, $A$ is a commutative ring with $1$.
\begin{definition}
Let $\mathcal{F}$ be a sheaf of $A$-algebras on $X$.
We say that $\mathcal{F}$ is an {\it{$A$-Poisson sheaf}} if for each open set $U \subset X$,
we can endow $\mathcal{F}\lrp{U}$ with an $A$-bilinear Poisson bracket $\lrc{ \cdot , \cdot }_U$, with the Poisson brackets satisfying the following compatibility property:
\begin{center}
 If $V \subset U$ and $f,g \in \mathcal{F}\lrp{U}$, then $\lrc{f|_{V},g|_{V} }_V = \left.\lrc{f,g}_U\right|_V$.
\end{center}
\end{definition}
\begin{definition}
Let $X$ be a scheme over $A$.
We say $X$ is a {\it{Poisson scheme over $A$}} if $\ssO_X$ is an $A$-Poisson sheaf. 
\end{definition}
\begin{prop} \thlabel{PFam}
$\Xfsp$ is a Poisson scheme over $R$, with the Poisson structure inherited from the Poisson structure on affine patches 
$\A^{\lrm{I}}_{M;\cham}\lrp{R}:= \Spec\lrp{R\lrb{\Xt_{i;\cham}: i\in I}}$:
\eqn{\lrc{\Xt^{v_1}_{\cham} , \Xt^{v_2}_{\cham} }_{\cham} := 2 \lrc{ v_{1} , v_{2} } \Xt^{v_1 + v_2}_{\cham}. }
\end{prop}

\begin{proof}
By definition, $\Xfsp_q$ is a union of non-commutative affine $R$-schemes $\A_{M;\cham,q}^{\lrm{I}}$, 
each corresponding to a ring $A_\cham= R[q^{\pm \frac{1}{d}}]\lra{\tilde{X}_{\cham}^n: n \in N }/\sim$ where the variables satisfy the $q$-commutation relations given by 
\eqn{ \tilde{X}_{\cham}^n \tilde{X}_{\cham}^{n'} = q^{\lrc{n,n'}} \tilde{X}_{\cham}^{n+n'} .}
These patches are glued via the mutations given in \eqref{eq:qdmutation}. 
By \thref{prop:*hom}, these mutation maps are $*$\=/algebra homomorphisms, and by construction they restrict to the identity on $R[q^{\pm \frac{1}{d}}]$. 
We have precisely the same semi-classical limit as the coefficient free case: 
\eqn{ &\lim_{q \to 1} \frac{\Xt^{v_1}_{\cham} \Xt^{v_2}_{\cham}  - \Xt^{v_2}_{\cham}  \Xt^{v_1}_{\cham} }{q-1}\\
= &\lim_{q \to 1} \frac{ q^{\lrc{v_1,v_2}}  - q^{\lrc{v_2,v_1}}}{q-1} \Xt_{\cham}^{v_1 +v_2} \\= &2 \lrc{v_1,v_2} \Xt_{\cham}^{v_1 +v_2}.}
Next, by \thref{cor:recover} the classical mutation formula with coefficients is recovered as the $q\to 1$ limit of the quantum mutation formula with coefficients.
The argument of the quantum coefficient-free case directly applies to establish that
\eqn{\mu_{k,\vb{t}}^*\lrp{ \lrc{ \Xt^{v_1}_{\cham'} , \Xt^{v_2}_{\cham'} }_{\cham'} } =  \lrc{ \mu_{k,\vb{t}}^*\lrp{ \Xt^{v_1}_{\cham'} }, \mu_{k,\vb{t}}^*\lrp{ \Xt^{v_2}_{\cham'}} }_{\cham} .}  
The $R$-Poisson structures on affine patches glue together giving a global $R$-Poisson structure on $\Xfsp$. 
\end{proof}

An immediate corollary of \thref{PFam} is that the fibers are also Poisson:

\begin{cor}\thlabel{PFib}
Let $\vb{t}$ be a closed point of $\Spec\lrp{R}$. 
Then $\Xsp_{\vb{t}}$ is a Poisson scheme over $\C$.
\end{cor}


\begin{appendices}

\section{Direct calculation of the Poisson structure}\label{app:poisson}

It can of course be shown that $\Xfsp$ is an $R$-Poisson scheme without passing through its quantization.
As some readers may appreciate having a low-tech, direct proof of \thref{PFam} for reference, even if said proof is inefficient,
we provide one here.

\begin{proof}
Let $\mu_k: \A^{\lrm{I}}_{M;\cham}\lrp{R} \dashrightarrow \A^{\lrm{I}}_{M;\cham'}\lrp{R}$ be a mutation gluing neighboring affine patches in $\Xfsp$.
To rephrase the statement of the proposition, we have the following commutative diagram:
\begin{equation}\label{CD}
\begin{tikzcd}
    \sK\lrp{\A^{\lrm{I}}_{M;\cham}\lrp{R}} \times \sK\lrp{\A^{\lrm{I}}_{M;\cham}\lrp{R}} \arrow[d,"\lrc{ \cdot , \cdot }_{\cham}"] & \arrow[l,"\lrp{\mu_k^*, \mu_k^*}"'] \sK\lrp{\A^{\lrm{I}}_{M;\cham'}\lrp{R}} \times \sK\lrp{\A^{\lrm{I}}_{M;\cham'}\lrp{R}} \arrow[d,"\lrc{ \cdot , \cdot }_{\cham'}"]\\
     \sK\lrp{\A^{\lrm{I}}_{M;\cham}\lrp{R}} & \arrow[l,"\mu_k^*"] \sK\lrp{\A^{\lrm{I}}_{M;\cham'}\lrp{R}} 
\end{tikzcd}.
\end{equation}
We claim first that this diagram commutes when we restrict to cluster variables $\Xt_{i;\cham'}$.
Assuming this claim, we would have
\eqn{
\lrc{\mu_k^*\lrp{\Xt_{i;\cham'}}, \mu_k^*\lrp{\Xt_{j_1;\cham'}\cdots \Xt_{j_r;\cham'}} }_{\cham} 
=& 
\lrc{\mu_k^*\lrp{\Xt_{i;\cham'}}, \mu_k^*\lrp{\Xt_{j_1;\cham'}}}_{\cham} \mu_k^*\lrp{\Xt_{j_2;\cham'}} \cdots \mu_k^*\lrp{\Xt_{j_r;\cham'}}+ \cdots \\
&+
\lrc{\mu_k^*\lrp{\Xt_{i;\cham'}}, \mu_k^*\lrp{\Xt_{j_r;\cham'}}}_{\cham} \mu_k^*\lrp{\Xt_{j_1;\cham'}} \cdots \mu_k^*\lrp{\Xt_{j_{r-1};\cham'}} \\
=& 
\mu_k^*\lrp{\lrc{\Xt_{i;\cham'}, \Xt_{j_1;\cham'}}_{\cham'}} \mu_k^*\lrp{\Xt_{j_2;\cham'}} \cdots \mu_k^*\lrp{\Xt_{j_r;\cham'}}+ \cdots \\
&+
\mu_k^*\lrp{\lrc{\Xt_{i;\cham'}, \Xt_{j_r;\cham'}}_{\cham'}} \mu_k^*\lrp{\Xt_{j_1;\cham'}} \cdots \mu_k^*\lrp{\Xt_{j_{r-1};\cham'}}\\
=&\mu_k^*\lrp{\lrc{\Xt_{i;\cham'}, \Xt_{j_1;\cham'}\cdots \Xt_{j_r;\cham'}}_{\cham'}}.
}
We extend to monomials in both arguments similarly, and use linearity of $\mu_k^*$ and bilinearity of Poisson brackets to extend to polynomials in each argument.
Now suppose we have a pair of rational functions.
\eqn{
\lrc{\mu_k^*\lrp{\frac{f_1}{g_1}}, \mu_k^*\lrp{\frac{f_2}{g_2}} }_{\cham} =& 
\mu_k^*\lrp{\frac{1}{g_1 g_2}}
\lrc{\mu_k^*\lrp{f_1}, \mu_k^*\lrp{f_2} }_{\cham} 
+\mu_k^*\lrp{\frac{f_2}{g_1}}
\lrc{\mu_k^*\lrp{f_1}, \mu_k^*\lrp{\frac{1}{g_2}} }_{\cham} \\
&+\mu_k^*\lrp{\frac{f_1}{g_2}}
\lrc{\mu_k^*\lrp{\frac{1}{g_1}}, \mu_k^*\lrp{f_2} }_{\cham}
+\mu_k^*\lrp{f_1 f_2}
\lrc{\mu_k^*\lrp{\frac{1}{g_1}}, \mu_k^*\lrp{\frac{1}{g_2}} }_{\cham}  
}
To simplify we note that  $ \lrc{f,1} = 0 $, so
\eqn{
0 = \lrc{f,\frac{g}{g}} = \frac{1}{g}\lrc{f,g} + g \lrc{f,\frac{1}{g}} 
}
and
\eqn{
\lrc{f,\frac{1}{g}} = - \frac{1}{g^2} \lrc{f,g}.
}
So,
\eqn{
\lrc{\mu_k^*\lrp{\frac{f_1}{g_1}}, \mu_k^*\lrp{\frac{f_2}{g_2}} }_{\cham} =& 
\mu_k^*\lrp{\frac{1}{g_1 g_2}}
\lrc{\mu_k^*\lrp{f_1}, \mu_k^*\lrp{f_2} }_{\cham} 
+\mu_k^*\lrp{\frac{f_2}{g_1 g_2^2}}
\lrc{\mu_k^*\lrp{f_1}, \mu_k^*\lrp{g_2} }_{\cham} \\
&+\mu_k^*\lrp{\frac{f_1}{g_1^2 g_2}}
\lrc{\mu_k^*\lrp{g_1}, \mu_k^*\lrp{f_2} }_{\cham}
+\mu_k^*\lrp{\frac{f_1 f_2}{g_1^2 g_2^2}}
\lrc{\mu_k^*\lrp{g_1}, \mu_k^*\lrp{g_2} }_{\cham} \\
=&
\mu_k^*\lrp{\frac{1}{g_1 g_2}}
\mu_k^*\lrp{\lrc{f_1, f_2}_{\cham'}} 
+\mu_k^*\lrp{\frac{f_2}{g_1 g_2^2}}
\mu_k^*\lrp{\lrc{f_1, g_2}_{\cham'}} \\
&+\mu_k^*\lrp{\frac{f_1}{g_1^2 g_2}}
\mu_k^*\lrp{\lrc{g_1, f_2}_{\cham'}} 
+\mu_k^*\lrp{\frac{f_1 f_2}{g_1^2 g_2^2}}
\mu_k^*\lrp{\lrc{g_1, g_2}_{\cham'}} \\
=& \mu_k^*\lrp{\lrc{\frac{f_1}{g_1}, \frac{f_2}{g_2}}_{\cham'}}. 
}
This shows that the claim that the diagram commutes for cluster variables implies the proposition.  We now tackle this claim.

Denote the top path $\lrc{\cdot, \cdot}_{\cham} \circ \lrp{\mu_k^*, \mu_k^*}$ by $\tp$ and the bottom path $ \mu_k^* \circ \lrc{\cdot, \cdot}_{\cham'}$ by $\bp$.
Both $\tp\lrp{f,f}$ and $\bp\lrp{f,f}$ are clearly $0$, so assume from now on that the arguments are distinct.

Note that $\lrc{ \cdot, \cdot}_\cham$ induces the Poisson bivector field 
\eqn{\pi_\cham= \sum_{i,j} \lrc{e_{i;\cham}, e_{j;\cham}}\Xt_{i;\cham} \Xt_{j;\cham}\, \partial_{\Xt_{i;\cham}}\wedge \partial_{\Xt_{j;\cham}} }
and $\lrc{f,g}_\cham = \pi_\cham\lrp{df\wedge dg}$.

{\bf{Case 1:}}
Check $\tp\lrp{\Xt_{i;\cham'},\Xt_{k;\cham'} } {=} \bp\lrp{\Xt_{i;\cham'},\Xt_{k;\cham'} }  $
\eqn{
\tp\lrp{\Xt_{i;\cham'},\Xt_{k;\cham'} } =& \pi_\cham \lrp{ \sum_{l,m}{ \pd{\mu_k^*\lrp{\Xt_{i;\cham'}}}{\Xt_{l;\cham}}  \pd{\mu_k^*\lrp{\Xt_{k;\cham'}}}{\Xt_{m;\cham}}     d\Xt_{l;\cham} \wedge d\Xt_{m;\cham}  }}}
We compute the partial derivatives.  If $\epsilon_{ik} \neq 0 $, $ \pd{\mu_k^*\lrp{\Xt_{i;\cham'}}}{\Xt_{l;\cham}} $ is given by
\eqn{ &\pds{\Xt_{i;\cham}\lrp{\vb{t}^{\lrb{\sgn\lrp{\epsilon_{ik}}\cv_{k;\cham}}_+} + \vb{t}^{\lrb{-\sgn\lrp{\epsilon_{ik}}\cv_{k;\cham}}_+}\Xt_{k;\cham}^{-\sgn\lrp{\epsilon_{ik}}}}^{-\epsilon_{ik}}}{\Xt_{l;\cham}} \\
&= \delta_{il} 
\lrp{\vb{t}^{\lrb{\sgn\lrp{\epsilon_{ik}}\cv_{k;\cham}}_+} + \vb{t}^{\lrb{-\sgn\lrp{\epsilon_{ik}}\cv_{k;\cham}}_+}\Xt_{k;\cham}^{-\sgn\lrp{\epsilon_{ik}}}}^{-\epsilon_{ik}}\\
&\phantom{=}+ \delta_{kl} \sgn\lrp{\epsilon_{ik}}\epsilon_{ik} \vb{t}^{\lrb{-\sgn\lrp{\epsilon_{ik}}\cv_{k;\cham}}_+} \Xt_{i;\cham}\Xt_{k;\cham}^{-\sgn\lrp{\epsilon_{ik}}-1} \lrp{\vb{t}^{\lrb{\sgn\lrp{\epsilon_{ik}}\cv_{k;\cham}}_+} + \vb{t}^{\lrb{-\sgn\lrp{\epsilon_{ik}}\cv_{k;\cham}}_+}\Xt_{k;\cham}^{-\sgn\lrp{\epsilon_{ik}}}}^{-\epsilon_{ik}-1}.
}
If $\epsilon_{ik} = 0$, instead we simply get
\eqn{
\pd{\mu_k^*\lrp{\Xt_{i;\cham'}}}{\Xt_{l;\cham}} = \delta_{il}.
}
Meanwhile
\eqn{
\pd{\mu_k^*\lrp{\Xt_{k;\cham'}}}{\Xt_{m;\cham}} = 
\pd{\Xt_{k;\cham}^{-1}}{\Xt_{m;\cham}} = - \delta_{km} \Xt_{k;\cham}^{-2}.
}

{\bf{Subcase 1.a: }} $\epsilon_{ik} \neq 0$

Using $d\Xt_{k;\cham} \wedge d\Xt_{k;\cham} = 0$ we have
\eqn{
\tp\lrp{\Xt_{i;\cham'},\Xt_{k;\cham'}}  =& \pi_{\cham} \lrp{-  \Xt_{k;\cham}^{-2}  \lrp{\vb{t}^{\lrb{\sgn\lrp{\epsilon_{ik}}\cv_{k;\cham}}_+} + \vb{t}^{\lrb{-\sgn\lrp{\epsilon_{ik}}\cv_{k;\cham}}_+}\Xt_{k;\cham}^{-\sgn\lrp{\epsilon_{ik}}}}^{-\epsilon_{ik}} d\Xt_{i;\cham} \wedge d\Xt_{k;\cham}   }\\
=& -\lrc{e_{i;\cham}, e_{k;\cham}} \Xt_{i;\cham} \Xt_{k;\cham} \Xt_{k;\cham}^{-2}  \lrp{\vb{t}^{\lrb{\sgn\lrp{\epsilon_{ik}}\cv_{k;\cham}}_+} + \vb{t}^{\lrb{-\sgn\lrp{\epsilon_{ik}}\cv_{k;\cham}}_+}\Xt_{k;\cham}^{-\sgn\lrp{\epsilon_{ik}}}}^{-\epsilon_{ik}}\\
=& \lrc{e_{i;\cham'}, e_{k;\cham'}} \Xt_{i;\cham} \lrp{\vb{t}^{\lrb{\sgn\lrp{\epsilon_{ik}}\cv_{k;\cham}}_+} + \vb{t}^{\lrb{-\sgn\lrp{\epsilon_{ik}}\cv_{k;\cham}}_+}\Xt_{k;\cham}^{-\sgn\lrp{\epsilon_{ik}}}}^{-\epsilon_{ik}}
\Xt_{k;\cham}^{-1}  \\
=& \bp\lrp{\Xt_{i;\cham'},\Xt_{k;\cham'}}
}

{\bf{Subcase 1.b: }} $\epsilon_{ik} = 0$

Then $\lrc{e_{i;\cham},e_{k;\cham}} =  \lrc{e_{i;\cham'},e_{k;\cham'}} = 0$,
and $\tp\lrp{\Xt_{i;\cham'},\Xt_{k;\cham'}} = \bp\lrp{\Xt_{i;\cham'},\Xt_{k;\cham'}} = 0 $.

{\bf{Case 2:}}
Check $\tp\lrp{\Xt_{i;\cham'},\Xt_{j;\cham'} } = \bp\lrp{\Xt_{i;\cham'},\Xt_{j;\cham'} }, \quad i, j \neq k  $

\eqn{ \tp\lrp{\Xt_{i;\cham'},\Xt_{j;\cham'} } =& \pi_\cham \lrp{ \sum_{l,m}{ \pd{\mu_k^*\lrp{\Xt_{i;\cham'}}}{\Xt_{l;\cham}}  \pd{\mu_k^*\lrp{\Xt_{j;\cham'}}}{\Xt_{m;\cham}}     d\Xt_{l;\cham} \wedge d\Xt_{m;\cham}  }}}
We computed the necessary partial derivatives while addressing the previous case. 

{\bf{Subcase 2.a: }} $\epsilon_{ik}, \epsilon_{jk} \neq 0$

Then $\tp\lrp{\Xt_{i;\cham'},\Xt_{j;\cham'} }$ is given by
\eqn{ &\lrc{e_{i;\cham}, e_{j;\cham}} \Xt_{i;\cham} \Xt_{j;\cham} \lrp{\vb{t}^{\lrb{\sgn\lrp{\epsilon_{ik}}\cv_{k;\cham}}_+} + \vb{t}^{\lrb{-\sgn\lrp{\epsilon_{ik}}\cv_{k;\cham}}_+}\Xt_{k;\cham}^{-\sgn\lrp{\epsilon_{ik}}}}^{-\epsilon_{ik}} \\
&\times \lrp{\vb{t}^{\lrb{\sgn\lrp{\epsilon_{jk}}\cv_{k;\cham}}_+} + \vb{t}^{\lrb{-\sgn\lrp{\epsilon_{jk}}\cv_{k;\cham}}_+}\Xt_{k;\cham}^{-\sgn\lrp{\epsilon_{jk}}}}^{-\epsilon_{jk}}\\
+& \lrc{e_{i;\cham}, e_{k;\cham}}\Xt_{i;\cham} \Xt_{j;\cham} \Xt_{k;\cham}^{-\sgn\lrp{\epsilon_{jk}}} \lrp{\vb{t}^{\lrb{\sgn\lrp{\epsilon_{ik}}\cv_{k;\cham}}_+} + \vb{t}^{\lrb{-\sgn\lrp{\epsilon_{ik}}\cv_{k;\cham}}_+}\Xt_{k;\cham}^{-\sgn\lrp{\epsilon_{ik}}}}^{-\epsilon_{ik}}\\ 
&\times \lrm{\epsilon_{jk} } \vb{t}^{\lrb{-\sgn\lrp{\epsilon_{jk}}\cv_{k;\cham}}_+} \lrp{\vb{t}^{\lrb{\sgn\lrp{\epsilon_{jk}}\cv_{k;\cham}}_+} + \vb{t}^{\lrb{-\sgn\lrp{\epsilon_{jk}}\cv_{k;\cham}}_+}\Xt_{k;\cham}^{-\sgn\lrp{\epsilon_{jk}}}}^{-\epsilon_{jk}-1}\\
+& \lrc{e_{k;\cham}, e_{j;\cham}}\Xt_{i;\cham}\Xt_{j;\cham}\Xt_{k;\cham}^{-\sgn\lrp{\epsilon_{ik}}} 
\lrm{\epsilon_{ik} } \vb{t}^{\lrb{-\sgn\lrp{\epsilon_{ik}}\cv_{k;\cham}}_+} \lrp{\vb{t}^{\lrb{\sgn\lrp{\epsilon_{ik}}\cv_{k;\cham}}_+} + \vb{t}^{\lrb{-\sgn\lrp{\epsilon_{ik}}\cv_{k;\cham}}_+}\Xt_{k;\cham}^{-\sgn\lrp{\epsilon_{ik}}}}^{-\epsilon_{ik}-1}\\
&\times \lrp{\vb{t}^{\lrb{\sgn\lrp{\epsilon_{jk}}\cv_{k;\cham}}_+} + \vb{t}^{\lrb{-\sgn\lrp{\epsilon_{jk}}\cv_{k;\cham}}_+}\Xt_{k;\cham}^{-\sgn\lrp{\epsilon_{jk}}}}^{-\epsilon_{jk}}
}
This simplifies as follows.
\eqn{
&\Xt_{i;\cham}\lrp{\vb{t}^{\lrb{\sgn\lrp{\epsilon_{ik}}\cv_{k;\cham}}_+} + \vb{t}^{\lrb{-\sgn\lrp{\epsilon_{ik}}\cv_{k;\cham}}_+}\Xt_{k;\cham}^{-\sgn\lrp{\epsilon_{ik}}}}^{-\epsilon_{ik}}
\Xt_{j;\cham}\lrp{\vb{t}^{\lrb{\sgn\lrp{\epsilon_{jk}}\cv_{k;\cham}}_+} + \vb{t}^{\lrb{-\sgn\lrp{\epsilon_{jk}}\cv_{k;\cham}}_+}\Xt_{k;\cham}^{-\sgn\lrp{\epsilon_{jk}}}}^{-\epsilon_{jk}}\\
&\times
\bigg( 
\lrc{e_{i;\cham}, e_{j;\cham}} \\
&\phantom{\times \bigg(} \negphantom{+}
+\lrc{e_{i;\cham}, e_{k;\cham}}
\lrm{\epsilon_{jk} } \Xt_{k;\cham}^{-\sgn\lrp{\epsilon_{jk}}}  \vb{t}^{\lrb{-\sgn\lrp{\epsilon_{jk}}\cv_{k;\cham}}_+} \lrp{\vb{t}^{\lrb{\sgn\lrp{\epsilon_{jk}}\cv_{k;\cham}}_+} + \vb{t}^{\lrb{-\sgn\lrp{\epsilon_{jk}}\cv_{k;\cham}}_+}\Xt_{k;\cham}^{-\sgn\lrp{\epsilon_{jk}}}}^{-1}\\
&\phantom{\times \bigg(}\left. \negphantom{+}+
\lrc{e_{k;\cham}, e_{j;\cham}}  
\lrm{\epsilon_{ik} } \Xt_{k;\cham}^{-\sgn\lrp{\epsilon_{ik}}} 
\vb{t}^{\lrb{-\sgn\lrp{\epsilon_{ik}}\cv_{k;\cham}}_+} \lrp{\vb{t}^{\lrb{\sgn\lrp{\epsilon_{ik}}\cv_{k;\cham}}_+} + \vb{t}^{\lrb{-\sgn\lrp{\epsilon_{ik}}\cv_{k;\cham}}_+}\Xt_{k;\cham}^{-\sgn\lrp{\epsilon_{ik}}}}^{-1}
\right)
}
The top line of this expression is just $\mu_k^*\lrp{\Xt_{i;\cham'}} \mu_k^*\lrp{\Xt_{j;\cham'}}$.  
So, we would like to show that the three-line expression in parenthesis reduces to 
\eqn{
\lrc{e_{i;\cham'}, e_{j;\cham'}}
&= 
\lrc{e_{i;\cham}+ \lrb{\epsilon_{ik}}_+e_{k;\cham}, e_{j;\cham}+\lrb{\epsilon_{jk}}_+e_{k;\cham}}\\
&=\lrc{e_{i;\cham}, e_{j;\cham}}
+\lrc{e_{i;\cham},e_{k;\cham}}\lrb{\epsilon_{jk}}_+
+\lrc{e_{k;\cham}, e_{j;\cham}}\lrb{\epsilon_{ik}}_+
.
}
That is, we would like to see that
\eqn{
&\lrc{e_{i;\cham}, e_{k;\cham}}
\lrm{\epsilon_{jk} } \Xt_{k;\cham}^{-\sgn\lrp{\epsilon_{jk}}}  \vb{t}^{\lrb{-\sgn\lrp{\epsilon_{jk}}\cv_{k;\cham}}_+} \lrp{\vb{t}^{\lrb{\sgn\lrp{\epsilon_{jk}}\cv_{k;\cham}}_+} + \vb{t}^{\lrb{-\sgn\lrp{\epsilon_{jk}}\cv_{k;\cham}}_+}\Xt_{k;\cham}^{-\sgn\lrp{\epsilon_{jk}}}}^{-1}\\
+&\lrc{e_{k;\cham}, e_{j;\cham}}  
\lrm{\epsilon_{ik} } \Xt_{k;\cham}^{-\sgn\lrp{\epsilon_{ik}}} 
\vb{t}^{\lrb{-\sgn\lrp{\epsilon_{ik}}\cv_{k;\cham}}_+} \lrp{\vb{t}^{\lrb{\sgn\lrp{\epsilon_{ik}}\cv_{k;\cham}}_+} + \vb{t}^{\lrb{-\sgn\lrp{\epsilon_{ik}}\cv_{k;\cham}}_+}\Xt_{k;\cham}^{-\sgn\lrp{\epsilon_{ik}}}}^{-1}
}
is just an overly complicated way to write
$\lrc{e_{i;\cham},e_{k;\cham}}\lrb{\epsilon_{jk}}_+
+\lrc{e_{k;\cham}, e_{j;\cham}}\lrb{\epsilon_{ik}}_+$.

If $\sgn\lrp{\epsilon_{ik}} = \sgn\lrp{\epsilon_{jk}} =: \sigma $, we have the following simplification.
\eqn{
&\Xt_{k;\cham}^{-\sigma}  \vb{t}^{\lrb{-\sigma\cv_{k;\cham}}_+} \lrp{\vb{t}^{\lrb{\sigma\cv_{k;\cham}}_+} + \vb{t}^{\lrb{-\sigma\cv_{k;\cham}}_+}\Xt_{k;\cham}^{-\sigma}}^{-1}
\lrp{ \lrc{e_{i;\cham}, e_{k;\cham}}
\lrm{\epsilon_{jk} }  + \lrc{e_{k;\cham}, e_{j;\cham}}  
\lrm{\epsilon_{ik} } }\\
=& \Xt_{k;\cham}^{-\sigma}  \vb{t}^{\lrb{-\sigma\cv_{k;\cham}}_+} \lrp{\vb{t}^{\lrb{\sigma\cv_{k;\cham}}_+} + \vb{t}^{\lrb{-\sigma\cv_{k;\cham}}_+}\Xt_{k;\cham}^{-\sigma}}^{-1}
\lrp{ \lrc{e_{i;\cham}, e_{k;\cham}}
\lrm{\lrc{e_{j;\cham}, e_{k;\cham}} d_k }  
+ 
\lrc{e_{k;\cham}, e_{j;\cham}} 
\lrm{\lrc{e_{i;\cham}, e_{k;\cham}} d_k} }\\
=& \sigma d_k \Xt_{k;\cham}^{-\sigma}  \vb{t}^{\lrb{-\sigma\cv_{k;\cham}}_+} \lrp{\vb{t}^{\lrb{\sigma\cv_{k;\cham}}_+} + \vb{t}^{\lrb{-\sigma\cv_{k;\cham}}_+}\Xt_{k;\cham}^{-\sigma}}^{-1}
\lrp{  \lrc{e_{i;\cham}, e_{k;\cham}}
\lrc{e_{j;\cham}, e_{k;\cham}}
- 
\lrc{e_{j;\cham}, e_{k;\cham}} \lrc{e_{i;\cham}, e_{k;\cham}} }\\
=& 0
}
Likewise,
\eqn{ \lrc{e_{i;\cham},e_{k;\cham}}\lrb{\epsilon_{jk}}_+
+\lrc{e_{k;\cham}, e_{j;\cham}}\lrb{\epsilon_{ik}}_+ &=
\lrc{e_{i;\cham},e_{k;\cham}}\lrb{\lrc{e_{j;\cham},e_{k;\cham}}d_k}_+
+\lrc{e_{k;\cham}, e_{j;\cham}}\lrb{\lrc{e_{i;\cham},e_{k;\cham}}d_k}_+\\
&=
d_k \lrc{e_{i;\cham},e_{k;\cham}}\lrc{e_{j;\cham},e_{k;\cham}}\lrp{\lrb{\sigma}_+ -\lrb{\sigma}_+}\\
&= 0.
} 

On the other hand, if $\sgn\lrp{\epsilon_{ik}} = -\sgn\lrp{\epsilon_{jk}} =: \sigma $, we have
\eqn{&\lrc{e_{i;\cham}, e_{k;\cham}}
\lrm{\epsilon_{jk} } \Xt_{k;\cham}^{\sigma}  \vb{t}^{\lrb{\sigma\cv_{k;\cham}}_+} \lrp{\vb{t}^{\lrb{-\sigma\cv_{k;\cham}}_+} + \vb{t}^{\lrb{\sigma\cv_{k;\cham}}_+}\Xt_{k;\cham}^{\sigma}}^{-1}\\
+&\lrc{e_{k;\cham}, e_{j;\cham}}  
\lrm{\epsilon_{ik} } \Xt_{k;\cham}^{-\sigma} 
\vb{t}^{\lrb{-\sigma\cv_{k;\cham}}_+} \lrp{\vb{t}^{\lrb{\sigma\cv_{k;\cham}}_+} + \vb{t}^{\lrb{-\sigma\cv_{k;\cham}}_+}\Xt_{k;\cham}^{-\sigma}}^{-1}\\
=& \bigg(
-\sigma d_k \lrc{e_{i;\cham}, e_{k;\cham}}
\lrc{e_{j;\cham}, e_{k;\cham}}  \Xt_{k;\cham}^{\sigma}  \vb{t}^{\lrb{\sigma\cv_{k;\cham}}_+}
\lrp{\vb{t}^{\lrb{\sigma\cv_{k;\cham}}_+} + \vb{t}^{\lrb{-\sigma\cv_{k;\cham}}_+}\Xt_{k;\cham}^{-\sigma}}\\
&+
\sigma d_k \lrc{e_{k;\cham}, e_{j;\cham}}  
\lrc{e_{i;\cham}, e_{k;\cham}} \Xt_{k;\cham}^{-\sigma} 
\vb{t}^{\lrb{-\sigma\cv_{k;\cham}}_+}
\lrp{\vb{t}^{\lrb{-\sigma\cv_{k;\cham}}_+} + \vb{t}^{\lrb{\sigma\cv_{k;\cham}}_+}\Xt_{k;\cham}^{\sigma}} \bigg)\\
&\times 
\lrp{\vb{t}^{\lrb{\sigma\cv_{k;\cham}}_+} + \vb{t}^{\lrb{-\sigma\cv_{k;\cham}}_+}\Xt_{k;\cham}^{-\sigma}}^{-1}
\lrp{\vb{t}^{\lrb{-\sigma\cv_{k;\cham}}_+} + \vb{t}^{\lrb{\sigma\cv_{k;\cham}}_+}\Xt_{k;\cham}^{\sigma}}^{-1} \\
=& -\sigma d_k \lrc{e_{i;\cham}, e_{k;\cham}}
\lrc{e_{j;\cham}, e_{k;\cham}} 
\frac{
2 \vb{t}^{\lrb{-\sigma\cv_{k;\cham}}_+}\vb{t}^{\lrb{-\sigma\cv_{k;\cham}}_+}
+ \vb{t}^{2 \lrb{-\sigma\cv_{k;\cham}}_+} \Xt_{k;\cham}^{-\sigma}
+ \vb{t}^{2 \lrb{\sigma\cv_{k;\cham}}_+} \Xt_{k;\cham}^{\sigma}
}{
2 \vb{t}^{\lrb{-\sigma\cv_{k;\cham}}_+}\vb{t}^{\lrb{-\sigma\cv_{k;\cham}}_+}
+ \vb{t}^{2 \lrb{-\sigma\cv_{k;\cham}}_+} \Xt_{k;\cham}^{-\sigma}
+ \vb{t}^{2 \lrb{\sigma\cv_{k;\cham}}_+} \Xt_{k;\cham}^{\sigma}
}\\
=& -\sigma d_k \lrc{e_{i;\cham}, e_{k;\cham}}
\lrc{e_{j;\cham}, e_{k;\cham}} 
}
Similarly,
\eqn{ \lrc{e_{i;\cham},e_{k;\cham}}\lrb{\epsilon_{jk}}_+
+\lrc{e_{k;\cham}, e_{j;\cham}}\lrb{\epsilon_{ik}}_+ &=
\lrc{e_{i;\cham},e_{k;\cham}}\lrb{\lrc{e_{j;\cham},e_{k;\cham}}d_k}_+
+\lrc{e_{k;\cham}, e_{j;\cham}}\lrb{\lrc{e_{i;\cham},e_{k;\cham}}d_k}_+\\
&=
d_k \lrc{e_{i;\cham},e_{k;\cham}}\lrc{e_{j;\cham},e_{k;\cham}}\lrp{\lrb{-\sigma}_+ -\lrb{\sigma}_+}\\
&= -\sigma d_k \lrc{e_{i;\cham},e_{k;\cham}}\lrc{e_{j;\cham},e_{k;\cham}} .}

{\bf{Subcase 2.b: }} $\epsilon_{ik} = 0$ or $\epsilon_{jk} = 0$

If both are $0$, the result is immediate.  Assume one is non-zero, say $\epsilon_{jk} \neq 0$. 
\eqn{
\tp\lrp{\Xt_{i;\cham'},\Xt_{j;\cham'}} =& 
\lrc{e_{i;\cham},e_{j;\cham}}\Xt_{i;\cham}\Xt_{j;\cham} \lrp{\vb{t}^{\lrb{\sgn\lrp{\epsilon_{jk}}\cv_{k;\cham}}_+} + \vb{t}^{\lrb{-\sgn\lrp{\epsilon_{jk}}\cv_{k;\cham}}_+}\Xt_{k;\cham}^{-\sgn\lrp{\epsilon_{jk}}}}^{-\epsilon_{jk}}\\
&+\lrc{e_{i;\cham},e_{k;\cham}}\Xt_{i;\cham}\Xt_{k;\cham} \lrp{ \cdots }\\
=&\lrc{e_{i;\cham},e_{j;\cham}}\mu_k^*\lrp{\Xt_{i;\cham'}} \mu_k^*\lrp{\Xt_{j;\cham'}}  \qquad \lrp{\text{since } \epsilon_{ik}=0 }\\
=& \lrc{e_{i;\cham'},e_{j;\cham'}}\mu_k^*\lrp{\Xt_{i;\cham'}} \mu_k^*\lrp{\Xt_{j;\cham'}} \\
=& \bp\lrp{\Xt_{i;\cham'},\Xt_{j;\cham'}}
}
This establishes the proposition.
\end{proof}


\section{Computation of \texorpdfstring{$\lrp{\fp_\gamma^1}^{-1}$}{wall crossing automorphism}}\label{app:wallx}
Here we compute $\lrp{\fp_\gamma^1}^{-1} =g_1^{-1} \circ g_2  \circ g_1 \circ g_2^{-1} $ to order 2, meaning we keep terms whose exponent vectors $m$ satisfy $d(m) \leq 2$.
We used this in Section~\ref{sec:non-pos} to compute $\widehat{\scat}_2^{\cA}$.

First note that for $u = u_1 f_1 + u_2 f_2$,

\eq{g_1 (A^u)
=& \lrp{ \prod_{\ell=1}^{\lrm{u_2}} \lrp{1+v^{\sgn(u_2) 2 \lrp{2 \ell -1}} A_1^{2} }^{\sgn(u_2)} } A^u\\
=&
\lrp{ \prod_{\ell = 1 }^{\lrm{u_2}} 
    \lrp{ 1 
        + \sgn(u_2) v^{\sgn(u_2) 2 (2 \ell -1) } A_1^2  
        + \lrb{-\sgn(u_2)}_{+}  v^{\sgn(u_2) 4 (2 \ell -1) } A_1^4
        + \cdots}    
}A^u\\
=&
\Bigg( 1
    + \sgn(u_2) \sum_{\ell=1}^{\lrm{u_2}} v^{\sgn(u_2) 2 (2 \ell -1) } A_1^2 \\ 
& \phantom{\Bigg(} + \bigg(
        \sum_{1\leq a < b \leq \lrm{u_2}} v^{\sgn(u_2) 4 \lrp{ a+b -1}}
        + \lrb{-\sgn(u_2)}_{+} \sum_{\ell=1}^{\lrm{u_2}} v^{\sgn(u_2) 4 (2 \ell -1) }
    \bigg) A_1^4 
    + \cdots
\Bigg)A^u
}
{eq:g1ap}
and
\eq{g_2 (A^u) =&  \lrp{ \prod_{\ell=1}^{\lrm{u_1}} \lrp{1+v^{\sgn(u_1) 3 \lrp{2 \ell -1}} A_2^{-3} }^{\sgn(u_1)} } A^u \\
=&\lrp{ \prod_{\ell = 1 }^{\lrm{u_1}} 
    \lrp{ 1 
        + \sgn(u_1) v^{\sgn(u_1) 3 (2 \ell -1) } A_2^{-3}  
        + \lrb{-\sgn(u_1)}_{+}  v^{\sgn(u_1) 3 (2 \ell -1) } A_2^{-6}
        + \cdots}    
}A^u\\
=&
\Bigg( 1
    + \sgn(u_1) \sum_{\ell=1}^{\lrm{u_1}} v^{\sgn(u_1) 3 (2 \ell -1) } A_2^{-3} \\ 
& \phantom{\Bigg(} + \bigg(
        \sum_{1\leq a < b \leq \lrm{u_1}} v^{\sgn(u_1) 6 \lrp{a+b -1}}
        + \lrb{-\sgn(u_1)}_{+} \sum_{\ell=1}^{\lrm{u_1}} v^{\sgn(u_1) 6 (2 \ell -1) }
    \bigg) A_2^{-6} 
    + \cdots
\Bigg)A^u
. }{eq:g2ap}
Then 
\eq{
g_1^{-1} (A^u) =& 
\Bigg( 1
    - \sgn(u_2) \sum_{\ell=1}^{\lrm{u_2}} v^{\sgn(u_2) 2 (2 \ell -1) } A_1^2 \\ 
& \phantom{\Bigg(} + \bigg(
        \sum_{1\leq a < b \leq \lrm{u_2}} v^{\sgn(u_2) 4 \lrp{a+b -1}}
        + \lrb{\sgn(u_2)}_{+} \sum_{\ell=1}^{\lrm{u_2}} v^{\sgn(u_2) 4 (2 \ell -1) }
    \bigg) A_1^4 
    + \cdots
\Bigg)A^u
}{eq:g1inv}
and
\eq{
g_2^{-1} (A^u) =& 
\Bigg( 1
    - \sgn(u_1) \sum_{\ell=1}^{\lrm{u_1}} v^{\sgn(u_1) 3 (2 \ell -1) } A_2^{-3} \\ 
& \phantom{\Bigg(} + \bigg(
        \sum_{1\leq a < b \leq \lrm{u_1}} v^{\sgn(u_1) 6 \lrp{a+b -1}}
        + \lrb{\sgn(u_1)}_{+} \sum_{\ell=1}^{\lrm{u_1}} v^{\sgn(u_1) 6 (2 \ell -1) }
    \bigg) A_2^{-6} 
    + \cdots
\Bigg)A^u.
}{eq:g2inv}
Using \eqref{eq:g2inv} and \eqref{eq:g1ap}, we compute:
\eqn{
g_1 \circ g_2^{-1} (A^u)
=&
\Bigg( 1
    - \sgn(u_1) \sum_{\ell=1}^{\lrm{u_1}} v^{\sgn(u_1) 3 (2 \ell -1) } \lrp{A_2^{-3} - \lrp{v^{-4} + 1 + v^4 } A^{2 f_1 - 3 f_2} + \cdots }\\ 
& \phantom{\Bigg(} + \bigg(
        \sum_{1\leq a < b \leq \lrm{u_1}} v^{\sgn(u_1) 6 \lrp{a+b -1}}
        + \lrb{\sgn(u_1)}_{+} \sum_{\ell=1}^{\lrm{u_1}} v^{\sgn(u_1) 6 (2 \ell -1) }
    \bigg) A_2^{-6}  
    + \cdots
\Bigg)\\
& \Bigg( 1
    + \sgn(u_2) \sum_{\ell=1}^{\lrm{u_2}} v^{\sgn(u_2) 2 (2 \ell -1) } A_1^2 \\ 
& \phantom{\Bigg(} + \bigg(
        \sum_{1\leq a < b \leq \lrm{u_2}} v^{\sgn(u_2) 4 \lrp{ a+b -1}}
        + \lrb{-\sgn(u_2)}_{+} \sum_{\ell=1}^{\lrm{u_2}} v^{\sgn(u_2) 4 (2 \ell -1) }
    \bigg) A_1^4 
    + \cdots
\Bigg)A^u\\
=&
\Bigg( 1 
    - \sgn(u_1) \sum_{\ell=1}^{\lrm{u_1}} v^{\sgn(u_1) 3 (2 \ell -1) } A_2^{-3} 
    + \sgn(u_2) \sum_{\ell=1}^{\lrm{u_2}} v^{\sgn(u_2) 2 (2 \ell -1) } A_1^2 \\
& \phantom{\Bigg(} + \bigg(
        \sum_{1\leq a < b \leq \lrm{u_1}} v^{\sgn(u_1) 6 \lrp{a+b -1}}
        + \lrb{\sgn(u_1)}_{+} \sum_{\ell=1}^{\lrm{u_1}} v^{\sgn(u_1) 6 (2 \ell -1) }
    \bigg) A_2^{-6}\\
& \phantom{\Bigg(} + \bigg(
        \sum_{1\leq a < b \leq \lrm{u_2}} v^{\sgn(u_2) 4 \lrp{a+b -1}}
        + \lrb{-\sgn(u_2)}_{+} \sum_{\ell=1}^{\lrm{u_2}} v^{\sgn(u_2) 4 (2 \ell -1) }
    \bigg) A_1^4 \\
& \phantom{\Bigg(} + \bigg(
    \sgn(u_1) \lrp{v^{-4} + 1 + v^4 } 
    \sum_{\ell=1}^{\lrm{u_1}} v^{\sgn(u_1) 3 (2 \ell -1) }\\
& \phantom{\Bigg( + \bigg(}   - \sgn(u_1)\sgn(u_2) v^{-6} \sum_{\ell_1=1}^{\lrm{u_1}} \sum_{\ell_2=1}^{\lrm{u_2}} v^{\sgn(u_1) 3 (2 \ell_1 -1) + \sgn(u_2) 2 (2 \ell_2 -1) } \bigg) A^{2 f_1 - 3 f_2}  + \cdots   
\Bigg)A^u.
}
Next, using \eqref{eq:g2ap} we compute:
\eqn{g_2 \circ g_1 \circ g_2^{-1} (A^u) =&
\Bigg( 1 
    - \sgn(u_1) \sum_{\ell=1}^{\lrm{u_1}} v^{\sgn(u_1) 3 (2 \ell -1) } A_2^{-3}\\ 
& \phantom{\Bigg(}  + \sgn(u_2) \sum_{\ell=1}^{\lrm{u_2}} v^{\sgn(u_2) 2 (2 \ell -1) } \lrp{A_1^2 + \lrp{v^{-3} + v^3} A^{2 f_1 - 3 f_2} + \cdots } \\
& \phantom{\Bigg(} + \bigg(
        \sum_{1\leq a < b \leq \lrm{u_1}} v^{\sgn(u_1) 6 \lrp{a+b -1}}
        + \lrb{\sgn(u_1)}_{+} \sum_{\ell=1}^{\lrm{u_1}} v^{\sgn(u_1) 6 (2 \ell -1) }
    \bigg) A_2^{-6}\\
& \phantom{\Bigg(} + \bigg(
        \sum_{1\leq a < b \leq \lrm{u_2}} v^{\sgn(u_2) 4\lrp{a+b -1}}
        + \lrb{-\sgn(u_2)}_{+} \sum_{\ell=1}^{\lrm{u_2}} v^{\sgn(u_2) 4 (2 \ell -1) }
    \bigg) A_1^4 \\
& \phantom{\Bigg(} + \bigg(
    \sgn(u_1) \lrp{v^{-4} +1 + v^4 } 
    \sum_{\ell=1}^{\lrm{u_1}} v^{\sgn(u_1) 3 (2 \ell -1) }\\
& \phantom{\Bigg( + \bigg(}   - \sgn(u_1)\sgn(u_2) v^{-6} \sum_{\ell_1=1}^{\lrm{u_1}} \sum_{\ell_2=1}^{\lrm{u_2}} v^{\sgn(u_1) 3 (2 \ell_1 -1) + \sgn(u_2) 2 (2 \ell_2 -1) } \bigg) A^{2 f_1 - 3 f_2}  + \cdots   
\Bigg)\\
&\Bigg( 1
    + \sgn(u_1) \sum_{\ell=1}^{\lrm{u_1}} v^{\sgn(u_1) 3 (2 \ell -1) } A_2^{-3} \\ 
& \phantom{\Bigg(} + \bigg(
        \sum_{1\leq a < b \leq \lrm{u_1}} v^{\sgn(u_1) 6 \lrp{a+b -1}}
        + \lrb{-\sgn(u_1)}_{+} \sum_{\ell=1}^{\lrm{u_1}} v^{\sgn(u_1) 6 (2 \ell -1) }
    \bigg) A_2^{-6} 
    + \cdots
\Bigg)A^u\\
=&
\Bigg( 1 + \sgn(u_2) \sum_{\ell=1}^{\lrm{u_2}} v^{\sgn(u_2) 2 (2 \ell -1) } A_1^2 \\
&\phantom{\Bigg(} +{\tc{red}{
\bigg(2 \sum_{1\leq a < b \leq \lrm{u_1}} v^{\sgn(u_1) 6 \lrp{ a+b-1}}
        +\sum_{\ell=1}^{\lrm{u_1}} v^{\sgn(u_1) 6 (2 \ell -1) }
 - \sum_{\substack{1\leq a \leq \lrm{u_1}\\ 1\leq b \leq \lrm{u_1} }} v^{\sgn(u_1) 6 \lrp{a+b -1}} \bigg)}}A_2^{-6}\\
&\phantom{\Bigg(} + \bigg(
        \sum_{1\leq a < b \leq \lrm{u_2}} v^{\sgn(u_2) 4\lrp{a+b -1}}
        + \lrb{-\sgn(u_2)}_{+} \sum_{\ell=1}^{\lrm{u_2}} v^{\sgn(u_2) 4 (2 \ell -1) }
    \bigg) A_1^4 \\
& \phantom{\Bigg(} + \bigg(
    \sgn(u_1) \lrp{v^{-4} +1 + v^4 } 
    \sum_{\ell=1}^{\lrm{u_1}} v^{\sgn(u_1) 3 (2 \ell -1) }
    + \sgn(u_2) \lrp{v^{-3} + v^3} \sum_{\ell=1}^{\lrm{u_2}} v^{\sgn(u_2) 2 (2 \ell -1) }  \\
& \phantom{\Bigg( +  \bigg(}   + \sgn(u_1)\sgn(u_2) \lrp{v^6 - v^{-6} }\sum_{\ell_1=1}^{\lrm{u_1}} \sum_{\ell_2=1}^{\lrm{u_2}} v^{\sgn(u_1) 3 (2 \ell_1 -1) + \sgn(u_2) 2 (2 \ell_2 -1) } \bigg) A^{2 f_1 - 3 f_2} \\
&\phantom{\Bigg(} + \cdots\Bigg)A^u.
}
Observe that the {\tc{red}{red}} portion cancels.
\eqn{g_2 \circ g_1 \circ g_2^{-1} (A^u) =&
\Bigg( 1 + \sgn(u_2) \sum_{\ell=1}^{\lrm{u_2}} v^{\sgn(u_2) 2 (2 \ell -1) } A_1^2 \\
&\phantom{\Bigg(}+ \bigg(
        \sum_{1\leq a < b \leq \lrm{u_2}} v^{\sgn(u_2) 4\lrp{a+b -1}}
        + \lrb{-\sgn(u_2)}_{+} \sum_{\ell=1}^{\lrm{u_2}} v^{\sgn(u_2) 4 (2 \ell -1) }
    \bigg) A_1^4 \\
& \phantom{\Bigg(} + \bigg(
    \sgn(u_1) \lrp{v^{-4} +1 + v^4 } 
    \sum_{\ell=1}^{\lrm{u_1}} v^{\sgn(u_1) 3 (2 \ell -1) }
    + \sgn(u_2) \lrp{v^{-3} + v^3} \sum_{\ell=1}^{\lrm{u_2}} v^{\sgn(u_2) 2 (2 \ell -1) }  \\
& \phantom{\Bigg( + \bigg(}   + \sgn(u_1)\sgn(u_2) \lrp{v^6 - v^{-6} }\sum_{\ell_1=1}^{\lrm{u_1}} \sum_{\ell_2=1}^{\lrm{u_2}} v^{\sgn(u_1) 3 (2 \ell_1 -1) + \sgn(u_2) 2 (2 \ell_2 -1) } \bigg) A^{2 f_1 - 3 f_2} \\
&\phantom{\Bigg(} + \cdots\Bigg)A^u.
}
Finally, we use \eqref{eq:g1inv} and compute:
\eqn{ \lrp{\fp_\gamma^1}^{-1}(A^u) =&
\Bigg( 1 + \sgn(u_2) \sum_{\ell=1}^{\lrm{u_2}} v^{\sgn(u_2) 2 (2 \ell -1) } A_1^2  \\
&\phantom{\Bigg(}+ \bigg(
        \sum_{1\leq a < b \leq \lrm{u_2}} v^{\sgn(u_2) 4 \lrp{a+b -1}}
        + \lrb{-\sgn(u_2)}_{+} \sum_{\ell=1}^{\lrm{u_2}} v^{\sgn(u_2) 4 (2 \ell -1) }
    \bigg) A_1^4 \\
& \phantom{\Bigg(} + \bigg(
    \sgn(u_1) \lrp{v^{-4} +1 + v^4 } 
    \sum_{\ell=1}^{\lrm{u_1}} v^{\sgn(u_1) 3 (2 \ell -1) }
    + \sgn(u_2) \lrp{v^{-3} + v^3} \sum_{\ell=1}^{\lrm{u_2}} v^{\sgn(u_2) 2 (2 \ell -1) }  \\
& \phantom{\Bigg( + \bigg(}   + \sgn(u_1)\sgn(u_2) \lrp{v^6 - v^{-6} }\sum_{\ell_1=1}^{\lrm{u_1}} \sum_{\ell_2=1}^{\lrm{u_2}} v^{\sgn(u_1) 3 (2 \ell_1 -1) + \sgn(u_2) 2 (2 \ell_2 -1) } \bigg) A^{2 f_1 - 3 f_2} \\
&\phantom{\Bigg(} + \cdots\Bigg)\\
&\Bigg( 1
    - \sgn(u_2) \sum_{\ell=1}^{\lrm{u_2}} v^{\sgn(u_2) 2 (2 \ell -1) } A_1^2 \\ 
& \phantom{\Bigg(} + \bigg(
        \sum_{1\leq a < b \leq \lrm{u_2}} v^{\sgn(u_2) 4 \lrp{a+b -1}}
        + \lrb{\sgn(u_2)}_{+} \sum_{\ell=1}^{\lrm{u_2}} v^{\sgn(u_2) 4 (2 \ell -1) }
    \bigg) A_1^4 
    + \cdots
\Bigg)A^u\\
=& \Bigg( 1 +{\tc{red}{
\bigg(2 \sum_{1\leq a < b \leq \lrm{u_2}} v^{\sgn(u_2) 4 \lrp{ a+b-1}}
        +\sum_{\ell=1}^{\lrm{u_2}} v^{\sgn(u_2) 4 (2 \ell -1) }
 - \sum_{\substack{1\leq a \leq \lrm{u_2}\\ 1\leq b \leq \lrm{u_2} }} v^{\sgn(u_2) 4 \lrp{a+b -1}} \bigg)}}A_1^{4}\\
& \phantom{\Bigg(} + \bigg(
    \sgn(u_1) \lrp{v^{-4} +1 + v^4 } 
    \sum_{\ell=1}^{\lrm{u_1}} v^{\sgn(u_1) 3 (2 \ell -1) }
    + \sgn(u_2) \lrp{v^{-3} + v^3} \sum_{\ell=1}^{\lrm{u_2}} v^{\sgn(u_2) 2 (2 \ell -1) }  \\
& \phantom{\Bigg( + \bigg(}   + \sgn(u_1)\sgn(u_2) \lrp{v^6 - v^{-6} }\sum_{\ell_1=1}^{\lrm{u_1}} \sum_{\ell_2=1}^{\lrm{u_2}} v^{\sgn(u_1) 3 (2 \ell_1 -1) + \sgn(u_2) 2 (2 \ell_2 -1) } \bigg) A^{2 f_1 - 3 f_2} \\
&\phantom{\Bigg(} + \cdots\Bigg)A^u.}
Once again the \tc{red}{red} portion cancels and we arrive at 
\eqn{ \lrp{\fp_\gamma^1}^{-1}(A^u) =&\Bigg( 1 + \bigg(
    \sgn(u_1) \lrp{v^{-4} +1 + v^4 } 
    \sum_{\ell=1}^{\lrm{u_1}} v^{\sgn(u_1) 3 (2 \ell -1) }
    + \sgn(u_2) \lrp{v^{-3} + v^3} \sum_{\ell=1}^{\lrm{u_2}} v^{\sgn(u_2) 2 (2 \ell -1) }  \\
& \phantom{\Bigg( }   + \sgn(u_1)\sgn(u_2) \lrp{v^6 - v^{-6} }\sum_{\ell_1=1}^{\lrm{u_1}} \sum_{\ell_2=1}^{\lrm{u_2}} v^{\sgn(u_1) 3 (2 \ell_1 -1) + \sgn(u_2) 2 (2 \ell_2 -1) } \bigg) A^{2 f_1 - 3 f_2}  + \cdots\Bigg)A^u.}

\end{appendices}

\subsubsection*{Acknowledgements} 
The authors would like to thank Mark Gross, Dylan Rupel and Elmar Wagner for helpful discussions. 
They would also like to thank the anonymous referee for a very careful review with many helpful suggestions.
The article has improved significantly as a result.

\subsubsection*{Funding}
M. Cheung is supported by NSF grant DMS-1854512, and AMS Simons Travel Grants. J.B. Fr\'ias-Medina was supported by Programa de Becas Posdoctorales 2019, DGAPA, UNAM during the preparation of the first version of this article and by ``Programa de Estancias Posdoctorales por M\'exico Convocatorias 2021 y 2022 de CONACYT" during the revisions.
T. Magee was supported by EPSRC grant EP/P021913/1 during the preparation of the first version of this article and by EPSRC grant EP/V002546/1 during revisions. 

\subsubsection*{Data Availability}
Data sharing not applicable to this article as no datasets were generated or analysed during the current study.

\subsection*{Compliance with Ethical Standards}
\textbf{Conflict of interest.} The authors declare that they have no conflict of interest.

\vspace{4ex}

\noindent{\sc{Man-Wai Cheung\\
Department of Mathematics, Kavli IPMU (WPI), UTIAS, The University of Tokyo, Kashiwa, Japan, 277-8583}}\\
{\it{e-mail:}} \href{mailto:manwai.cheung@ipmu.jp}{manwai.cheung@ipmu.jp} \medskip

\noindent{\sc{Corresponding author: Juan Bosco Fr\'ias-Medina\\
Instituto de F\'isica y Matem\'aticas, Universidad Michoacana de San Nicol\'as de Hidalgo. Edificio C-3, Ciudad Universitaria. Avenida Francisco J. M\'ugica s/n, Colonia Felicitas del R\'io,
C.P. 58040,  Morelia, Mich., Mexico}}\\
{\it{e-mail:}} \href{mailto:juan.frias@umich.mx}{juan.frias@umich.mx}\medskip

\noindent{\sc{Timothy Magee\\
Department of Mathematics, King's College London, Strand, London WC2R 2LS, UK}}\\
{\it{e-mail:}} \href{mailto:timothy.magee@kcl.ac.uk}{timothy.magee@kcl.ac.uk} \medskip


\begin{thebibliography}{CGM{\etalchar{+}}17}
\bibitem[BFMN20]{BFMNC}
L.~Bossinger, B.~{Fr\'ias-Medina}, T.~Magee and A.~{N\'ajera Ch\'avez},
  \textsl{ Toric degenerations of cluster varieties and cluster duality},
\newblock Compos. Math. \textbf{ 156}(10), 2149-- 2206 (2020).

\bibitem[Bri17]{Bri17}
T.~Bridgeland, \textsl{ Scattering diagrams, {H}all algebras and stability
  conditions},
\newblock Algebr. Geom. \textbf{ 4}(5), 523--561 (2017).

\bibitem[BZ05]{BZquantum}
A.~Berenstein and A.~Zelevinsky, \textsl{ Quantum cluster algebras},
\newblock Advances in Mathematics \textbf{ 195}(2), 405--455 (2005).

\bibitem[CGM{\etalchar{+}}17]{greedy}
M.~W. Cheung, M.~Gross, G.~Muller, G.~Musiker, D.~Rupel, S.~Stella and
  H.~Williams, \textsl{ The greedy basis equals the theta basis: a rank two
  haiku},
\newblock Journal of Combinatorial Theory, Series A \textbf{ 145}, 150--171
  (2017).

\bibitem[CM20]{dtpaper}
M.-W. Cheung and T.~Mandel, \textsl{ Donaldson--Thomas invariants from tropical
  disks},
\newblock Selecta Mathematica \textbf{ 26}(4), 1--46 (2020).

\bibitem[CMN21]{CMNcpt}
M.-W. {Cheung}, T.~{Magee} and A.~{N\'ajera Ch\'avez}, \textsl{
  Compactifications of cluster varieties and convexity},
\newblock Int. Math. Res. Not.  (2021),
\newblock
  \href{https://doi.org/10.1093/imrn/rnab030}{https://doi.org/10.1093/imrn/rnab030}.

\bibitem[{Dav}18]{Davison}
B.~{Davison}, \textsl{ Positivity for quantum cluster algebras},
\newblock Ann. Math. \textbf{ 187}(1), 157--219 (2018).

\bibitem[DM21]{davison2019strong}
B.~Davison and T.~Mandel, \textsl{ Strong positivity for quantum theta bases of
  quantum cluster algebras},
\newblock Invent. math. \textbf{ 226}(3), 725--843 (2021).

\bibitem[FG09]{FG_cluster_ensembles}
V.~V. Fock and A.~B. Goncharov, \textsl{ Cluster ensembles, quantization and
  the dilogarithm},
\newblock Ann. Sci. \'Ec. Norm. Sup\'er. (4) \textbf{ 42}(6), 865--930 (2009).

\bibitem[FG16]{FG_X}
V.~V. Fock and A.~B. Goncharov, \textsl{ Cluster {P}oisson varieties at
  infinity},
\newblock Selecta Math. (N.S.) \textbf{ 22}(4), 2569--2589 (2016).

\bibitem[FZ02]{FZ_clustersI}
S.~Fomin and A.~Zelevinsky, \textsl{ Cluster algebras. {I}. {F}oundations},
\newblock J. Amer. Math. Soc. \textbf{ 15}(2), 497--529 (2002).

\bibitem[FZ07]{FZ_clustersIV}
S.~Fomin and A.~Zelevinsky, \textsl{ Cluster algebras. {IV}. {C}oefficients},
\newblock Compos. Math. \textbf{ 143}(1), 112--164 (2007).

\bibitem[GHK15]{GHK_birational}
M.~Gross, P.~Hacking and S.~Keel, \textsl{ Birational geometry of cluster
  algebras},
\newblock Algebr. Geom. \textbf{ 2}(2), 137--175 (2015).

\bibitem[GHKK18]{GHKK}
M.~Gross, P.~Hacking, S.~Keel and M.~Kontsevich, \textsl{ Canonical bases for
  cluster algebras},
\newblock J. Amer. Math. Soc. \textbf{ 31}(2), 497--608 (2018).

\bibitem[Gol58]{Goldie}
A.-W. Goldie, \textsl{ The structure of prime rings under ascending chain
  conditions},
\newblock Proc. London Math. Soc. \textbf{ s3-8}, 589--608 (1958).

\bibitem[GS11]{GS11}
M.~Gross and B.~Siebert, \textsl{ From real affine geometry to complex
  geometry},
\newblock Ann. of Math. (2) \textbf{ 174}(3), 1301--1428 (2011).

\bibitem[GW04]{GW}
K.~R. {Goodearl} and R.~B. {Warfield, Jr.},
\newblock \textsl{ An Introduction to Noncommutative Noetherian Rings},
\newblock London Mathematical Society Student Texts, Cambridge University
  Press, 2004.

\bibitem[KN11]{KN11}
R.~M. Kashaev and T.~Nakanishi, \textsl{ Classical and quantum dilogarithm
  identities},
\newblock SIGMA Symmetry Integrability Geom. Methods Appl. \textbf{ 7}, Paper
  102, 29 pp. (2011).

\bibitem[KS06]{KS06}
M.~Kontsevich and Y.~Soibelman,
\newblock Affine structures and non-{A}rchimedean analytic spaces,
\newblock in \textsl{ The unity of mathematics}, volume 244 of \textsl{ Progr.
  Math.}, pages 321--385, Birkh\"{a}user Boston, Boston, MA, 2006.

\bibitem[KS14]{KS14}
M.~Kontsevich and Y.~Soibelman,
\newblock Wall-crossing structures in {D}onaldson-{T}homas invariants,
  integrable systems and mirror symmetry,
\newblock in \textsl{ Homological mirror symmetry and tropical geometry},
  volume~15 of \textsl{ Lect. Notes Unione Mat. Ital.}, pages 197--308,
  Springer, Cham, 2014.

\bibitem[LLRZ14]{Lee9712}
K.~Lee, L.~Li, D.~Rupel and A.~Zelevinsky, \textsl{ Greedy bases in rank 2
  quantum cluster algebras},
\newblock Proceedings of the National Academy of Sciences \textbf{ 111}(27),
  9712--9716 (2014).

\bibitem[LLRZ16]{lee2016existence}
K.~Lee, L.~Li, D.~Rupel and A.~Zelevinsky, \textsl{ The existence of greedy
  bases in rank 2 quantum cluster algebras},
\newblock Advances in Mathematics \textbf{ 300}, 360--389 (2016).

\bibitem[{Lor}18]{Lorenz}
M.~{Lorenz},
\newblock \textsl{ A tour of representation theory}, volume 193 of \textsl{
  Graduate studies in mathematics},
\newblock American Mathematical Society, 2018.

\bibitem[Man21]{mandel2015refined}
T.~Mandel, \textsl{ Scattering diagrams, theta functions, and refined tropical
  curve counts},
\newblock Journal of the London Mathematical Society \textbf{ 104}(5),
  2299--2334 (2021).

\bibitem[MR01]{McRob}
J.~C. {McConnell} and J.~C. {Robson},
\newblock \textsl{ Noncommutative {N}oetherian Rings}, volume~30 of \textsl{
  Graduate Studies in Mathematics},
\newblock American Mathematical Society, 2001.

\bibitem[Nak21]{Nak21}
T.~Nakanishi, \textsl{ Synchronicity phenomenon in cluster patterns},
\newblock Journal of the London Mathematical Society \textbf{ 103}(3),
  1120--1152 (2021).

\bibitem[Nak22]{Nak22}
T.~Nakanishi, \textsl{ Pentagon relation in quantum cluster scattering
  diagrams},
\newblock (2022), {arXiv:2202.01588 [math.CO]}.

\bibitem[NZ12]{NZ}
T.~Nakanishi and A.~Zelevinsky,
\newblock On tropical dualities in cluster algebras,
\newblock in \textsl{ Algebraic groups and quantum groups}, volume 565 of
  \textsl{ Contemp. Math.}, pages 217--226, Amer. Math. Soc., Providence, RI,
  2012.

\end{thebibliography}
\end{document}